\documentclass{amsart}
\usepackage[english]{babel}

\usepackage{amssymb, amsmath,mathtools,mathabx}
\usepackage{mathrsfs}
\usepackage{esint}
\usepackage{amscd}
\usepackage{verbatim}
\usepackage{stmaryrd}
\usepackage[dvipsnames]{xcolor}
\usepackage{pgf,tikz}
\usetikzlibrary{arrows}

\usepackage{enumerate}

\usepackage[colorlinks,linkcolor={blue},citecolor={blue},urlcolor={purple},]{hyperref}

\usepackage[colorinlistoftodos,prependcaption,textsize=tiny]{todonotes}

\usepackage{forest}

\usepackage{tabularx}
\newcolumntype{L}{>{\arraybackslash}X}
\usepackage{multirow}

\theoremstyle{plain}

\newtheorem{theorem}{Theorem}[section]
\theoremstyle{remark}
\newtheorem{remark}[theorem]{Remark}
\newtheorem{roadmap}[theorem]{\em \textbf{Roadmap}}
\newtheorem{example}[theorem]{Example}
\theoremstyle{plain}
\newtheorem{corollary}[theorem]{Corollary}
\newtheorem{lemma}[theorem]{Lemma}
\newtheorem{proposition}[theorem]{Proposition}
\newtheorem{definition}[theorem]{Definition}

\newtheorem{assumption}[theorem]{Assumption}

\numberwithin{equation}{section}

\def\N{{\mathbb N}}
\def\Z{{\mathbb Z}}

\def\R{{\mathbb R}}
\def\C{{\mathbb C}}

\newcommand{\one}{{{\bf 1}}}

\newcommand{\E}{{\mathbf E}}
\renewcommand{\P}{{\mathbf P}}
\newcommand{\F}{{\mathscr F}}
\newcommand{\Progress}{\mathscr{P}}

\newcommand{\g}{\gamma}
\newcommand{\s}{\delta}
\newcommand{\om}{\omega}
\renewcommand{\O}{\Omega}
\renewcommand{\a}{\kappa}

\newcommand{\crit}{\mathrm{c}}

\newcommand{\Borel}{\mathscr{B}}

\newcommand{\Tor}{\mathbb{T}}
\newcommand{\Dom}{\mathcal{O}}
\newcommand{\T}{\mathbb{T}}

\newcommand{\Hs}{\mathbb{H}}
\newcommand{\Bs}{\mathbb{B}}

\newcommand{\p}{\mathbb{P}}
\newcommand{\D}{\mathscr{D}}

\newcommand{\wt}{\widetilde}
\newcommand{\wh}{\widehat}

\newcommand{\embed}{\hookrightarrow}
\newcommand{\ps}{\partial^2\,}
\newcommand{\ph}{\nabla_{x,y}\,}

\newcommand{\ft}{\wt{f}_h}
\newcommand{\fts}{\wt{f}_{h,\sigma}}
\newcommand{\ftss}{\mathcal{L}_{h,\sigma}}
\newcommand{\ftg}{\wt{f}_{h,g}}
\newcommand{\fover}{\overline{\mathcal{L}}_{h,\sigma}}
\newcommand{\fwt}{\wt{\mathcal{L}}_{h,\sigma}}

\newcommand{\dd}{\mathrm{d}}
\newcommand{\loc}{\mathrm{loc}}
\newcommand{\Rsec}{\mathcal{R}}
\newcommand{\NN}{\mathsf{N}}
\newcommand{\q}{\vec{q}}
\newcommand{\qq}{\mathbb{Q}}
\newcommand{\x}{\mathbf{x}}
\newcommand{\kb}{\mathbf{k}}
\newcommand{\Br}{\mathcal{B}}
\newcommand{\diff}{\mathsf{d}}
\newcommand{\tr}{\mathsf{s}}
\newcommand{\set}{\mathcal{S}}

\newcommand{\nt}{\chi}
\newcommand{\X}{\mathcal{X}}
\newcommand{\Y}{\mathcal{Y}}
\newcommand{\nn}{N_{v}}
\newcommand{\JJ}{\mathcal{I}}
\newcommand{\norm}{\mathcal{N}}
\newcommand{\pz}{p_0}

\newcommand{\A}{\widehat{A}}
\newcommand{\B}{\widehat{B}}
\newcommand{\FF}{\widehat{F}}
\newcommand{\GG}{\widehat{G}}
\newcommand{\uu}{U}
\newcommand{\XX}{\widehat{X}}
\newcommand{\zero}{0}

\newcommand{\ellip}{\nu}

\newcommand{\Xap}{X^{{\rm Tr}}_{\a,p}}

\allowdisplaybreaks

\usepackage{scalerel,stackengine}
\stackMath
\newcommand\reallywidehat[1]{%
\savestack{\tmpbox}{\stretchto{%
  \scaleto{%
    \scalerel*[\widthof{\ensuremath{#1}}]{\kern-.6pt\bigwedge\kern-.6pt}%
    {\rule[-\textheight/2]{1ex}{\textheight}}
  }{\textheight}%
}{0.5ex}}%
\stackon[1pt]{#1}{\tmpbox}%
}

\begin{document}

\date\today

\title[The primitive equations with rough transport noise]{The primitive equations with rough transport noise: 
Global well-posedness and regularity}


\thanks{The author has received funding from the VICI subsidy VI.C.212.027 of the Netherlands Organisation for Scientific Research (NWO), and from the European Research Council (ERC) under the Eu\-ropean Union’s Horizon 2020 research and innovation programme (grant agreement No 948819) \includegraphics[height=0.4cm]{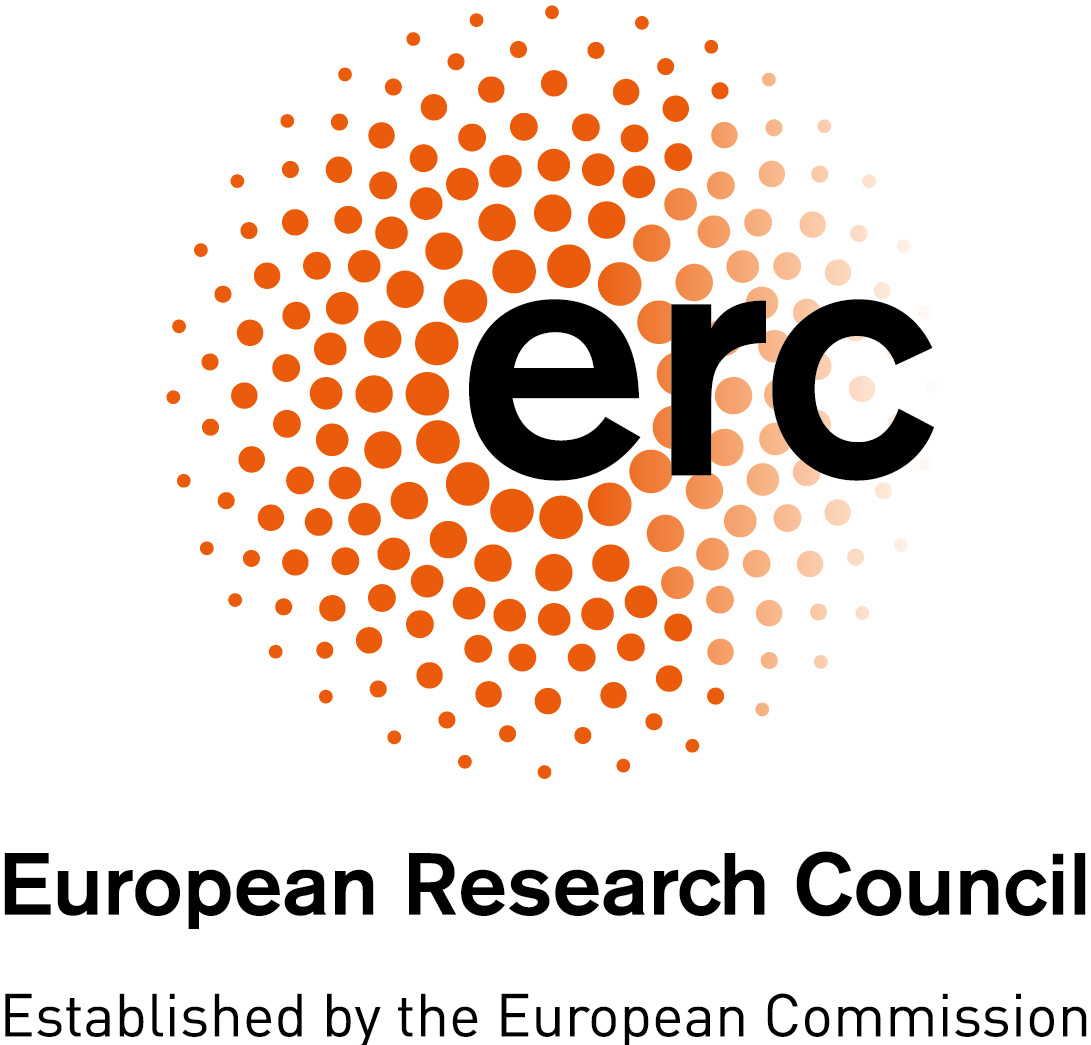}\,\includegraphics[height=0.4cm]{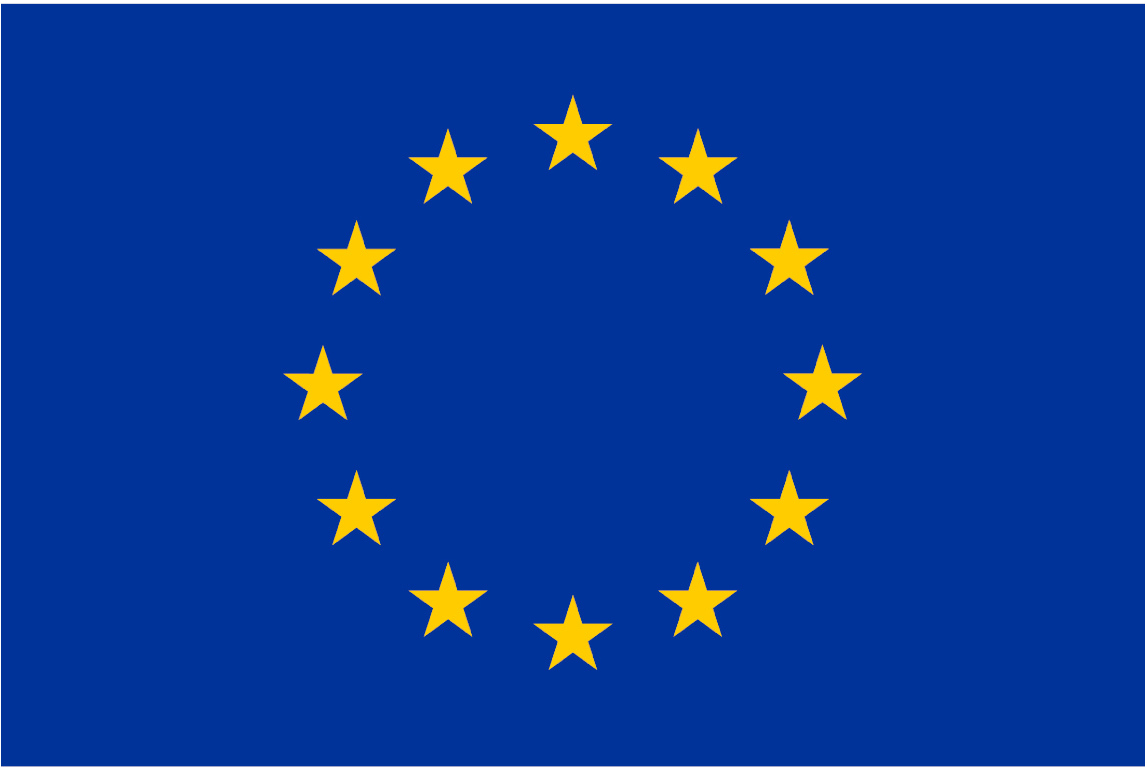}.
}

\author{Antonio Agresti}
\address{Delft Institute of Applied Mathematics\\
Delft University of Technology \\ P.O. Box 5031\\ 2600 GA Delft\\The
Netherlands} 
\email{antonio.agresti92@gmail.com}


\begin{abstract}
In this paper we establish global well-posedness and instantaneous regularization results for the primitive equations with transport noise of H\"older regularity $ \g>\frac{1}{2}$. It is known that if $\g<1$, then the noise is too rough for a strong formulation of primitive equations in an $L^2$-based setting. 
To handle rough noise, we crucially use $L^q$-techniques with $q> 2$. Interestingly, we identify a family of critical anisotropic Besov spaces for primitive equations, which is new even in the deterministic case. The behavior of these spaces reflects the intrinsic anisotropy of the primitive equations and plays an essential role in establishing global well-posedness and regularization. Our results cover Kraichnan's type noise with correlation greater than one, and as a by-product, a 2D noise reproducing the Kolmogorov spectrum of turbulence.
Moreover, the instantaneous regularization is new also in the widely studied case of $H^1$-data and $\g>1 $.
\end{abstract}

\maketitle

\addtocontents{toc}{\protect\setcounter{tocdepth}{1}}
\tableofcontents

\section{Introduction}
\label{s:intro}
The primitive equations (PEs in the following) are one of the fundamental models for geophysical flows used to describe oceanic and atmospheric dynamics. They are derived from the Navier-Stokes equations in domains where the vertical scale is much smaller than the horizontal scale by the small aspect ratio limit. We refer to \cite{KGHHKW20,LT19}  for rigorous justifications in the deterministic setting, and \cite[Section 2]{Primitive2} for the stochastic counterpart. Detailed information on the geophysical background for the various versions of the deterministic primitive equations can be found e.g.\ in \cite{K21_global,Ped,Vallis06}. 
In this paper we establish global well-posedness and instanstaneous regularization for PEs with  \emph{rough} transport noise on the periodic box $\T^3=[0,1]^3$:
\begin{equation}
\label{eq:primitive_intro}
\left\{
\begin{aligned}
		&\partial_t v   -\Delta v= \big[-\ph p-(u\cdot\nabla)v\big] \\ 
		&\qquad\qquad\quad 
		+\sum_{n\geq 1}\big[(\sigma_{n}\cdot\nabla) v-\ph \wt{p}_n\big]\, \dot{ \beta}_t^n, \quad &\text{ on }&\Tor^3,\\
		&\partial_t \theta   -\Delta \theta= -(u\cdot\nabla)\theta 		+\sum_{n\geq 1}(\nt_{n}\cdot\nabla) \theta\, \dot{ \beta}_t^n, \qquad\quad &\text{ on }&\Tor^3,\\
		&\partial_z p +\theta=0, \qquad  \partial_z \wt{p}_n=0,&\text{ on }&\Tor^3,\\
&\nabla \cdot u=0,&\text{ on }&\Tor^3,\\
&v(0,\cdot)=v_0,\qquad \theta(0,\cdot)=\theta_0,&\text{ on }&\Tor^3.
\end{aligned}
\right.
\end{equation}
Here the unknown processes are the velocity field $u:\R_+\times \O\times \T^3\to \R^3$, the temperature $\theta:\R_+\times \O\times \T^3\to \R$ and the pressures $p,\wt{p}_n:\R_+\times \O\times \T^3\to \R$. The velocity field $u$ is decomposed as $(v,w)$ where $v:\R_+\times \O\times \T^3\to \R^2$  and $w:\R_+\times \O\times \T^3\to \R$ denote the horizontal and the vertical component of $u$, respectively. Finally, $(\beta^n)_{n\geq 1}$ is a sequence of standard independent Brownian motions on a filtered probability space, and \eqref{eq:primitive_intro} is complemented with the following boundary conditions:
\begin{equation}
\label{eq:boundary_conditions_intro}
w=0 \  \text{ on }\ \T_{x,y}^2\times\{0,1\}.
\end{equation}
Our results also cover the case of stochastic PEs with additional lower-order nonlinearities and transport noise in the Stratonovich form, see Section \ref{ss:stratonovich} for details. 
In fluid mechanics, and in particular for geophysical flows, the Stratonovich formulation is relevant, and it is often seen as a more realistic model compared to the It\^o formulation, see e.g.\ \cite{BiFla20,Franzke14,HL84,W_thesis} and the references therein. 
A discussion on different boundary conditions can be found in Subsection \ref{ss:open_problems}. 
Before going further, let us mention that, 
in the physics literature, the PEs are often coupled with a balance for the salinity. However, it is omitted here as its mathematical treatment (but not its physical contribution) is analogous to the one for the temperature.

In the deterministic setting, the mathematical study of PEs began with the works of J.L.\ Lions, R.\ Temam, S.\ Wang \cite{LiTeWa1,LiTeWa2}, where they established the global existence of Leray's type solutions to the deterministic primitive equations. A breakthrough result on \emph{global well-posedness} of PEs with $H^1$-data was obtained by C.\ Cao and E.S.\ Titi in \cite{CT07}; see also \cite{Ko07,Kukavica_2007,GGHHK20_bounded,GGHK21_scaling,HH20_fluids_pressure,Salinity_MHK,HK16,Ju17} for further works.

In the stochastic setting, PEs with additive and/or multiplicative noise have been studied by several authors, see e.g.\ \cite{Primitive2,Primitive1,BS21,Debussche_2012,GHKVZ14} and the references therein.
In the current work, we are mainly concerned with stochastic perturbations of \emph{transport} type.  
In the context of stochastic fluid dynamics, transport noise has attracted a lot of interest in the last decades see e.g.\ \cite{AV21_NS,BiFla20,FG95,FL19,FlaPa21,HLN19,HLN21_annals,H15_SVP,MR01,MR04,MR05} and the references therein.
The PEs with transport noise \eqref{eq:primitive_intro} can be derived in two ways: using either two-scale systems \cite[Subsection 7.3]{DP22_two_scale}, or the above-mentioned hydrostatic approximation and using the Navier-Stokes equations with transport noise as a starting point \cite[Subsection 2.3]{Primitive2}. In \cite{MR01,MR04} the authors derived Euler and Navier-Stokes equations with transport noise by using the Newton law and assuming that a fluid particle located at $\x_0$ at time $t=0$ obeys the following stochastic dynamics:
\begin{equation}
\label{eq:flow_map}
\dot{\x}(t)=u(t,\x(t))+ \sigma(t,\x(t))\circ \dot{\Br}_t \ \text{ for }\  t>0, \ \ \ \text{ and }\ \  \
\x(0)=\x_0.
\end{equation}
In the above $\circ$ and $\Br$ denote the Stratonovich integration and a cylindrical Brownian motion, respectively. The generalized random field $\sigma(t,\x(t))\circ \dot{\Br}_t$ models the `turbulent' part of the velocity field, while $u(t,\x(t))$ models the `regular' part.

The interest in stochastic flows induced by \eqref{eq:flow_map} is related to the seminal works by R.\ Kraichnan \cite{K68,K94} on turbulent transportation (see e.g.\ \cite{GK96,GV00,MK99_simplified} and the references therein for subsequent results). In the former works, a turbulent velocity field on a $d$-dimensional set is modeled via a Gaussian vector field which is white in time and colored in space with correlation of the form: For 
 $\x,\x'\in \T^d$,
\begin{equation}
\label{eq:correlation_intro}
\mathscr{C}(\x,\x')
=\sum_{\kb\in \Z^d\setminus \{0\}} 
\Big(\mathrm{Id}-\frac{\kb \otimes \kb}{|\kb|^2}\Big)\frac{\mathscr{E}_{\alpha}(\kb)}{|\kb|^{d-1}}\, e^{i \kb\cdot (\x-\x')}\  \text{ where } \ 
\mathscr{E}_{\alpha}(\kb)\eqsim \frac{1}{|\kb|^{1+\alpha}}
\end{equation}
for some $\alpha\in (0,\infty)$. In applications to turbulent flows, one is interested in the range $\alpha\in (0,2)$ (cf.\ \cite[pp.\ 138]{MR05}). 
In particular, the \emph{Kolmogorov spectrum of turbulence} corresponds to $\alpha=\frac{4}{3}$ in \eqref{eq:correlation_intro}, cf.\ \cite[pp. 426-427 and 436]{MK99_simplified} and \cite[Remark 5.3]{GY21_stabilization}. 
Such Gaussian vector field can be realized via the random vector field $\wh{\sigma}_d(\x)\, \dot{\Br}_t= \sum_{n\geq 1}\wh{\sigma}_{d,n} (\x)\dot{\beta}^n_t$ where $(\beta^n)_{n\geq 1}$ is a sequence of standard independent Brownian motions. In Proposition \ref{prop:smoothness_Kraichnan}, we show that the parameter $\alpha$ uniquely determines the smoothness of $\wh{\sigma}_d$: 
\begin{equation}
\label{eq:regularity_Knoise_intro}
\wh{\sigma}_d\not\in H^{\alpha}(\T^d;\ell^2) \quad 
\text{ and }\quad \wh{\sigma}_d\in C^{\g}(\T^d;\ell^2) \ \ \text{ for all }\g<\tfrac{\alpha}{2}.
\end{equation} 
With this application in mind, we study the PEs \eqref{eq:primitive_intro} with transport noise of H\"older coefficients with  as little regularity as possible. However, in our main results, we restrict to the case $\alpha>1$. We refer to Subsection \ref{ss:open_problems} for comments on the case $\alpha\in (0,1]$. Let us note that our analysis covers the case $\alpha=\frac{4}{3}$ and hence the case of the Kolmogorov spectrum (see \cite{P91_Kolmogorov} for its relevance in oceanic flows and in the heat balance). 
Further comments and references are given below Theorem \ref{t:global_intro}.

The remaining part of this section is organized as follows. In Subsection \ref{ss:global_intro}, for illustrative purposes we state a special case of our main results. Compared to the existing literature, the key novelties are related to the \emph{roughness} of the transport noise and \emph{instantaneous regularization} of solutions. In particular, we are able to consider a 2D transport noise that reproduces the Kolmogorov spectrum of turbulence in the equation for the fluid motion, while we can allow the full 3D noise in the temperature balance. In Subsection \ref{ss:scaling_intro} we provide a throughout discussion on critical spaces for the (stochastic) PEs. Finally, in Subsection \ref{ss:overview}  and  \ref{ss:notation} we provide an overview of our results and introduce some notation, respectively.

\subsection{Illustration -- Global well-posedness in the Kolmogorov regime}
\label{ss:global_intro}
We begin by stating the following special case of our main results.

\begin{theorem}[Global well-posedness of PEs with rough transport noise -- Informal version]
\label{t:global_intro}
Suppose that 
\begin{align}
\label{eq:ass_intro_1}
\sigma&=(\sigma_n)_{n\geq 1}\in C^{\g_0}(\T^3;\ell^2(\N_{>0};\R^3)) \text{ for some } \g_0>\tfrac{1}{2},\\
\label{eq:ass_intro_2}
\sigma_n &=  \sigma_n(x,y,z)  \text{ is independent of }z \text{ for all }n\geq 1,\\
\label{eq:ass_intro_3}
\nt&=(\nt_n)_{n\geq 1}\in C^{\g_1}(\T^3;\ell^2(\N_{>0};\R^3))  \text{ for some } \g_1>0,
\end{align}
and that there exists $r\in (0,2)$ such that a.e.\ on $\R_+\times \O\times \T^d$ and for all $\xi\in \R^3$
\begin{equation}
\label{eq:parabolicity_intro}
\sum_{n\geq 1} |\sigma_n \cdot\xi|^2 \leq r|\xi|^2 \quad \text{ and }\quad 
\sum_{n\geq 1} |\nt_n \cdot\xi|^2 \leq r|\xi|^2  \ \ \ \emph{\text{(parabolicity)}}.
\end{equation}
Then for all $v_0\in H^1(\T^3;\R^2)$, $\theta_0\in H^1(\T^3)$ such that $\int_{\T_z} \nabla_{x,y}\cdot  v_0(\cdot,z)\,\dd z=0$, 
the stochastic PEs \eqref{eq:primitive_intro} admits a \emph{unique global solution} 
$$
v:\R_+\times \O\times \T^3\to \R^2, \qquad  \theta:\R_+\times \O\times \T^3\to \R,
$$ which instantaneously regularizes in space and time, i.e.\
\begin{align}
\label{eq:regularity_intro_1}
v
&\in C^{\mu_0,\nu_0}_{\loc} (\R_+\times \T^3;\R^2) & \text{ a.s.\ for all }&\mu_0\in [0,\tfrac{1}{2}),\ \nu_0\in (0,2\wedge(1+\g_0)),\\
\label{eq:regularity_intro_2}
\theta &\in C^{\mu_1,\nu_1}_{\loc} (\R_+\times \T^3; \R) & \text{ a.s.\ for all }&\mu_1\in [0,\tfrac{1}{2}),\ \nu_1\in (0,2\wedge(1+\g_1)).
\end{align}
Finally, the global solution $(v,\theta)$ depends continuously w.r.t.\ $(v_0,\theta_0)$.
\end{theorem}

The above is a special case of Theorems \ref{t:regularity2}, \ref{t:global} and \ref{t:continuous_dependence} (see also  Section \ref{s:non_isothermal} for the non-isothermal versions). 
The above global well-posedness result also holds for the PEs with transport noise in Stratonovich form, see Section \ref{ss:stratonovich}. In the latter case the parabolicity assumption \eqref{eq:parabolicity_intro} is redundant.

In Theorem \ref{t:global_intro}, the unknowns $w,p,\wt{p}$ are not specified as they are determined by $v$ and $\theta$, see Subsection \ref{ss:reformulation}. Let us anticipate that in Theorems \ref{t:regularity2}, \ref{t:global} and \ref{t:continuous_dependence} we also consider rougher initial data, lower-order nonlinearities in \eqref{eq:primitive_intro} and $z$-dependent $\sigma_{n,z}=\sigma_n \cdot e_{z}$ (see Assumption \ref{ass:global}).
The condition \eqref{eq:ass_intro_2} forces the transport noise to be two-dimensional. Physical motivations for the assumption \eqref{eq:ass_intro_2} are discussed in \cite[Remark 2.2 and 2.3]{Primitive2}, and \eqref{eq:ass_intro_2} can be justified via the \emph{stochastic hydrostatic approximation} introduced in \cite[Subsection 2.2]{Primitive2} (see \cite{AG01_approx,KGHHKW20,LT19} for the deterministic counterpart). We emphasize that \eqref{eq:ass_intro_2} is in accordance with measurements of turbulent oceanic streams which can be well-approximated by two-dimensional flows, see e.g.\
\cite{BE12,C01_twod,R73_twod,T02_twod,Y088_twod}.

Compared to the existing literature, the main novelties of Theorem \ref{t:global_intro} are:
\begin{enumerate}[{\rm(a)}]
\item \label{it:rough_transp_intro} The choice $\g_0<1$ and $\g_1<1$ in \eqref{eq:ass_intro_1} and \eqref{eq:ass_intro_3}, respectively.
\item \label{it:regularity_intro}  The instantaneous regularization result of \eqref{eq:regularity_intro_1}-\eqref{eq:regularity_intro_2}. 
\end{enumerate}

\eqref{it:rough_transp_intro}: 
In the author's opinion, this is the \emph{key} novelty of 
Theorem \ref{t:global_intro}. Indeed, the latter allows us to cover Kraichnan-type noise with correlation intensity $\alpha\in (1,2]$, see \eqref{eq:correlation_intro}-\eqref{eq:regularity_Knoise_intro}. In particular, it covers the case where $\sigma=(\wh{\sigma}_2,0)$ where $\wh{\sigma}_2$ is as below \eqref{eq:correlation_intro} with $\alpha=\frac{4}{3}$, hence reproducing the Kolmogorov spectrum of turbulence (cf.\ the comments below \eqref{eq:correlation_intro}). As recalled below Theorem \ref{t:global_intro}, the 2d nature of the choice $\sigma=(\wh{\sigma}_2,0)$ is in accordance with the physical derivation in \cite[Section 2]{Primitive2}. For the transport noise in the heat balance, due to \eqref{eq:ass_intro_3}, Kraichnan's noise with arbitrarily small correlation $\alpha>0$ is allowed. In particular, we may take $\nt=\wh{\sigma}_3$ with $\alpha=\frac{4}{3}$ (see e.g.\ \cite{P91_Kolmogorov} for physical motivations).

In the comments below we mainly focus on the velocity field $v$, as related remarks on the temperature $\theta$ are similar and simpler.

Kraichnan's type noises with correlation parameter $\alpha\leq 2$ were out of reach in previous works (see e.g.\ \cite{Primitive2,Primitive1,BS21,Debussche_2012}). To the best of our knowledge, in all previous results the transport noise coefficient $\sigma$ is assumed to be at least in the class $W^{1,r}(\T^3;\ell^2)$ for some $r>3$ (see e.g.\ \cite[Assumption 3.1(2)]{Primitive1}). 
However, as \eqref{eq:regularity_Knoise_intro} shows, this forces $\alpha>2$.
Next we discuss the reasons behind the latter assumption in previous works on stochastic PEs.
The smoothness condition on $\sigma$ is a consequence of the \emph{criticality} of PEs in the $H^1$-setting (see \cite[Subsection 5.1]{Primitive1} and Subsection \ref{ss:scaling_intro} below) and therefore in an $L^2$-scale one \emph{cannot} treat the PEs in a weaker setting than in the $H^1$-one. As a by-product one cannot handle transport noise with smoothness coefficients as in \eqref{eq:ass_intro_1} for $\g<1$ (therefore $\alpha\geq 2$). To see this one can argue as follows. As standard in SPDEs (see e.g.\ \cite{AV19_QSEE_1,DPZ,LR15}), one recasts the PEs as a stochastic evolution equation on a suitable Banach space. In the $H^1$-setting, to make sense of the transport noise term, one has to require that
\begin{equation}
\label{eq:transport_noise_mapping}
  v\in  H^{2}(\T^3) \quad \Longrightarrow \quad
((\sigma_n\cdot \nabla)v)_{n\geq 1 } \in H^{1}(\T^3;\ell^2).
\end{equation}
Reasoning as in \cite[Subsection 4.2]{Primitive1} this above is essentially equivalent to $\sigma\in W^{1,3}(\T^3;\ell^2)$ (or $\sigma\in H^{1}(\T^3;\ell^2)$ if $\sigma$ also satisfies \eqref{eq:ass_intro_2}). In light of \eqref{eq:regularity_Knoise_intro}, in a Hilbert setting there is no way to study the PEs with Kraichnan's type noise with correlation $\alpha \leq 2$. To overcome this difficulty we use the $L^q$-setting instead of a Hilbert space one. The basic advantage of the $L^q$-setting is that the choice of integrability $q\gg 2$ reduces the spatial smoothness required for the critical space, i.e.\ from $H^1$ to $B^{2/q}_{(q,2),p}=B^{2/q}_{(q,2),p}(\T^3;\R^2)$ for some $2<p,q<\infty$, see Subsection \ref{ss:notation}. Therefore reasoning as above one needs (roughly) to check that
\begin{equation}
\label{eq:transport_noise_mapping_2}
v\in  B^{1+2/q}_{(q,2),p}(\T^3) \quad \Longrightarrow  \quad
((\sigma_n\cdot \nabla)v)_{n\geq 1} \in B^{2/q}_{(q,2),p}(\T^3;\ell^2) .
\end{equation}  
The above holds provided $\sigma\in C^{\g}(\T^3;\ell^2)$ for $\g>2/q$, cf.\ Lemma \ref{l:pointwise}. In particular, if $q$ increases, then the smoothness requirement for $\sigma$ becomes weaker than the one needed for the implication \eqref{eq:transport_noise_mapping} to hold. However, as it will turn out in Subsection \ref{ss:proof_local}, for local well-posedness one needs $q<4$ and hence $2/q>1/2$, which leads to \eqref{eq:ass_intro_1}. 
The choice of the \emph{anisotropic Besov space} $B^{2/q}_{(q,2),p}$ in \eqref{eq:transport_noise_mapping_2} is intimately related to the scaling properties of PEs which are discussed in the next subsection.

\eqref{it:regularity_intro}: %
In the physical interesting case $\g_0<1$, \eqref{eq:regularity_intro_1} yields, for all $\nu_0<1+\g_0$, 
\begin{align}
\label{eq:v_increases_one_derivative}
v (t)\in C^{\nu_0}(\T^3;\R^2)\ \  \text{a.s.\ for all $t>0$\  \  while \ }   v(0)=v_0\in H^1(\T^3;\R^2).
\end{align}
In particular, $v$  gains almost \emph{one} derivative compared to the noise coefficients $(\sigma_n)_{n\geq 1}$. Due to the presence of the transport term $(\sigma_n\cdot \nabla) v$, the gain of regularity is sub-optimal (cf.\ \cite[Theorems 2.4 and 2.7]{AV21_NS} for a similar situation in case of stochastic Navier-Stokes equations). 
The regularity result \eqref{eq:regularity_intro_1} is a particular case of Theorem \ref{t:regularity_temp}. For high-order regularity, i.e.\ $\g_0\geq  1$, we refer to Theorem \ref{t:regularity2} and Remark \ref{r:high_order_temp}.
%
%
To the best of our knowledge, such instantaneous regularization results are \emph{new} in the stochastic framework even in the case $\g_0\gg 1$ and $H^1$-data (c.f.\ \cite{Primitive1,BS21,Debussche_2012}), and relies on the bootstrap techniques introduced in \cite{AV19_QSEE_2}.
The main difficulty behind instantaneous regularization results such as \eqref{eq:v_increases_one_derivative} is that, in the context of PEs, energy estimates are (typically) available in an $H^1$-setting. However, due to the above-mentioned criticality, one \emph{cannot} further bootstrap regularity by standard PDEs arguments (as for the stochastic Navier-Stokes equations \cite[Theorem 2.12]{AV21_NS}). Therefore one needs to argue differently.
In the deterministic setting, instantaneous regularization for PEs has been proven in  \cite{GGHHK20_analytic} by using the so-called `parameter trick' (see e.g.\ \cite[Subsection 5.2]{pruss2016moving}). However, as discussed at the end of \cite[Subsection 1.3]{AV19_QSEE_2}, the former is not available in a stochastic framework and thus the use of \cite{AV19_QSEE_2} cannot (in general) be avoided, cf.\ Subsection \ref{ss:scaling_intro} below.

As we will emphasize in Remarks \ref{r:regularity_necessary_anisotropic} and \ref{r:necessity_anisotropic}, the fact that $B^{2/q}_{(q,2),p}$ has the same scaling of $H^1$ (more precisely, $H^1\embed B^{2/q}_{(q,2),p} $ sharply) is of \emph{fundamental}  importance in proving the global well-posedness and instantaneous regularization results of \eqref{it:rough_transp_intro}-\eqref{it:regularity_intro}. More precisely, we do not know how to prove any of the statements in Theorem \ref{t:global_intro} in the isotropic setting directly, i.e.\ with $B^{2/q}_{(q,2),p}$ replaced by the classical Besov space $B^{2/q}_{q,p}$ as used in \cite{GGHHK20_analytic}. 
Of course the isotropic case can then be derived from the anisotropic one as $B^{2/q}_{q,p}\embed B^{2/q}_{(q,2),p}$ for $q\geq 2$.


\subsection{Scaling and criticality}
\label{ss:scaling_intro}
In the context of PEs \eqref{eq:primitive_intro}, the vertical direction plays a special role as the corresponding component $w$ of the velocity field $u$ does not obey an evolution PDE, but it is uniquely determined by the incompressibility condition and \eqref{eq:boundary_conditions_intro}, cf.\ \eqref{eq:def_w} below. In particular, there is no local (and a fortiori global) rescaling of the equations which leaves (locally) invariant the set of solutions. Therefore, it is not possible to identify critical spaces from a PDE point of view. 
In the deterministic literature, two approaches to the criticality of PEs have been used.
The first one is reminiscent of the fact that the PEs can be obtained from the (anisotropic) Navier-Stokes equations by means of the hydrostatic approximation (see e.g.\ \cite{AG01_approx,KGHHKW20,HH20_fluids_pressure,LT19}). Therefore one can define critical spaces for PEs as the one invariant under the natural rescaling of the Navier-Stokes equations, i.e.\ for $\lambda>0$,
\begin{equation}
\label{eq:rescaling_NS}
u\mapsto u_{\lambda}\qquad  \text{ where }\qquad  u_{\lambda}(t,\x)=\lambda^{1/2}u(\lambda t,\lambda^{1/2}\x)
\end{equation}
see e.g.\ \cite{Can04,LePi,PW18,Trie13} and \cite[Subsection 1.1]{AV21_NS} for the case of transport noise. 
In the PDE literature, spaces of initial data of the Navier-Stokes equations that are invariant under the induced rescaling $u_0\mapsto \lambda^{1/2}u_0(\lambda^{1/2}\cdot)$ are referred to as \emph{critical}. 
Therefore, 
 invariant spaces for initial data for Navier-Stokes equations are the ones with Sobolev index\footnote{The Sobolev index of $H^{s,q}(\T^d)$ and $B^{s}_{q,p}(\T^d)$ is given by $s-d/q$.} $-1$. 
Hence, one can define critical spaces for PEs as those having a Sobolev index $-1$. To the best of our knowledge, in the deterministic setting the only known space with such property and in which well-posedness results hold is $L^{(\infty,1)}\stackrel{{\rm def}}{=}L^{\infty}(\T^2_{x,y};L^{1}(\T_z;\R^2))$ (see \cite{GGHK21_scaling}, and also \cite{GGHHK20_bounded} for a related result). 
The second approach to define critical spaces is to use the theory of critical spaces for evolution equations developed by J.\ Pr\"uss, G.\ Simonett and M.\ Wilke \cite{PSW18_critical,addendum}, and its stochastic counterpart in  \cite{AV19_QSEE_1,AV19_QSEE_2} (see also \cite{Analysis3,W23_survey}). Such theory gives an `abstract' way to define critical spaces and in many situations, the latter turns out to enjoy the right local scaling of the underlined PDE. In the context of PEs, if one uses \emph{isotropic} $L^q(\T^d)$ spaces, then one obtains the Besov spaces $B^{2/q}_{q,p}=B^{2/q}_{q,p}(\T^3;\R^2)$ as critical spaces where $1<q<\infty$ and $p$ is large, see \cite{GGHHK20_analytic}. 
Interestingly, the previous spaces are consistent with the $H^1$-theory used in the seminal work \cite{CT07} on global well-posedness of PEs by considering $q=p=2$. However, the Sobolev index (and therefore the scaling) of the space $B^{2/q}_{q,p}$ is given by  $-\frac{1}{q}$, is $q$-dependent. Thus, the criticality of the latter Besov-type spaces might seem artificial as it depends on the choice of the ground space $L^q$. 

In this paper we introduce a \emph{new} class of critical spaces for stochastic PEs which, to the best of our knowledge, have not been considered in the deterministic literature. Our idea is to combine the above-mentioned two approaches. More precisely, we follow the stochastic evolution equation approach of \cite{AV19_QSEE_1,AV19_QSEE_2}, but we also take into account the natural anisotropy of the PEs in the vertical direction as exploited in \cite{GGHK21_scaling}. Eventually, this leads us to obtain local (and global) well-posedness in  (critical) \emph{anisotropic} Besov spaces of the form 
$$
B^{2/q}_{(q,2),p}=B^{2/q}_{(q,2),p}(\T^3;\R^2)\ \  \text{with $q\geq 2$ and $p$ large}.
$$ 
Here $q$ refers to the horizontal integrability, while the second parameter 2 refers to the vertical one (see Subsection \ref{ss:notation} for the precise definitions). The previous spaces have Sobolev index $-\frac{1}{2}$ that is independent of $(q,p)$, and therefore working with anisotropic spaces allows us to solve the issue concerning the Besov-type critical spaces used in \cite{GGHHK20_analytic}. Note that, for $q=p=2$, they coincide with the energy space $H^1=H^1(\T^3;\R^2)$ used in \cite{CT07} and $ H^{1}\embed B^{2/q}_{(q,2),p}$ sharply for all $q,p\geq 2$, due to the consistency of the Sobolev index. As commented before Subsection \ref{ss:scaling_intro} is of basic importance for the proof of global well-posedness and instantaneous regularization as in Theorem \ref{t:global_intro}.
%

%
To conclude, let us note that $B^{2/q}_{(q,2),p}$ does not enjoy the scaling property of the Navier-Stokes equations as it was for the space $L^{(\infty,1)}$. From the arguments of the current work, it is reasonable to expect local (and global) well-posedness results for the \emph{deterministic} PEs with initial data in $B^{2/q}_{(q,1),p}$ instead of $B^{2/q}_{(q,2),p}$. 
The space $B^{2/q}_{(q,1),p}$ has Sobolev index $-1$ and therefore respect the scaling invariance of the Navier-Stokes equations, i.e.\ \eqref{eq:rescaling_NS}. Moreover, the space $L^{(\infty,1)}$ used in \cite{GGHK21_scaling} can be seen as an `endpoint' as $q\to \infty$ of the latter scale of spaces. 
 However, one cannot use  $B^{2/q}_{(q,1),p}$ in the context of SPDEs with transport noise, as this does not allow to use important tools such as stochastic maximal $L^p_t$-regularity, which turned out to be very useful in handling transport noise in an $L^q_{\x}$-setting, cf.\ \cite{AV_torus,Kry,VP18}. Indeed, maximal $L^p_t$-regularity estimates are known to fail in $L^r_{\x}$-spaces with $r\in [1,2)$, see  \cite{NVW1,NVW3,MaximalLpregularity,NVW11}. In the study of SPDEs with transport noise, one is forced to use  Banach spaces with type 2 (cf.\ \cite[Chapter 7]{Analysis2} for the notion of type), and therefore $B^{2/q}_{(q,2),p}$ is the `closest' space to $B^{2/q}_{(q,1),p}$ with the latter property.

\subsection{Discussion on non-isothermal models and overview}
\label{ss:overview}
Below we list the main results of the present work. 
To keep the presentation as clear as possible the main body of this work deals with \emph{isothermal} models, i.e.\ in the case of constant temperature. 
The modifications needed to treat non-isothermal models such as \eqref{eq:primitive_intro} are only minor and are discussed in Subsection \ref{s:non_isothermal}. 
Let us recall that in the manuscript we consider a generalized version of \eqref{eq:primitive_intro}, i.e.\ \eqref{eq:primitive}   (see \eqref{eq:primitive_temperature} for the non-isothermal case) where we take into account also additional terms which allows us to also consider Stratonovich formulation of the transport noise, see Subsection \ref{ss:stratonovich}.
Let us remark that there exist more complicated non-isothermal models than the one considered here, see e.g.\ \cite{Primitive2,K21_global}. For the latter models, the extension from the isothermal to the non-isothermal case is not trivial. However,  we expect that by combining the methods of the current paper and the one of \cite{Primitive2} one can prove global well-posedness also for the non-isothermal model of \cite{Primitive2} with transport noise having the same roughness as considered here. 

The main results of this paper are as follows. Some open problems are discussed in Subsection \ref{ss:open_problems}.

\begin{itemize}
\item Local well-posedness and regularization in critical spaces -- Theorems \ref{t:local}, \ref{t:regularity1} and Proposition \ref{prop:continuity}.
\item Serrin's type blow-up criteria -- Theorem \ref{t:serrin}.
\item High-order regularity -- Theorem \ref{t:regularity2}.
\item Global well-posedness (under additional assumptions) -- Theorem \ref{t:global}.
\item Extension of the above results to non-isothermal models -- Section \ref{s:non_isothermal}.
\item Transport noise in Stratonovich form -- Subsection \ref{ss:stratonovich}.
\end{itemize}

\subsection{Comparison with the deterministic setting}
\label{ss:comparison_deterministic}
Compared to the deterministic setting, the approach taken here is drastically different. Indeed, to the best of the author's knowledge, all the results known for deterministic PEs with rough initial data $v_0$
rely on the global well-posedness in an $H^1$-setting  (see e.g.\ \cite{GGHHK20_bounded,GGHK21_scaling,Salinity_MHK,HK16,Ju17}). Indeed, in the deterministic setting, one combines instantaneous regularization results as in \eqref{eq:regularity_intro_1}-\eqref{eq:regularity_intro_2} with the global well-posedness in $H^1$ of \cite{CT07} to obtain global well-posedness in a weaker setting, cf.\ Step 1 in \cite[Theorem 7]{HH20_fluids_pressure}. 
However, as discussed below Theorem \ref{t:global_intro}, in the physically relevant case $\g_0\in (1,2)$, the $H^1$-setting is not available in case of rough transport noise.

Therefore,  all the issues which we deal with in the current work are of stochastic nature.
In particular, the family of critical spaces $B^{2/q}_{(q,2),p}$ used here, although not considered before, are not an enemy but rather the key tool that allows us to overcome the difficulties coming from the roughness of the noise.

\subsection{Notation}
\label{ss:notation}
Throughout the paper we let 
$\R_+=(0,\infty)$.
Moreover $\T^3$ denotes the three-dimensional torus, and we usually write $\x=(x,y,z)\in \T^3=\T_{x}\times \T_y\times \T_z$ and $\T_{x,y}^2 =\T_x\times \T_y$. 
Next, we introduce several function spaces which will be used later on. To this end, we employ the real and complex interpolation functors denoted by $(\cdot,\cdot)_{\vartheta,p}$ and $[\cdot,\cdot]_{\vartheta}$, respectively (see \cite{BeLo,InterpolationLunardi,Tri95} for details). Below $p,q,\zeta\in [1,\infty)$, $t\in (0,\infty]$, $\vartheta_1,\vartheta_2\in (0,1)$ and  $s\in\R$ are fixed.

\begin{itemize}
\item \emph{Anisotropic Lebesgue spaces.} $L^{(q,\zeta)}(\Tor^3)$ is the space of (equivalence classes of) measurable functions $f:\T^3\to \R$ such that 
$$
\|f\|_{L^{(q,\zeta)}(\T^3)}\stackrel{{\rm def}}{=}\Big[\int_{\T^2_{x,y}} \Big( \int_{\T_z} |f(x,y,z)|^{\zeta}\,\dd z \Big)^{q/\zeta}\dd x\dd y\Big]^{1/q}<\infty.
$$
\item  \emph{Anisotropic Bessel potential spaces}.  $H^{(q,\zeta)}(\Tor^3)$ is the space of all distribution $f\in \D'(\T^3)$ such that $(1-\Delta)^{s/2} f\in L^{(q,\zeta)}(\T^3)$ and
$$
\|f\|_{H^{s,(q,\zeta)}(\T^3)}\stackrel{{\rm def}}{=}\|(1-\Delta)^{s/2} f\|_{L^{(q,\zeta)}(\T^3)}.
$$

\item  \emph{Anisotropic Besov spaces}.  Let $s_0,s_1\in \R$ and $\vartheta\in (0,1)$ be such that $s_0(1-\vartheta)+s_1\vartheta=s$.  Then
$B^s_{(q,\zeta),p}(\Tor^3)$ is defined via real interpolation:
$$
B^s_{(q,\zeta),p}(\Tor^3)\stackrel{{\rm def}}{=}(H^{s_0,(q,\zeta)}(\T^3),H^{s_1,(q,\zeta)}(\T^3))_{\vartheta,p}.
$$
For the independence on the choice of $s_0,s_1$, see the comments below \eqref{eq:def_B_anisotropic}.
\item \emph{Anisotropic H\"{o}lder spaces}. $C^{\vartheta_1,\vartheta_2}((0,t)\times \T^3)$ is the set of all bounded maps $f:(0,t)\times \T^3\to \R$ such that
$$
\sup_{(s_0,\x_0),(s_1,\x_1)\in (0,t)\times \T^3} \frac{|f(s_0,\x_0)-f(s_1,\x_1)|}{ |s_0-s_1|^{\vartheta_1}+|\x_0-\x_1|^{\vartheta_2}}<\infty.
$$
\item \emph{Vector valued spaces}. For an integer $m\geq 1$ we let $\mathcal{A}(\T^3;\R^m)\stackrel{{\rm def}}{=}(\mathcal{A}(	\T^3))^m$ for $\mathcal{A}\in \{L^{(q,\zeta)},H^{s,(q,\zeta)}, 
B^s_{(q,\zeta),p}\}$, and similar for $C^{\vartheta_1,\vartheta_2}((0,t)\times \T^3;\R^m)$. 
Finally, $H^{s,(q,\zeta)}(\ell^2)$ denotes the set of all sequences $f=(f_n)_{n\geq 1}$ such that $f_n\in H^{s,(q,\zeta)}$ for all $n\geq 1$ and 
$$
\|f\|_{H^{s,(q,\zeta)}(\ell^2)}\stackrel{{\rm def}}{=} \|((1-\Delta)^{s/2}f_n)_{n\geq 1}\|_{L^{(q,\zeta)}(\ell^2)}<\infty
.$$
\end{itemize}
Below we do not consider the anisotropy in all directions, as they are not needed to deal with primitive equations. We refer to Appendix \ref{app:anisotropic} for the general situation. 

Next we introduce some weighted Banach valued function spaces. To this end, we fix $\vartheta\in (0,1)$, $p\in (1,\infty)$, $t\in (0,\infty]$, a Banach space $X$ and the weight 
$$
w_{\a}(\tau)={\tau}^{\a} \ \ \text{ where } \ \ \tau\geq 0,\ \a\in (-1,p-1).
$$
\begin{itemize}
\item \emph{Weighted Lebesgue spaces}. $L^p(0,t,w_{\a};X)$ is the set of all strongly measurable maps $f:(0,t)\to X$ such that 
$$
\|f\|_{L^p(0,t,w_{\a};X)}\stackrel{{\rm def}}{=} \Big(\int_{0}^t \|f(\tau)\|_X^pw_{\a}(\tau)\,\dd \tau \Big)^{1/p}<\infty.
$$
\item \emph{Weighted Sobolev spaces}. $W^{1,p}(0,t,w_{\a};X)$ is the set of all $f:(0,t)\to X$ such that $f,f'\in L^p(0,t,w_{\a};X)$, endowed with the norm 
$$
\|f\|_{W^{1,p}(0,t,w_{\a};X)}\stackrel{{\rm def}}{=}\|f\|_{L^p(0,t,w_{\a};X)}+\|f'\|_{L^p(0,t,w_{\a};X)}.
$$
\item 
\emph{Weighted Bessel potential spaces}.
$
H^{\vartheta,p}(0,t,w_{\a};X)$ is defined via complex interpolation as  
$$
H^{\vartheta,p}(0,t,w_{\a};X)\stackrel{{\rm def}}{=}[L^{p}(0,t,w_{\a};X),W^{1,p}(0,t,w_{\a};X)]_{\vartheta} .
$$ 
\end{itemize}
We often write $L^{(q,\zeta)}$ instead of $L^{(q,\zeta)}(\Tor^3)$ or $L^{(q,\zeta)}(\Tor^3;\R^m)$ with $m\in \N$ if no confusion seems likely. We use a similar notation for Bessel potential and Besov spaces. We write $L^q$, $H^{s,q}$, $B^s_{q,p}$ and similar for isotropic spaces, i.e.\ when $q=\zeta$.
We say that $f\in \mathcal{A}_{\loc}(I;X)$, whenever $\mathcal{A}$ is a function space and $I$ an interval, provided $f\in \mathcal{A}(J;X)$ for all $J\subseteq I$ with compact closure.

Finally we introduce the relevant probabilistic notation. Throughout the paper we fix a filtered probability space $(\O,\mathcal{A},(\F_t)_{t\geq 0},\P)$ and we denote by $\E$ the expectation w.r.t.\ $\P$.
$(\beta_n)_{n\geq 1}$ denotes a sequence of standard independent Brownian motions on such probability space. A measurable map $\tau:\O\to [0,\infty]$ is called a stopping time if $\{\tau\leq t\}\in \F_t$ for all $t\geq 0$. 

The Borel and progressive $\sigma$-algebra are denoted by $\Borel$ and $\Progress$, respectively. 

\subsubsection*{Acknowledgements}
The author thanks Amru Hussein and Emiel Lorist for useful suggestions and discussions. The author is indebted to Mark Veraar for his comments and suggestions on Appendix \ref{app:smr}. 

\section{Preliminaries}

\subsection{Smoothness of the Kraichnan model}
\label{ss:kraichnan}
In this subsection, we discuss some basic facts and properties of the Kraichnan model on the $d$-dimensional torus $\T^d$. 
Here, we partially follow the exposition in \cite[Section 5]{GY21_stabilization}. 

The Kraichnan model was introduced by R.\ Kraichnan \cite{K68} in the study of scalar advection by turbulent fluids (see also \cite{K68_enhancement,MK99_simplified}). The classical Kraichnan noise models \emph{isotropic} flows on $\R^d$ or $\mathbb{S}^d$, while the $\T^d$-case is described in \cite[Section 2.4]{CDG07_kraichnan}. 
The Kraichnan model $\wh{\sigma}_d$ on $\T^d$ with $d\geq 2$ can be introduced by specifying its correlation function, i.e.\  \eqref{eq:correlation_intro}.
The parameter $\alpha$ will be referred to as the `correlation parameter' of the Kraichnan model, as $\alpha$ rules its long-range spatial correlation. 
As recalled in the introduction, according to  \cite[pp.\ 426-427 and 436]{MK99_simplified} and \cite{GY21_stabilization}, the \emph{Kolmogorov spectrum of turbulence} corresponds to the choice $\alpha=\frac{4}{3}$ in \eqref{eq:correlation_intro}.
Indeed, as explained in \cite[pp.\ 426]{MK99_simplified}, taking into account that the inertial correlation time scales like $\eqsim |\kb|^{-2/3}$, the energy spectrum satisfies have scaling $$ 
\mathscr{E}_{4/3}(\kb) |\kb|^{2/3}\eqsim |\kb|^{-5/3},$$ 
that is the Kolmogorov spectrum in the inertial range.
Other choices of $\mathscr{E}_{\alpha}$ in the correlation function in \eqref{eq:correlation_intro} are possible. For instance by considering infrared and ultraviolet cutoffs at the integral and the Kolmogorov dissipation length scales, respectively (see \cite[pp.\ 426]{MK99_simplified}).

Set $\Z^d_0\stackrel{{\rm def}}{=}\Z^d\setminus\{0\}$ and let $\mathscr{E}_{\alpha}:\Z^d_0 \to [0,\infty)$ be such that $|\mathscr{E}_{\alpha}(\kb)|\eqsim |\kb|^{-1-\alpha}$ for some $\alpha>0$.
One can realize the Kraichnan model on $\T^d$ with correlation $\alpha\in (0,\infty)$ by choosing $\wh{\sigma}_d$ as $\wh{\sigma}=(\wh{\sigma}_{\kb,\ell})_{\kb\in \Z^d_0, \ell\in \{1,\dots,d-1\}}$, where for any pair $(\kb,-\kb)\in \Z^d_0\times \Z^d_0$ with $\kb$ lexicographically dominating $-\kb$ we have
\begin{align}
\label{eq:Kraich_1}
\wh{\sigma}_{\kb,\ell} (\x)
= a_{\kb,\ell}\,\frac{\big[\mathscr{E}_{\alpha}(\kb)\big]^{1/2}}{ |\kb|^{(d-1)/{2}}} \,\cos(2\pi \kb\cdot \x), \qquad \x\in \T^d,\\
\label{eq:Kraich_2}
\wh{\sigma}_{-\kb,\ell}(\x)=a_{\kb,\ell} \,\frac{\big[\mathscr{E}_{\alpha}(\kb)\big]^{1/2}}{|\kb|^{(d-1)/{2}}}\, \sin(2\pi  \kb\cdot \x), \qquad \x\in \T^d,
\end{align}
where $(a_{\kb,\ell})_{1\leq \ell\leq d-1}$ is an orthonormal basis of the hyperplane $\{\mathbf{y}\in \R^d\,:\, \kb\cdot \mathbf{y}=0\}$.

The following result relates the smoothness of $\wh{\sigma}$ to the correlation parameter $\alpha$.

\begin{proposition}
\label{prop:smoothness_Kraichnan}
Let $\wh{\sigma}=(\wh{\sigma}_\kb)_{\kb\in \Z^d_0,\ell\in \{1,\dots,d-1\}}$ be the Kraichnan model \eqref{eq:Kraich_1}-\eqref{eq:Kraich_2}  on $\T^d$ with $d\geq 2$ and correlation $\alpha\in (0,\infty)$, i.e.\ $\mathscr{E}_{\alpha}(\kb)\eqsim |\kb|^{-1-\alpha}$. Then 
\begin{enumerate}[{\rm(1)}]
\item\label{eq:smoothness_sigma_hat}
$\wh{\sigma}\not\in H^{\alpha/2,2}(\T^d;\ell^2)$.
\vspace{0.1cm}
\item\label{eq:smoothness_sigma_hat2}
$\wh{\sigma}\in C^{\g}(\T^d;\ell^2)$ for all 
$\g<\alpha/2.$
\end{enumerate}
\end{proposition}

Below we actually show the following stronger version of \eqref{eq:smoothness_sigma_hat2}: $\wh{\sigma}\in \ell^2(C^{\g}(\T^d))$ for all $\g<\alpha/2$. However, the latter is not used in the following.

\begin{proof}
%
\eqref{eq:smoothness_sigma_hat}: By \eqref{eq:Kraich_1}-\eqref{eq:Kraich_2}, for all  $\kb\in\Z^d_0$,
\begin{equation}
\label{eq:decay_fourier_coefficients_kraichnan_model}
\Big|\int_{\T^d} \wh{\sigma}_{\kb,\ell}(\x) e^{-2\pi i \, \kb\cdot \x}\,\dd \x\Big| \eqsim
\frac{1}{ |\kb|^{(d+\alpha)/{2}}}.
\end{equation}
Therefore
\begin{align*}
\|\wh{\sigma}\|_{H^{\alpha/2,2}(\ell^2)}^2
&\eqsim \max_{1\leq \ell\leq d-1}\sum_{\kb,\kb'\in \Z^d_0}(1+|\kb'|^2)^{\alpha/2}\Big| 
\int_{\T^d} \wh{\sigma}_{\kb,\ell}(\x) e^{-2\pi i \, \kb'\cdot \x}\,\dd \x\Big|^2\\
&\stackrel{(i)}{\geq}  \sum_{\kb\in \Z^d_0} \frac{(1+|\kb|^2)^{\alpha/2}}{|\kb|^{d+\alpha}}\eqsim \sum_{n\geq 1}n^{-1}=\infty.
\end{align*}
where $(i)$ follows by \eqref{eq:decay_fourier_coefficients_kraichnan_model} and considering only the terms $\kb=\kb'$ in the double sum.

\eqref{eq:smoothness_sigma_hat2}: Fix $\g<\frac{\alpha}{2}$. Without loss of generality, we assume that $\g\not\in \N$.
To begin, note that, for all multi-index $I=(I_i)_{i=1}^d\in \N_{0}^d$,
$$
\|(\partial^{I_1}_{x_1}\dots \partial_{x_d}^{I_d})\wt{\sigma}_{\kb,\ell}\|_{L^{\infty}}\lesssim |\kb|^{|I|-\frac{d+\alpha}{2}} 
\quad \text{ where } \quad 
|I|=\sum_{i=1}^d I_i. 
$$
Hence
$
\|\wt{\sigma}_{\kb,\ell}\|_{C^{\g}}\lesssim |\kb|^{\g-\frac{d+\alpha}{2}} $ by interpolation as $\g\not\in \N$.
Since $\g<\frac{\alpha}{2}$, we have
\begin{align*}
\max_{1\leq \ell\leq d-1}\sum_{\kb\in \Z^d_0} 
\|\wt{\sigma}_{\kb,\ell}\|_{C^{\g}}^2
\lesssim \sum_{\kb\in \Z^d_0} |\kb|^{2\g -d-\alpha} \eqsim \sum_{n\geq 1}n ^{-1-\alpha+2\g}<\infty .
\end{align*}
This completes the proof of Proposition \ref{prop:smoothness_Kraichnan}.
\end{proof}

\subsection{Hydrostatic Helmholtz projection and related function spaces}
\label{ss:projections}
In this subsection, we discuss the relevant function spaces of divergence-free vector fields which allow us to reformulate the PEs as stochastic evolution equation on a suitable Banach space. We begin by introducing Helmholtz-type projections. Below we employ the notation on anisotropic function spaces introduced in Subsection \ref{ss:notation}. 

Fix $q\in (1,\infty)$ and let $\p_{x,y}:L^q(\T^2_{x,y};\R^2)\to L^q(\T^2_{x,y};\R^2)$ be the Helmholtz projection, i.e.\ $\p_{x,y}f = f -\qq_{x,y} f$ where $\qq_{x,y} f \stackrel{{\rm def}}{=} \nabla_{x,y} \Psi_{f}$, and $\Psi_{f}\in H^{1,q}(\T^2_{x,y})$ is the unique solution to the following elliptic PDE:
$$
\Delta_{x,y} \Psi_f =\nabla_{x,y}\cdot f \  \text{ in }\D'(\T^2_{x,y}), \ \ \text{ and }\ \ \  \int_{\T^2_{x,y}}\Psi_f(x,y)\,\dd x\dd y=0.
$$
The \emph{hydrostatic} Helmholtz projection is given by
\begin{equation}
\label{eq:def_p_q}
\p f\stackrel{{\rm def}}{=}  f- \qq f , \quad \text{ where }\ \quad 
\qq f \stackrel{{\rm def}}{=} \qq_{x,y}\Big[\int_{\T_z} f(\cdot,z)\,\dd z\Big].
\end{equation}
It is easy to check that  $\p,\qq :L^{(q,\zeta)}(\T^3;\R^2)\to L^{(q,\zeta)}(\T^3;\R^2)$ are bounded linear operators for all $q,\zeta\in (1,\infty)$ and that the same holds with $L^{(q,\zeta)}(\T^3;\R^2)$ replaced by either $H^{s,(q,\zeta)}(\T^3;\R^2)$ or $B^{s}_{(q,\zeta),p}(\T^3;\R^2)$ where $s\in \R$ and $p\in (1,\infty)$. In the latter case, if $s<0$, then the operator $\int_{\T_z}\cdot\,\dd z$ in \eqref{eq:def_p_q} is understood in the distribution sense, i.e.\ 
$$
\Big\langle \int_{\T_z}f(\cdot,z)\,\dd z, \varphi\Big\rangle\stackrel{{\rm def}}{=}
\Big\langle f,\int_{\T_z}\varphi(\cdot,z)\,\dd z\Big\rangle \ \ \text{ for all }\ f\in \D'(\T^3),\, \varphi\in \D(\T^3). 
$$
Finally, we define spaces of divergence-free type vector fields that are suitable to study PEs: For 
$\mathcal{A}\in \{L^{(q,\zeta)},H^{s,(q,\zeta)},B^s_{(q,\zeta),p}\}$ we let 
\begin{align*}
\mathbb{A}(\Dom)
\stackrel{{\rm def}}{=} \p (\mathcal{A}(\T^3;\R^2)) 
=\Big\{f\in \mathcal{A}(\T^3;\R^2) \,:\, \nabla_{x,y}\cdot \big[\int_{\T_z} f(\cdot,z)\,\dd z\big]=0\text{ in }\D'(\T^3) \Big\}.
\end{align*}
For notational convenience, we often write $\mathbb{A}$ instead of $\mathbb{A}(\Dom)$.

\section{Statement of the main results -- The isothermal case}
\label{s:statement}
In this section, we state our main results concerning regularity, local and global well-posedness of the primitive equations with rough transport noise on the three-dimensional torus $\T^3=[0,1]^3$:
\begin{equation}
\label{eq:primitive}
\left\{
\begin{aligned}
		&\partial_t v  -\big[\nabla\cdot(a\cdot \nabla v) + (b\cdot\nabla) v\big]\\
		&\ \ = \big[- \ph p - (v\cdot \nabla_{x,y})v- w\partial_z v+\ps\wt{p} +f(\cdot,v,\nabla v)\big] \\ 
		&\ \ 		+\sum_{n\geq 1} 
		[-\ph \wt{p}+ (\sigma_{n}\cdot\nabla) v+ g_{n}(\cdot,v)]\, \dot{ \beta}_t^n, &\text{on }&\Tor^3,\\
		&\partial_z p +\theta=0, \qquad  \partial_z \wt{p}_n=0,&\text{on }&\Tor^3,\\
&\nabla_{x,y} \cdot v +\partial_z w=0,&\text{on }&\Tor^3,\\
&v(0,\cdot)=v_0,&\text{on }&\Tor^3.
\end{aligned}
\right.
\end{equation}
The above is complemented with the boundary conditions $w|_{\T^2_{x,y}\times \{0,1\}}=0$ and 
\begin{equation}
\label{eq:def_h_statement}
\ps \wt{p}\stackrel{{\rm def}}{=}\frac{h}{2}\sum_{n\geq 1} 
\big[\nabla \cdot (\ph \wt{p}_n\otimes \sigma_n) + b_{0,n} \ph \wt{p}_n\big] \ \ \text{ where }\ \ h\in [ -1,\infty).
\end{equation} 
The coefficients $(a,b,b_0)$ are specified in Assumption \ref{ass:primitive} below.

The terms $\nabla\cdot(a\cdot \nabla v)$, $(b\cdot\nabla) v$ and $\ps\wt{p}$ in \eqref{eq:primitive} are motivated by the corrective terms in the Stratonovich formulation of transport noise, see \eqref{eq:stratonovich_correction1}-\eqref{eq:stratonovich_correction2} in Subsection \ref{ss:stratonovich}. Indeed, with a particular choice of $(a,b,b_0,h)$ in \eqref{eq:def_abh_stra}, then \eqref{eq:def_h_statement} coincide with the PEs with transport noise in Stratonovich form \eqref{eq:primitive_stra}. The case $(a,b,h)=(\mathrm{Id}, 0,0)$ corresponds to the It\^o formulation of transport noise.

This section is organized as follows. In Subsection \ref{ss:reformulation} we reformulate the stochastic PEs as a stochastic evolution equation for the unknown $v$, while in Subsection \ref{ss:assumption_solution} we collect the main assumptions and definitions used to formulate our main results. In Subsection \ref{ss:local} and \ref{ss:serrin}
 we state our main results on local well-posedness, regularity, and blow-up criteria for \eqref{eq:primitive} in critical spaces (see Subsection \ref{ss:scaling_intro} for a discussion on criticality). In Subsection \ref{ss:global}, under additional assumptions on the noise coefficients, we state our main results on global well-posedness of \eqref{eq:primitive}. 
Finally, in Subsection \ref{ss:open_problems}, we discuss some open problems.
	
\subsection{Reformulation of the primitive equations}
\label{ss:reformulation}
In this subsection we reformulate the stochastic PEs as a stochastic evolution equation for the unknown horizontal velocity $v$, cf.\ \cite{HH20_fluids_pressure} or \cite[Subsection 2.3]{Primitive2}. 
 To begin, note that, integrating the divergence-free condition $\nabla_{x,y}\cdot v+\partial_z w=0$ and using $w|_{\T^2_{x,y}\times \{0,1\}}=0$, one obtains
\begin{align}
\label{eq:def_w}
w(t,\x)=[w(v)](t,\x)\stackrel{{\rm def}}{=}-\int_{0}^{z}\nabla_{x,y}\cdot v(t,x,y,z') \,\dd z'&  &\text{for }& \x=(x,y,z)\in \T^3,\\
\label{eq:incompressibility}
\int_{\T_z}\nabla_{x,y}\cdot v(t,x,y,z) \,\dd z=0\quad &  &\text{for } &(x,y)\in \T^2_{x,y}.
\end{align}
Note that \eqref{eq:incompressibility} is an incompressibility condition for the \emph{barotropic} mode $\int_{\T_z}v(\cdot,z)\,\dd z$.

Applying the hydrostatic Helmholtz projection $\p$ to the first equation of \eqref{eq:primitive},
\begin{align*}
\partial_t v  &=\p \big[ \nabla\cdot(a\cdot \nabla v)+(b\cdot\nabla)v- (v\cdot \nabla_{x,y})v- w(v)\partial_z v+\ps\wt{p} +f(\cdot,v,\nabla v)\big] \\ 
&\ +\sum_{n\geq 1} \p[(\sigma_{n}\cdot\nabla) v+ g_{n}(\cdot,v)]\, \dot{ \beta}_t^n,
\end{align*}
where we used that $\p v =v$ due to \eqref{eq:incompressibility}. Note that the pressures can be uniquely determined from $v$ via 
\begin{align}
\label{eq:pressure2}
\ph \wt{p}_n&=\qq\big[(\sigma_n\cdot\nabla) v+g_n(\cdot,v)\big],\\
\label{eq:pressure1}
\ph p &=\qq \big[\nabla\cdot(a\cdot \nabla v) +(b\cdot\nabla)v\\
\nonumber
&\qquad \quad - (v\cdot \nabla_{x,y})v- w(v)\partial_z v+\ps\wt{p} +f(\cdot,v,\nabla v)\big],
\end{align}
where $\qq$ is as in \eqref{eq:def_p_q}. Using \eqref{eq:pressure2}, we get $
\ps \wt{p}=\ft(\cdot,v)$ where
\begin{align}
\label{eq:def_f_tilde}
\ft(\cdot,v)&=\fts(\cdot)v+\ftg(\cdot,v),\\
\ftg(\cdot,v)
\label{eq:def_f_tilde2}
&\stackrel{{\rm def}}{=}
-\frac{h}{2}\sum_{n\geq 1} \Big(\nabla \cdot (\qq[g_n(\cdot,v)]\otimes \sigma_n)
+b_{0,n}  \qq[g_n(\cdot,v)]\Big),\\
\label{eq:def_f_tilde1}
\fts(\cdot)v
&\stackrel{{\rm def}}{=}
-\frac{h}{2}\sum_{n\geq 1}\Big( \nabla \cdot (\qq[(\sigma_n\cdot\nabla) v]\otimes \sigma_n)+b_{0,n} \qq[(\sigma_n\cdot\nabla) v]\Big).
\end{align}
The above decomposition of $\ft$ is motivated by the fact that $\fts $ is a \emph{linear} operator having the same order of the leading differential operators, while $\ftg$ is a nonlinear mapping and of lower-order type. In particular, looking at \eqref{eq:primitive} as a semilinear stochastic evolution equations (see \eqref{eq:def_ABFG}-\eqref{eq:SEE} below),  we will therefore consider $\fts$ (resp.\ $\ftg$) as a contribution of the linear (resp.\ non-linear) part.

Hence, at least formally, the PEs \eqref{eq:primitive} are equivalent to 
\begin{equation}
\label{eq:primitive2}
\left\{
\begin{aligned}
		&\partial_t v  		
		= \p\big[\nabla\cdot(a\cdot \nabla v)+(b\cdot\nabla) v- (v\cdot \nabla_{x,y})v- w(v)\partial_z v\big]\\
		&\  \  +\p\big[\ft(\cdot,v) +f(\cdot,v,\nabla v)\big]
		+\sum_{n\geq 1} \p\big[ (\sigma_{n}\cdot\nabla) v+ g_{n}(\cdot,v)\big]\, \dot{ \beta}_t^n, &\text{on }&\Tor^3,\\
&v(0,\cdot)=v_0,&\text{on }&\Tor^3.
\end{aligned}
\right.
\end{equation}
Note that in the system \eqref{eq:primitive2} the only unknown is the horizontal velocity $v$, while $w$, $p$ and $\wt{p}_n$ can be recovered from $v$ by using \eqref{eq:def_w}, \eqref{eq:pressure1} and \eqref{eq:pressure2}.
Finally, the divergence-free condition for the barotropic mode \eqref{eq:incompressibility} is preserved under the flow induced by \eqref{eq:primitive2} provided $\int_{\T_z}\nabla_{x,y}\cdot v_0(\cdot,z) \,\dd z=0 $ as it will be assumed below. 

\subsection{Main assumption and $(p,\a,\s,q)$-solutions}
\label{ss:assumption_solution}
We begin by listing our main assumptions. Further assumptions will be introduced where needed.

\begin{assumption} We say that Assumption \ref{ass:primitive}$(p,\s,q)$ holds if the following conditions are satisfied.
\label{ass:primitive}
\begin{enumerate}[{\rm(1)}]
\item $h\in [-1,\infty)$.
\item\label{it:integrability_parameter_primitive} $\delta\in [0,1)$ and either $[\, p\in (2,\infty)\text{ and }q\in [2,\infty)\, ]$ or $[\, p=q=2\, ]$.
\item\label{it:ass_primitive1} For all $n\geq 1$ the following mapping are $\Progress\otimes \Borel(\T^3)$-measurable
\begin{align*}
 a=(a_{\eta,\xi})_{\eta,\xi\in \{x,y,z\}}&:\R_+\times \O\times \T^3\to \R^{3\times 3},\\
b=(b_{\xi})_{\xi\in \{x,y,z\}},\ \sigma_n=(\sigma_{n,\xi})_{\xi\in \{x,y,z\}}&:\R_+\times \O\times \T^3\to \R^3,\\
 b_{0,n}&:\R_+\times \O\times \T^3\to \R.
\end{align*}
\item\label{it:ass_primitive4} For all $n\geq 1$, the mappings
$$
(a_{\eta,\xi})_{\eta,\xi\in \{x,y\}},  (\sigma_{n,\xi})_{\xi\in \{x,y\}}   \text{ are independent of }z\in \T_z.
$$
\item\label{it:ass_primitive3} There exist $M,\g>0$ such that $\g> 1-\s$, and for a.a.\ $(t,\om)\in  \R_+\times \O$,
$$
\|a(t,\om,\cdot)\|_{C^{\g}(\Tor^3;\R^{3\times 3})}
+\|(\sigma_n(t,\om,\cdot))_{n\geq 1}\|_{C^{\g}(\Tor^3;\ell^2)}\leq M   .
$$
\item\label{it:ass_primitive2} There exists $\nu\in (0,1)$ such that for all $\lambda=(\lambda_{\xi})_{\xi\in \{x,y,z\}}\in \R^3$
\begin{equation*}
\sum_{\eta,\xi\in \{x,y,z\}} \Big(a_{\eta,\xi}-\frac{1}{2}\sum_{n\geq 1} \sigma_{n,\eta} \sigma_{n,\xi}\Big) 
\lambda_{\eta}\lambda_{\xi}  \geq  \ellip|\lambda|^2 \ \ \text{ a.e.\ on }\R_+\times \O.
\end{equation*}
 \item\label{it:ass_primitiveb_b0} There exists $N>0$ such that,  for a.a.\ $(t,\om)\in \R_+\times \O$,
 \begin{equation*}
 \|b(t,\om,\cdot)\|_{L^{\infty}(\T^3;\R^3)}+
 \|(b_{0,n}(t,\om,\cdot))_{n\geq 1}\|_{L^{\infty}(\T^3;\ell^2)}\leq N .
 \end{equation*}
\item\label{it:ass_primitive7} For all $n\geq 1$, the maps $f:\R_+\times \O\times \T^3\times \R^2\times \R^{3\times 2}\to \R^2$, 
$g_n:\R_+\times \O\times \T^3\times \R^2\to \R^2$ are $\Progress\otimes \Borel(\T^3\times \R^2\times \R^{3\times 2})$- and $\Progress\otimes \Borel(\T^3\times \R^2)$-measurable, respectively; and moreover
\begin{align*}
f(\cdot,0,0)
&\in L^{\infty}(\R_+\times \O\times \T^3;\R^2),\\
(g_n(\cdot,0))_{n\geq 1}  
&\in L^{\infty}(\R_+\times \O;W^{1,\infty}(\T^3;\ell^2)).
\end{align*}
\item\label{it:ass_primitive8} For all $n\geq 1$, the map $\T^3\times \R^2 \ni (\x,\xi)\to g_n(\cdot,\x,\xi)$ is differentiable a.e.\ on $\R_+\times \O\times \T^3$. Finally, there exists $L>0$ for which the following estimates hold for all $\xi,\xi'\in \R^3$, $\eta,\eta'\in \R^{3\times 2}$ and a.e.\ on $\R_+\times \O\times \T^3$:
\begin{align*}
|f(\cdot,\xi,\eta)-f(\cdot,\xi',\eta')|&\leq L\big(|\xi-\xi'|+ |\eta-\eta'|\big),\\
\|(g_n(\cdot,\xi)-g_n(\cdot,\xi'))_{n\geq 1}\|_{\ell^2}&\leq L|\xi-\xi'|,\\
\|(\nabla_{\x} g_n(\cdot,\xi)-\nabla_{\x} g_n(\cdot,\xi'))_{n\geq 1}\|_{\ell^2}&\leq L|\xi-\xi'|,\\
\|(\nabla_{\xi} g_n(\cdot,\xi)-\nabla_{\xi} g_n(\cdot,\xi'))_{n\geq 1}\|_{\ell^2}&\leq L|\xi-\xi'|.
\end{align*}
\end{enumerate}
\end{assumption}

%

Physical motivation for Assumption \ref{ass:primitive}\eqref{it:ass_primitive4} are given \cite[Remarks 2.2 and 2.3]{Primitive2}.

The parameters $(p,\s,q)$ determine the solution space for \eqref{eq:primitive} which is a weighted space of the form $L^p_t(H^{2-\s,(q,2)}_{\x})$, cf.\ Definition \ref{def:solution} below. In particular, $\s$ rules the smoothness in space. 
The use of time weights is connected with the trace theory of anisotropic spaces (see e.g.\ \cite{MV14,ALV23} for details), while the use of the anisotropy $(q,2)$ for space variable is motivated by the scaling of PEs as discussed in Subsection \ref{ss:scaling_intro}.
Our main results are formulated for $\s<\frac{1}{2}$ and thus $\g>\frac{1}{2}$ due to Assumption \ref{ass:primitive}\eqref{it:ass_primitive3}. Some comments on the case $\g\leq \frac{1}{2}$ are given in Subsection \ref{ss:open_problems}.

\begin{example}[Kolmogorov's spectrum of turbulence via 2d transport noise]
\label{ex:Kolgomorov}
Let $\wh{\sigma}$ be the Kraichnan model with intensity parameter $\alpha\in (1,\infty)$ on $\T^2$, see Subsection \ref{ss:kraichnan}. Then taking $\sigma_{k,\ell}\stackrel{{\rm def}}{=}(\wh{\sigma}_{k,\ell},0)\in \R^3$ and an enumeration of such coefficients, by Proposition \ref{prop:smoothness_Kraichnan} one has that Assumption \ref{ass:primitive}\eqref{it:ass_primitive3} holds provided $\delta> 1-\frac{\alpha}{2}$. In particular, for the case $\alpha=\frac{4}{3}$, where $\wh{\sigma}$ reproduces the Kolmogorov spectrum of turbulence (see \cite[pp. 427 and 436]{MK99_simplified}), one is forced to take $\delta >\frac{1}{3}$. 
\end{example}

\begin{remark}[Weakening the assumptions on $f,g$]
Arguing as in \cite{Primitive1} (see Assumption 3.1(7) and Assumption 3.5(2) there), Assumption \ref{ass:primitive}\eqref{it:ass_primitive8} can be weakened a locally Lipschitz condition jointly with a sublinear condition. 
\end{remark}

Next, we introduce the notion of $(p,\a,\s,q)$-solutions to \eqref{eq:primitive} in light of the reformulation \eqref{eq:primitive2} in Subsection \ref{ss:reformulation}.
To this end, we recall that the sequence of standard independent Brownian motions $(\beta^n)_{n\geq 1}$ uniquely induces an $\ell^2$-cylindrical Brownian motion given by 
 (see e.g.\ \cite[Definition 2.11 and Example 2.12]{AV19_QSEE_1})
\begin{equation}
\label{eq:def_Br}
\Br_{\ell^2}(F)=\textstyle{\sum}_{n\geq 1} \int_{\R_+}F_n\,\dd \beta^n_t \ \ \text{ for  } \ F=(F_n)_{n\geq 1}\in L^2(\R_+;\ell^2).
\end{equation}

\begin{definition}[$(p,\a,\s,q)$-solution]
\label{def:solution}
Suppose that Assumption \ref{ass:primitive}$(p,\s,q)$ holds. Let $\a\in[0,\frac{p}{2}-1)$ if $p>2$ or $\a=0$ otherwise.
Let $\tau$ be a stopping time, and let $v:[0,\tau)\times \O\to \Hs^{2-\s,(q,2)}(\T^3)$ be a stochastic process. 
\begin{enumerate}[{\rm(1)}]
\item\label{it:def_sol1} We say that $(v,\tau)$ is a \emph{local $(p,\a,\s,q )$-solution} to \eqref{eq:primitive} if the there exists a sequence of stopping times $(\tau_j)_{j\geq 1 }$ for which the following hold for all $j\geq 1$:
\begin{itemize}
\item $\tau_j\leq \tau$ a.s.;
\item $\one_{[0,\tau_j]} v$ is progressively measurable, and a.s.\ 
\begin{align*}
v&\in L^p(0,\tau_j,w_{\a};H^{2-\s,(q,2)}),\\
-(v\cdot\nabla_{x,y})v - w(v)\partial_z v+\ftg(\cdot,v)+f(\cdot,v,\nabla v) &\in L^p(0,\tau_j,w_{\a};H^{-\s,(q,2)}),\\
(g_n(\cdot,v))_{n\geq 1} &\in L^p(0,\tau_j,w_{\a};H^{1-\s,(q,2)}(\ell^2));
\end{align*}
\item for all $t\in [0,\tau_j]$ and a.s.\ 
\begin{align*}
v(t)-v_0 
&=\int_0^t  \p\big[\nabla\cdot(a\cdot\nabla v)+ (b\cdot\nabla) v\\
&\qquad \qquad \qquad-(v\cdot\nabla_{x,y})v- w(v)\partial_z v+\ft(\cdot,v)+f(\cdot,v,\nabla v)\big]\,\dd s \\
& +\int_0^t \one_{[0,\tau_j]} \Big(\p[(\sigma_n\cdot\nabla) v+ g_n(\cdot,v) ]\Big)_{n\geq 1}\,\dd \Br_{\ell^2}(s),
\end{align*}
\end{itemize}
where $w(v)$ and $\ft(\cdot,v),\ftg(\cdot,v)$ are as in \eqref{eq:def_w} and \eqref{eq:def_f_tilde}-\eqref{eq:def_f_tilde2}, respectively.
\item\label{it:def_sol2} A local $(p,\a,\s,q)$-solution to \eqref{eq:primitive} is said to be a \emph{$(p,\a,\s,q)$-solution}  to \eqref{eq:primitive} if any other local $(p,\a,\s,q)$-solution $(v',\tau')$ to \eqref{eq:primitive} we have $\tau'\leq \tau$ a.s.\ and $v=v'$ a.e.\ on $[0,\tau')\times \O$.
\item A $(p,\a,\s,q)$-solution $(v,\tau)$  to \eqref{eq:primitive} is said to be \emph{global} if $\tau=\infty$ a.s. 
\end{enumerate}
\end{definition}

Note that $(p,\a,\s,q)$-solutions are unique by definition.

Employing Assumption \ref{ass:primitive} and Lemma \ref{l:pointwise}, one can check that, a.s.\ for all $j\geq 1$,
\begin{align*}
\nabla (a\cdot \nabla v) , \ \fts(\cdot)v&\in L^p(0,\tau_j,w_{\a};H^{-\s,(q,2)}),\\
((\sigma_n\cdot \nabla)v)_{n\geq 1}&\in L^p(0,\tau_j,w_{\a};H^{1-\s,(q,2)}(\ell^2)),
\end{align*}
where $(\tau_j)_{j\geq 1}$ is a sequence of stopping times as in Definition \ref{def:solution}\eqref{it:def_sol1}. Thus,
all the integrals appearing in the integral formulation of \eqref{eq:primitive} are well-defined as either Bochner or It\^o integrals. In latter case, one uses $L^p(w_{\a})\embed L^2$ as $\a<\frac{p}{2}-1$.

Under natural conditions, our results show that a different choice of the parameters $(p,\a,\s,q)$ leads to the same solution, see Corollary \ref{cor:compatibility} below.

\subsection{Local well-posedness and regularity in critical spaces}
\label{ss:local}
We begin with the following result on the existence of local unique solutions in critical spaces.

\begin{theorem}[Local existence and uniqueness in critical spaces]
\label{t:local}
Assume that  
\begin{equation}
\label{eq:assumptions_critical_setting}
\begin{aligned}
\text{ either }\ &\big[ q=p=2 \text{ and }\s=0\big],\\
 \text{ or }\  & \Big[
 \s\in \Big(0,\frac{1}{2}\Big),\   q\in \Big(\frac{2}{2-\s},\frac{2}{1-\s}\Big) \text{ and }\
 \frac{1}{p}+\frac{1}{q}+\frac{\delta}{2}\leq 1\Big].
\end{aligned}
\end{equation}
Let Assumption \ref{ass:primitive}$(p,\s,q)$ be satisfied and set $\a=\a_{\crit}\stackrel{{\rm def}}{=}p(1-\frac{1}{q}-\frac{\delta}{2})-1$. 
Then for each 
$$
v_0\in L^0_{\F_0}(\O;\Bs^{2/q}_{(q,2),p}(\T^3)),
$$
the problem \eqref{eq:primitive} has a (unique) $(p,\a_{\crit},\s,q)$-solution
$(v,\tau)$ such that  a.s.\ $\tau>0$  and
\begin{align}
\label{eq:reg_critical1}
v&\in  H^{\vartheta,p}_{\loc}([0,\tau),w_{\a_{\crit}};\Hs^{2-\s-2\vartheta,(q,2)}(\Tor^3))\  \text{ for all }\vartheta\in [0,\tfrac{1}{2}),\\
\label{eq:reg_critical2}
v&\in  C([0,\tau);\Bs^{2/q}_{(q,2),p}(\Tor^3));
\end{align}
where in the case $p=2$, \eqref{eq:reg_critical1} holds only for $\vartheta=0$.
\end{theorem}


Note that $\frac{2}{1-\s}>2$ in case $\s>0$. In particular, the second condition in \eqref{eq:assumptions_critical_setting} allows for some range of $q\geq 2$, cf.\ Assumption \ref{ass:primitive}\eqref{it:integrability_parameter_primitive}.
The space of initial data $\Bs^{2/q}_{(q,2),p}(\T^3)$ is \emph{critical} for the stochastic PEs as discussed in Subsection \ref{ss:scaling_intro}. For different values of $(q,p)$ we have, by Sobolev embeddings (see e.g.\ Corollary \ref{cor:sob_emb_besov}),
$$
\Bs^{2/q}_{(q,2),p}(\T^3)\embed \Bs^{2/q_1}_{(q_1,2),p_1}(\T^3) \ \ \text{ provided }\ \ 
q_1\in [ q,\infty) , \ p_1\in [p,\infty).
$$
By letting $\delta\uparrow \frac{1}{2}$ in the above, the second  and third conditions in \eqref{eq:assumptions_critical_setting} are satisfied if $q<4$ and $p$ large. Hence we may consider initial data with smoothness $>\frac{1}{2}$.
The possibility of choosing $\delta$ large is also crucial for allowing \emph{rough} transport noise. Indeed, by Assumption \ref{ass:primitive}\eqref{it:ass_primitive2}, the choice of $\delta$ is related to the regularity of $(\sigma_{n})_{n\geq 1}$: The higher the value of $\delta$, rougher the noise. In case $(\sigma)_{n\geq 1}$ are chosen so that they reproduce the 2d Kolmogorov's spectrum of turbulence then one has to choose $\delta\in (\frac{1}{3},\frac{1}{2})$, cf.\ Example \ref{ex:Kolgomorov}. It is interesting to note that, in the latter situation, the usual case $\delta=0$ analyzed in the previous works (see e.g.\ \cite{Primitive1,BS21,Debussche_2012}) is not allowed.

The following result ensures (local) continuity with respect to the initial data.
Hence Theorem \ref{t:local} and Proposition \ref{prop:continuity} yield that \eqref{eq:primitive} is \emph{local well-posed} in $\Bs^{2/q}_{(q,2),p}$. 

\begin{proposition}[Local continuity]
\label{prop:continuity}
Let the assumptions of Theorem \ref{t:local} be satisfied and let $(v,\tau)$ be the $(p,\a_{\crit},\s,q)$-solution to \eqref{eq:primitive}. Fix $\vartheta\in [0,\frac{1}{2})$. Then there exist $C_0,K_{\vartheta},T_0,\varepsilon_0>0$ and stopping times $\tau_0,\tau_1\in (0,\tau]$ a.s.\ for which the following assertion holds:

For each $v_0'\in L^0_{\F_0}(\O;\Bs^{2/q}_{(q,2),p})$ with $\E\|v_0-v_0'\|_{B^{2/q}_{(q,2),p}}^p\leq \varepsilon_0$,  the $(p,\a_{\crit},\s,q)$-solution $(v',\tau')$ to \eqref{eq:primitive} with initial data $v_0'$ has the property that there exists a stopping time $\tau_0'\in (0,\tau']$ a.s.\ such that, for all $t\in [0,T_0]$ and $\ell>0$, 
\begin{align}
\label{eq:continuity0}
\E \Big[\one_{\{\tau_0>t\}} \|v\|_{H^{\vartheta,p}(0,t,w_{\a_{\crit}};H^{2-\s-2\vartheta,(q,2)})}^p\Big]&\leq K_{\vartheta},\\
\label{eq:continuity1}
\E \Big[\one_{\{\tau_0\wedge \tau_0'>t\}} \|v-v'\|_{H^{\vartheta,p}(0,t,w_{\a_{\crit}};H^{2-\s-2\vartheta,(q,2)})}^p\Big]&\leq K_{\vartheta}\E\|v_0-v_0'\|_{B^{2/q}_{(q,2),p}}^p,\\
\label{eq:continuity2}
\P(\tau_0\wedge \tau_0'\leq t)\leq C_0\Big[\E\|v_0-v_0'\|_{B^{2/q}_{(q,2),p}}^p & + \P(\tau_1\leq t)\Big],
\end{align}
where in case $p=2$ and $\vartheta>0$, the space $H^{\vartheta,p}(0,t,w_{\a_{\crit}};H^{2-\s-2\vartheta,(q,2)})$ has to be replaced by $C([0,t];H^{1-\s})$.
\end{proposition}

Note that 
\eqref{eq:continuity1} ensures $v'\to v$ provided $v_0'\to v_0$ on $\{\tau_0\wedge \tau_0'>t\}$, while \eqref{eq:continuity2} shows that the latter set has large probability as $t\downarrow 0$ since $\tau_1>0$ a.s. The key point in \eqref{eq:continuity2} is that the RHS\eqref{eq:continuity2} is independent of $v_0'$, although $\tau_0'$ depends on the choice of $v_0'$.

For later convenience, let us remark that, by \cite[Theorem 1.2]{ALV23},  in the case $p>2$ the estimate \eqref{eq:continuity1} also holds with $H^{\vartheta,p}(0,t,w_{\a_{\crit}};H^{2-\s-2\vartheta,(q,2)})$ replaced by $C([0,t];B^{2/q}_{(q,2),p})$ and $C_{\a_{\crit}/p}((0,t];B^{2-\s-2/p}_{(q,2),p})$ (where $C_{\mu}((0,t];Y)$ is the set of continuous maps on $v:(0,t]\to Y$ such that $\sup_{r\in (0,t]} [r^{\mu}\|v(r)\|_{Y}]<\infty$).

Next we discuss regularization results for stochastic PEs \eqref{eq:primitive}.  

\begin{theorem}[Instantaneous regularization I]
\label{t:regularity1}
Let the assumptions of Theorem \ref{t:local} be satisfied and let $(v,\tau)$ be the $(p,\a_{\crit},\s,q)$-solution to \eqref{eq:primitive}. 
Let $\g$ be as in Assumption \ref{ass:primitive}\eqref{it:ass_primitive3}. Then $(v,\tau)$ instantaneously regularizes in time and space: 
\begin{align}
\label{eq:inst_reg1}
v&\in H^{\vartheta,r}_{\loc}(0,\tau; H^{1+\alpha-2\vartheta,r}(\T^3;\R^2))\text{ a.s.\ for all }\vartheta\in [0,\tfrac{1}{2})
,\, \alpha< \g',  \, r \in ( 2,\infty),\\
\label{eq:inst_reg2}
v& \in C^{\mu,\nu}_{\loc}((0,\tau)\times \T^3;\R^2) \text{ a.s.\ for all }\mu\in [0,\tfrac{1}{2}),\,  \nu\in (0,1+\g'),
\end{align}
where $\g'\stackrel{{\rm def}}{=}1\wedge \g$.
\end{theorem}

As suggested by the presence of the transport type term $(\sigma_n\cdot \nabla) v$ in \eqref{eq:primitive}, $\nabla v$ is expected to be at most as regular as $\sigma_n$. Hence, \eqref{eq:inst_reg1} and \eqref{eq:inst_reg2} show a sub-optimal gain of regularity in the  H\"older scale. It is not clear whether \eqref{eq:inst_reg1} (resp.\ \eqref{eq:inst_reg2}) also holds in the optimal case $\eta=\g'$ (resp.\ $\nu=1+\g'$) if $\g<1$. 
For future convenience, let us note that, \eqref{eq:inst_reg2} implies $C^1$-regularization of $v$, i.e.\
\begin{equation}
\label{eq:regularity_C1}
v\in C((0,\tau)\times \T^3;\R^2)\ \text{  a.s.} \ \ \ \text{ and } \ \ \   \nabla v\in C((0,\tau)\times \T^3;\R^{2\times 3})\ \text{ a.s. }
\end{equation}
Under additional assumptions on $(f,g,b,b_0)$, Theorem \ref{t:regularity1} also holds with $\g'$ is replaced by $\g$, see Theorem \ref{t:regularity2} below.

Before going further, let us point out the following interesting consequence of the above results which ensures that the solutions to \eqref{eq:primitive} provided in Theorem \ref{t:local} do not depend on the specific choice of the parameters.

\begin{corollary}[Compatibility]
\label{cor:compatibility}
Suppose that Theorem \ref{t:local} can be applied for two different sets of parameters $(p_1,\s_1,q_1)$ and  $(p_2,\s_2,q_2)$. Let $(v_1,\tau_1)$ and $(v_2,\tau_2) $ be the corresponding solutions. Then $\tau_1=\tau_2$ a.s.\ and $v_1=v_2$ a.e.\ on $[0,\tau_1)\times \O$.
\end{corollary}

The proof of the above result follows almost verbatim from the proof of \cite[Proposition 3.5]{AV22_localRD} by using Proposition \ref{prop:continuity} and Theorem \ref{t:regularity1}.  Hence, we omit the proof for brevity.

The proofs of Theorems \ref{t:local}, \ref{t:regularity1}, and Proposition \ref{prop:continuity} are given in Section \ref{s:proofs_local}.

\subsection{Serrin's type blow-up criteria and high-order regularity}
\label{ss:serrin}
Below we formulate blow-up criteria for solutions to \eqref{eq:primitive} provided by Theorem \ref{t:local}. Roughly speaking, blow-up criteria ensure that if $\P(\tau<\infty)>0$, then the norm of $v$ in a certain space explodes in finite time. As commented in \cite[Subsection 2.3]{AV22_localRD}, on the one hand, it is desirable to formulate blow-up criteria by using the largest function spaces possible and, on the other hand, such spaces cannot be too rough as at least the nonlinearities in the SPDE have to be well-defined. From this perspective, the critical setting is the optimal one as it the weakest for local well-posedness (cf.\ \cite[Subsection 2.2]{PSW18_critical}), and in particular for the well-definiteness of the nonlinearities.
 
The following $L^p(L^q)$-type blow-up criteria can be thought of as an analog of the Serrin-type blow-up criteria for stochastic PEs. 

\begin{theorem}[Serrin's blow-up criteria]
\label{t:serrin}
Let the assumptions of Theorem \ref{t:local} be satisfied and let $(v,\tau)$ be the $(p,\a_{\crit},\s,q)$-solution to \eqref{eq:primitive}, where $\a_{\crit}\stackrel{{\rm def}}{=} p(1-\frac{1}{q}-\frac{\delta}{2})-1$.
Assume that Assumption \ref{ass:primitive}$(p_{\zero},\s_{\zero},q_{\zero})$ also holds where 
\begin{equation}
\label{eq:assumptions_critical_setting_0}
\begin{aligned}
\text{ either }\ &\big[ q_0=p_0=2 \text{ and }\s_0=0\big],\\
 \text{ or }\  & \Big[
 \s_0\in \Big(0,\frac{1}{2}\Big),\   q_0\in \Big(\frac{2}{2-\s_0},\frac{2}{1-\s_0}\Big) \text{ and }\
 \frac{1}{p_0}+\frac{1}{q_0}+\frac{\delta_0}{2}\leq 1\Big], 
\end{aligned}
\end{equation}
Set $\mu_{\zero} \stackrel{{\rm def}}{=} \displaystyle{\frac{2}{q_{\zero}}+\frac{2}{p_{\zero}}}$. 
Then the following hold for all $0<s<T\leq \infty$:
\begin{enumerate}[{\rm(1)}]
\item\label{it:serrin1} $\displaystyle{\P\Big(s<\tau<T,\,\|v\|_{L^{p_{\zero}}(s,\tau;H^{\mu_{\zero},(q_{\zero},2)}(\T^3;\R^2))}<\infty\Big)=0}$.
\vspace{0.1cm}  
\item\label{it:serrin2} $\displaystyle{\P\Big(s<\tau<T,\,\sup_{t\in [s,\tau)}\|v(t)\|_{B^{\lambda_{\zero}}_{(q_{\zero},2),\infty}(\T^3;\R^2)}<\infty\Big)=0}$ for all $\displaystyle{\lambda_0\in \Big(\frac{2}{q_{\zero}}}, 2\wedge (1+ \g)\Big)$.
\end{enumerate}
\end{theorem}

Due to Theorem \ref{t:regularity1} and \eqref{eq:assumptions_critical_setting_0} the norms in \eqref{it:serrin1}-\eqref{it:serrin2} are well-defined as $s>0$. 

Note that the space-time Sobolev index of the space $L^{p_0}(H^{\mu_0,(q_0,2)})$ in \eqref{it:serrin1} is given by $-\frac{2}{p_0}+\mu_0-\frac{2}{q_0}-\frac{1}{2}=-\frac{1}{2}$ and coincide with the one used in for the initial data in Theorem \ref{t:local}. 
Arguing as below Theorem \ref{t:local}, by letting $\delta_0\downarrow\frac{1}{2}$, Theorem \ref{t:serrin} provides blow-up criteria in spaces with smoothness  $>\frac{1}{2}$. 

Note that as a special case of \eqref{it:serrin1} we have: For all $0<s<T\leq \infty$
\begin{equation}
\label{eq:serrin}
\P\Big(s<\tau<T,\,\|v\|_{L^{p_0}(s,\tau;H^{1,(q_0,2)}(\T^3;\R^2))}<\infty\Big)=0\ \ \text{ provided }\ \frac{2}{p_0}+\frac{2}{q_0}=1.
\end{equation}
Besides the analogy of \eqref{eq:serrin} with the well-known Serrin's criteria for the Navier-Stokes equations (see e.g.\ \cite[Theorem 11.2]{LePi}), another motivation behind the name of Theorem \ref{t:serrin} is that \eqref{it:serrin1} is a direct consequence of the abstract Serrin criterion of \cite[Theorem 4.11]{AV19_QSEE_2} which, in the case of stochastic Navier-Stokes equations, leads to an extension of the Serrin's criteria to the stochastic setting \cite[Theorem 2.9]{AV21_NS}.
Finally, let us note that Theorem \ref{t:serrin}\eqref{it:serrin2} follows from \eqref{it:serrin1} by choosing $p_0$ so large that $\mu_0=\frac{2}{q_0}+\frac{2}{p_0}<\lambda_0$ and therefore $B^{\lambda_0}_{(q_0,2),\infty}\embed H^{\mu_0,(q_0,2)}$ due to the elementary embeddings \eqref{eq:elementary_emb1}-\eqref{eq:elementary_emb2}.

The following high-order regularity result complements Theorem \ref{t:regularity1}.

\begin{theorem}[Instantaneous regularization II]
\label{t:regularity2}
Let the assumptions of Theorem \ref{t:local} be satisfied and let $(v,\tau)$ be the $(p,\a_{\crit},\s,q)$-solution to \eqref{eq:primitive}. 
Let $\g$ be as in Assumption \ref{ass:primitive}\eqref{it:ass_primitive3}. Suppose that:
\begin{enumerate}[{\rm(1)}]
\item\label{it:reg_f_highreg1} $(f,g)$ is independent of $\x\in \T^3$.
\item\label{it:reg_f_highreg2} a.e.\ on $\R_+\times \O$ and  $n\geq 1$, the mappings $(\xi,\eta)\mapsto f(\cdot,\xi,\eta)$ and $\xi\mapsto g_n(\cdot,\xi)$ are $C^{\lceil \g\rceil}$ and $C^{\lceil \g+1\rceil}$, respectively; moreover for all $N_1\geq 1$ there exists $C_{N_1}>0$ such that, for all $|\xi|,|\eta|\leq N_1$ and a.e.\ on $\R_+\times \O$,
\begin{equation*}
\sum_{j=1}^{\lceil \g\rceil} \big(|\nabla^j_{\xi} f(\cdot,\xi,\eta)|+ |\nabla^j_{\eta} f(\cdot,\xi,\eta)|\big)+ 
\sum_{j=1}^{\lceil \g+1 \rceil}\|(\nabla^j_{\xi} g_n(\cdot,\xi))_{n\geq 1}\|_{\ell^2} \leq C_{N_1}.
\end{equation*}
\item\label{it:reg_f_highreg3} There exists $N_2>0$ such that,  for a.a.\ $(t,\om)\in \R_+\times \O$,
 \begin{equation*}
 \|b(t,\om,\cdot)\|_{C^{\g-1}(\T^3;\R^3)}+
 \|(b_{0,n}(t,\om,\cdot))_{n\geq 1}\|_{C^{\g-1}(\T^3;\ell^2)}\leq N_2 .
 \end{equation*}
\end{enumerate}
Then $(v,\tau)$ instantaneously regularizes in time and space:
\begin{align}
\label{eq:inst_reg21}
v&\in H^{\vartheta,r}_{\loc}(0,\tau; H^{1+\beta-2\vartheta,r}(\T^3;\R^2)) \text{ a.s.\ for all }  \vartheta\in [0,\tfrac{1}{2}),\, \beta<\g,\,  r \in ( 2,\infty),\\
\label{eq:inst_reg22}
v& \in C^{\vartheta_1,\vartheta_2}_{\loc}((0,\tau)\times \T^3;\R^2) \text{ a.s.\ for all }\vartheta_1\in [0,\tfrac{1}{2}), \, \vartheta_2\in (0,1+\g).
\end{align}
\end{theorem}

In particular, Theorem \ref{t:regularity2} provides high-order regularity of $v$ at the expense of additional regularity assumptions on $(f,g,b,b_0)$. The $\x$-independence condition of \eqref{it:reg_f_highreg1} can be weakened to smoothness conditions. We omit this for brevity.

\subsection{Global well-posedness}
\label{ss:global}
In this subsection, we state the main result of the current work, namely the global well-posedness of stochastic PEs \eqref{eq:primitive}. To this end, we need the following additional smoothness assumption on $a$ and $\sigma$.

\begin{assumption}
\label{ass:global}
Let $a=(a_{\eta,\xi})_{\eta,\xi\in \{x,y,z\}}$ and $ \sigma_n=(\sigma_{n,\xi})_{\xi\in \{x,y,z\}}$ be as in Assumption \ref{ass:primitive}.
 There exists $K>0$ such that for a.a.\ $(t,\om)\in \R_+\times \O$ and for all $n\geq 1$, $\eta\in \{x,y\}$ we have
\begin{align*}
 \partial_{z}a_{z,\eta}(t,\om,\cdot),\, 
\partial_z a_{\eta,z}(t,\om,\cdot),\,
   \partial_z\sigma_{n,z}(t,\om,\cdot)&\in L^{\infty}(\T^3)  \text{ and}
\\
 \|\partial_{z}a_{\eta,z}(t,\om,\cdot)\|_{L^{\infty}(\T^3)}
 +
 \|\partial_{z}a_{z,\eta}(t,\om,\cdot)\|_{L^{\infty}(\T^3)}\qquad &\\
  +
 \big\|(\partial_{z}\sigma_{n,z}(t,\om,\cdot))_{n\geq 1}\big\|_{L^{\infty}(\T^3;\ell^2)}
 &\leq K.
\end{align*}
\end{assumption}

In the above the partial derivative $\partial_z $ is understood in the distributional sense.

Assumption \ref{ass:global} is automatically satisfied in case $a_{z,\eta},a_{z,\eta}$ and $\sigma_{n,z}$ are $z$-independent for 
$\eta\in \{x,y\}$.
As recalled in Subsection \ref{ss:global_intro}, the latter situation fits the physical derivation of the stochastic PEs via stochastic hydrostatic approximation \cite[Remarks 2.2 and 2.3]{Primitive2}. It is important that in Assumption \ref{ass:global} there are no additional regularity assumptions on the horizontal variables. In particular, Assumption \ref{ass:global} is satisfied in case $\sigma$ is as in Example \ref{ex:Kolgomorov}, and therefore it is a 2d field reproducing the Kolmogorov spectrum.
We expect that the integrability requirements in Assumption \ref{ass:primitive} can be weakened (e.g.\ $L^4_{x,y}(L^8_z)$ instead of $L^{\infty}$ seems sufficient). As we are not aware of physical motivations, we do not pursue this here.

The following ensure that the stochastic PEs \eqref{eq:primitive} are \emph{global well-posed} in $\Bs^{2/q}_{(q,2),p}$.

\begin{theorem}[Global existence of unique solutions]
\label{t:global}
Let Assumptions \ref{ass:primitive}$(p,\s,q)$ and \ref{ass:global} be satisfied. 
Assume that $(p,\s,q)$ satisfy 
\eqref{eq:assumptions_critical_setting} and set $\a=\a_{\crit}\stackrel{{\rm def}}{=} p(1-\frac{1}{q}-\frac{\s}{2})-1$. Let $(v,\tau)$ be the (unique) $(p,\s,\a_{\crit},q)$-solution provided by Theorem \ref{t:local}. 
Then $(v,\tau)$ is \emph{global}, i.e.\  $\tau=\infty$ a.s. In particular \eqref{eq:reg_critical1}-\eqref{eq:reg_critical2} hold with $\tau=\infty$. Moreover, the following hold:
\begin{itemize}
\item $\partial_z v \in L^2_{\loc}(0,\infty;H^1)$ a.s.
\item 
\eqref{eq:inst_reg1}-\eqref{eq:inst_reg2} hold with $\tau=\infty$.
\item \eqref{eq:inst_reg21}-\eqref{eq:inst_reg22} hold with $\tau=\infty$ if the assumptions \eqref{it:reg_f_highreg1}-\eqref{it:reg_f_highreg3} of Theorem \ref{t:regularity2} are satisfied.
\end{itemize}
\end{theorem} 

As usual, for a Banach space $Y$, we say that $\xi_n\to \xi$ in probability in $Y$ provided $\lim_{n\to \infty}\P(\|\xi-\xi_n\|_{Y}\geq \varepsilon)=0$ for all $\varepsilon>0$.

\begin{theorem}[Continuous dependence on the initial data]
\label{t:continuous_dependence}
%
Let the assumptions of Theorem \ref{t:global} be satisfied and set $\a=\a_{\crit}\stackrel{{\rm def}}{=} p(1-\frac{1}{q}-\frac{\s}{2})-1$.
Let $v_0,(v^{(n)}_0)_{n\geq 1}\subseteq L^0_{\F_0}(\O;\Bs^{2/q}_{(q,2),p}(\T^3))$ be such that  $v_0^{(n)}\to v_0$ in probability in $\Bs^{2/q}_{(q,2),p}(\T^3)$. Let $v^{(n)},v$ be the global $(p,\s,\a_{\crit},q)$-solution of \eqref{eq:primitive} with initial data $v^{(n)}_0$ and $v_0$ provided by Theorem \ref{t:global}, respectively. 
Then
$$
v^{(n)} \to v \ \text{ in probability in } \ C([0,T];\Bs^{2/q}_{(q,2),p}(\T^3))\ \text{ for all }\ T<\infty.
$$
\end{theorem}

The proofs of Theorems \ref{t:global} and \ref{t:continuous_dependence} are given in Subsections \ref{ss:proof_global} and \ref{ss:proof_continuous_dependence}, respectively.
As one may expect, the above results are consequence of \emph{energy estimates} for $(p,\a_{\crit},\s,q)$-solutions to \eqref{eq:primitive}, which can be found in Lemma \ref{l:strong_energy_estimate}, \ref{l:main_intermediate_estimate_wt} and \ref{l:energy_estimate}.
In particular, the key step in the proof of Theorems \ref{t:global} and \ref{t:continuous_dependence} is to exploit 
Assumption \ref{ass:global}  to obtain $\partial_z v\in L^2_{\loc}(0,\tau;H^1)$ a.s.\ and certain bounds for such quantity, see 
Lemma \ref{l:main_intermediate_estimate_vz}. This is consistent with the results of \cite{Ju17}, which ensures that  uniqueness for deterministic PEs holds provided $\partial_z v\in L^2_{\loc}(0,\tau;H^1)\cap C([0,\tau);L^2)$. However, our arguments differ from the one of \cite{Ju17}, as in Step 3 of \cite[Theorem 4.1]{Ju17} the author uses the global well-posedness of PEs in case of sufficiently smooth initial data (cf.\ the discussion in Subsection \ref{ss:comparison_deterministic}). 

As a by-product we obtain the following extension of \cite[Theorem 1.5]{GHKVZ14}.

\begin{corollary}[Feller property]
\label{cor:Feller}
Let the assumptions of Theorem \ref{t:global} be satisfied and let $\a=\a_{\crit}\stackrel{{\rm def}}{=} p(1-\frac{1}{q}-\frac{\s}{2})-1$. Set 
$$
(\mathcal{P}_t\varphi)(v_0)\stackrel{{\rm def}}{=}\E[\varphi(v(t))]\quad  \text{ where }\quad  \varphi\in C_{{\rm b}}(\Bs^{2/q}_{(q,2),p}(\T^3)),  \ v_0\in \Bs^{2/q}_{(q,2),p}(\T^3),
$$
and $v$ is the global $(p,\s,\a_{\crit},q)$-solution to \eqref{eq:primitive} with initial data $v_0$ provided by Theorem \ref{t:global}. Then 
$$
\mathcal{P}_t: C_{{\rm b}}(\Bs^{2/q}_{(q,2),p}(\T^3))\to C_{{\rm b}}(\Bs^{2/q}_{(q,2),p}(\T^3))\  \text{ continuously}.
$$
\end{corollary}
%
%
%

\subsection{Open problems}
\label{ss:open_problems}
In the author's opinion, the following are the major open problems concerning the global well-posedness of PEs with transport noise:
\begin{enumerate}[{\rm(a)}]
\item\label{it:rough_noise_open} The very rough regime $\g\leq \frac{1}{2}$.
\item\label{it:boundary_conditions} 
Natural boundary conditions: PEs \eqref{eq:primitive} on $\T^2\times [0,\frac{1}{2}]$ complemented with
$$
v=0\  \text{ on }\ \T^2_{x,y}\times\{0\} \quad \text{ and }\quad 
\partial_zv=0\ \text{ on }\ \T^2_{x,y}\times\{\tfrac{1}{2}\}.
$$
\end{enumerate}

\eqref{it:rough_noise_open}: Our global well-posedness result of Theorem \ref{t:global} covers the case $\g>\frac{1}{2}$. This follows by combining the restrictions $\g>1-\s$ and $\s<\frac{1}{2}$ in Assumption \ref{ass:primitive}\eqref{it:ass_primitive3} and Theorem \ref{t:local}, respectively. However, as we will show in Proposition \ref{prop:local} below, \emph{local} well-posedness of \eqref{eq:primitive} holds also if $\g<\frac{1}{2}$. However, in the latter case, the `critical spaces' coming from the application of \cite{AV19_QSEE_1,AV19_QSEE_2} have regularity strictly greater than the one considered in Theorem \ref{t:local} (i.e.\ with Sobolev index $>-\frac{1}{2}$). Therefore, it is not clear how to extend the proof of Theorems \ref{t:global} and \ref{t:continuous_dependence} to such situations, as in the latter results we essentially use the well-posedness in spaces of Sobolev index $-\frac{1}{2}$.

\eqref{it:boundary_conditions}: From a  modeling point of view, the natural boundary conditions arise by interpreting $\T^2_{x,y}\times\{0\}$ and $\T^2_{x,y}\times\{\tfrac{1}{2}\}$ as the `bottom' and the `top' of the ocean, respectively. 
In contrast to the deterministic case (see e.g.\ \cite{Kukavica_2007,HK16}), the stochastic PEs with natural boundary conditions are much more complicated. To see this, recall that in the study of linear parabolic SPDEs with transport noise, it is known that compatibility conditions on the transport noise coefficients are needed to obtain estimates in sufficiently smooth Sobolev spaces, see e.g.\ \cite{Du20,Fla90} and \cite[Subsection 6.3]{AV19_QSEE_1}. Counterexamples to regularity estimates, if such compatibility conditions are violated, can be found in 
\cite{Krylov03} and \cite[Examples 1.2 and 1.3]{Du20}. 
However, looking at \eqref{eq:primitive2}, the presence of the hydrostatic Helmholtz projection $\p$ in stochastic perturbation shows that such compatibility conditions do not have a straightforward extension to the (linearized) PEs. Moreover, due to the presence of the integral operator $\int_{\T_z}\cdot\,\dd z$ in the definition of $\p$ in \eqref{eq:def_p_q}, it is also not clear how to use weights to avoid compatibility conditions (cf.\ \cite{Kim04a,Kry94b} for the case of parabolic SPDEs).

Finally, let us highlight that periodic boundary conditions considered here are also physically relevant. Indeed, considering the ocean dynamics far from the bottom and on the top of its boundary, then it is also correct from a modeling point of view to consider no-flux boundary conditions on the `bottom' and `top', i.e. 
\begin{equation}
\label{eq:no_flux_bc}
\partial_zv=0 \ \text{ on }\ \T^2_{x,y}\times\{0,\tfrac{1}{2}\}.
\end{equation}
Surprisingly enough, no-flux boundary conditions \eqref{eq:no_flux_bc} are a special case of the periodic boundary conditions. Indeed, one can perform an \emph{even} reflection of $v$ around $\{z=\frac{1}{2}\}$ and obtain the PEs \eqref{eq:primitive} on $\T^3$. The key observation is even reflections does not alter the nonlinearities in \eqref{eq:primitive} (indeed $w(v)\partial_z v $ is preserved due to $\int_{0}^{1/2} \nabla_{x,y}\cdot v (\cdot,z)\,\dd z=0$ which is the analog of \eqref{eq:no_flux_bc} on $\T^2_{x,y}\times (0,\frac{1}{2})$).

\section{Local well-posedness -- Proof of Theorems \ref{t:local}}
\label{s:proofs_local}
Below we write $H^{s,(q,\zeta)}$ instead of $H^{s,(q,\zeta)}(\T^3;\R^2)$, if no confusion seems likely. A similar convention is employed for $\Hs^{s,(q,\zeta)}=\Hs^{s,(q,\zeta)}(\T^3)$ and similar.

\subsection{Local well-posedness and the proof of Theorem \ref{t:local}}
\label{ss:proof_local}
We begin by proving a generalization of Theorem \ref{t:local}. 
Below we use an additional parameter $\zeta$ for ruling the integrability in the vertical direction. This parameter will be useful in the proof of high-order regularity results such as \eqref{eq:inst_reg1}-\eqref{eq:inst_reg2}.
To keep track of the additional parameter, we define 
$(p,\a,\s,q,\zeta)$-solution to \eqref{eq:primitive} as in Definition \ref{def:solution} by replacing the integrability exponent in the vertical direction $2$ by $\zeta\in [2,\infty)$. In particular 
$(p,\a,\s,q,2)$-solutions coincide with 
$(p,\a,\s,q)$-solutions.

\begin{proposition}[Local existence and uniqueness]
\label{prop:local}
Let Assumption \ref{ass:primitive}$(p,\s,q)$ be satisfied and fix $\zeta\in [2,\infty)$. 
Suppose that $q>\frac{2}{2-\s}$ and $(p,\a)$ satisfy of one of the following conditions:
\begin{align}
\label{eq:local_1_assumption_general1}
 \Big[\frac{1+\a}{p}+\frac{1}{q}+\frac{\delta}{2}\leq 1 \  \ \text{ and } \ \  q<\frac{2}{\s}\Big],&\\
\label{eq:local_1_assumption_general2}
 \Big[\frac{1+\a}{p}+\delta\leq 1 \ \  \text{ and } \ \  q\geq \frac{2}{\s}\Big].&
\end{align}
Then for each 
$
v_0\in L^0_{\F_0}(\O;\Bs^{2-\s-2\frac{1+\a}{p}}_{(q,\zeta),p}),
$
 \eqref{eq:primitive} has a (unique) $(p,\a,\s,q,\zeta)$-solution
$(v,\tau)$ such that  a.s.\ $\tau>0$  and
\begin{align}
\label{eq:reg_critical11}
v&\in  H^{\vartheta,p}_{\loc}([0,\tau),w_{\a};\Hs^{2-\s-2\vartheta,(q,\zeta)})\ \text{ for all }\vartheta\in [0,\tfrac{1}{2}),\\
\label{eq:reg_critical12}
v&\in C([0,\tau);\Bs^{2-\s-2\frac{1+\a}{p}}_{(q,\zeta),p});
\end{align}
where in the case $p=2$, then \eqref{eq:reg_critical11} holds only for $\vartheta=0$.
\end{proposition}

Before going into the proof of the above result we show that Theorem \ref{t:local} is a special case of Proposition \ref{prop:local}.
 
\begin{proof}[Proof of Theorem \ref{t:local}]
Recall that $\a=\a_{\crit}=p(1-\frac{\delta}{2}-\frac{1}{q})-1$. Note that $\a\in [0,\frac{p}{2}-1)$ and $q<\frac{2}{\s}$ due to the assumptions $q<\frac{2}{1-\s}$ and $\s<\frac{1}{2}$ in Theorem \ref{t:local}.
Now we can apply Proposition \ref{prop:local} to $\a=\a_{\crit}$, $\zeta=2$ and $(p,\s,q)$ as in Assumption \ref{ass:primitive}. To conclude, it remains to note 
$$
2-\s-2\frac{1+\a_{\crit}}{p}=\frac{2}{q}
\qquad\Longrightarrow \qquad
\Bs^{2-\s-2\frac{1+\a_{\crit}}{p}}_{(q,2),p}=\Bs^{2/q}_{(q,2),p}.
$$
Hence the space for the initial data coincide with the one in Theorem \ref{t:local}.
\end{proof}

The proof of Proposition \ref{prop:local} relies on the theory of critical spaces for stochastic evolution equations developed in \cite{AV19_QSEE_1,AV19_QSEE_2}. To this end, we introduce some notation. For $(p,\delta,q)$ as in Assumption \ref{ass:primitive} and $\zeta\in [2,\infty)$, let 
\begin{equation}
\label{eq:choice_X0X1}
\begin{aligned}
X_0&=\Hs^{-\s,(q,\zeta)}, \qquad \ \  X_1=\Hs^{2-\s,(q,\zeta)}, &\\ 
X_{\vartheta}&=[X_0,X_1]_{\vartheta}=\Hs^{-\s+2\vartheta,(q,\zeta)}\ \ \text{ for }  \ \vartheta\in (0,1);&
\end{aligned}
\end{equation}
where we used Lemma \ref{l:anisotropic_B}\eqref{it:anisotropic_B2} and \cite[Lemma 1.51]{DK13_mixed_order} to compute $[X_0,X_1]_{\vartheta}$.

Finally, on $\R_+\times \O$ and for all $v\in X_1$, we set
\begin{equation}
\label{eq:def_ABFG}
\begin{aligned}
A(\cdot) v&= -\nabla\cdot(a\cdot \nabla v )- (b\cdot\nabla) v+\fts(\cdot)v, \\  
B(\cdot)v&= (\p[(\sigma_n\cdot\nabla) v])_{n\geq 1},\\
F(\cdot,v)&= \p[-(v\cdot\nabla_{x,y})v -w(v)\partial_z v +\ftg(\cdot,v)+ f(\cdot,v)],\\
G(\cdot,v)&= (\p[g_n(\cdot,v)])_{n\geq 1},
\end{aligned}
\end{equation}
where $w(v)$ and $(\fts,\ftg)$ are as in \eqref{eq:def_w} and \eqref{eq:def_f_tilde2}-\eqref{eq:def_f_tilde1}, respectively.

With the above notation the stochastic PEs \eqref{eq:primitive} can be rewritten as a stochastic evolution equation on $X_0$:
\begin{equation}
\label{eq:SEE}
\left\{
\begin{aligned}
\dd v + A(t)v\, \dd t &= F(t,v)\,\dd t + (B(t)v+G(t,v))\,\dd \Br_{\ell^2}(t),\quad t\in\R_+,\\
 v(0)&=v_0.
\end{aligned}
\right.
\end{equation}
where $\Br_{\ell^2}$ is the $\ell^2$-cylindrical Brownian motion induced by $(\beta^n)_{n\geq 1}$, see \eqref{eq:def_Br}. 
To apply the main results of \cite{AV19_QSEE_1} to \eqref{eq:SEE} we need the following ingredients:
\begin{itemize}
\item Estimates of the nonlinearities $(F,G)$ in $X_{\vartheta}$-spaces.
\item Stochastic maximal $L^p(L^q)$-regularity estimates for the linearized problem. 
\end{itemize}

The above points are addressed in Lemmas \ref{l:nonlinearities} and \ref{l:smr} of Subsection \ref{sss:local_estimates_nonlinearity} and \ref{sss:local_smr}, respectively.
The proof of Proposition \ref{prop:local} is given in Subsection \ref{sss:proof_local} below.

\subsection{Estimates of the nonlinearities}
\label{sss:local_estimates_nonlinearity}
We begin with the following lemma which relates $X_{\vartheta}=H^{-\s+2\vartheta,(q,\zeta)}$ with iterated Bessel-potential ones. Throughout this subsection we use the following shorthand notation $H^{r,q}_{x,y}(H^{t,\zeta}_z)$ for  $H^{r,q}(\T^2_{x,y};H^{t,\zeta}(\T_z))$. 

We begin by stating the following special case of Lemma \ref{l:anisotropic_B}\eqref{it:anisotropic_B3}.

\begin{lemma}
\label{l:iterated_H}
Let $q,\zeta\in (1,\infty)$ and $s,t\geq 0$. Then 
$$
H^{s+t,(q,\zeta)}\embed H^{r,q}_{x,y}(H^{t,\zeta}_z) \quad \text{ and }\quad
 H^{-t,q}_{x,y}(H^{-s,\zeta}_z) 
\embed 
H^{-t-s,(q,\zeta)}.
$$
\end{lemma}

Below $\g(\ell^2,X_{1/2})$ denotes the space of $\g$-radonifying operators from $\ell^2$ to $X_{1/2}$, see e.g.\ \cite[Chapter 9]{Analysis2} for details. Their use is motivated by their role in stochastic integration and maximal $L^p$-regularity, see e.g.\ \cite{NVW1}, \cite[Section 7]{NVW13} and \cite[Section 3]{AV19_QSEE_1}. 
In the following we employ the Fubini-type result:
\begin{equation}
\label{eq:gamma_identification}
\g(\ell^2,H^{s,(q,\zeta)})=H^{s,(q,\zeta)}(\ell^2) \quad \text{ for all }\ s\in\R, \ q\in (1,\infty).
\end{equation}
The identification \eqref{eq:gamma_identification} is a consequence of \cite[Theorem 9.3.6]{Analysis2} and the fact that 
$$
(1-\Delta)^{s/2}:H^{s,(q,\zeta)}\to L^{(q,\zeta)}\ \ \text{is an isomoprhism}.$$

\begin{lemma}
\label{l:nonlinearities}
Let $(p,\s,q)$ and  $X_{\theta}$ be as in Assumption \ref{ass:primitive} and in \eqref{eq:choice_X0X1}, respectively. Fix $\zeta	\in [2,\infty)$ and suppose that $q>\frac{2}{2-\s}$.
Set
\begin{equation*}
\beta=
\left\{\begin{aligned}
&\frac{1}{2}+\frac{\delta}{4}+\frac{1}{2q}, \quad &\text{ if }&\ q<\frac{2}{\delta},\\
&\frac{1}{2}+\frac{\delta}{2}, \quad &\text{ if }&\ q\geq \frac{2}{\delta}.
\end{aligned}\right.
\end{equation*}
Then, for all $v,v'\in X_1$,
\begin{align*}
\|F(\cdot,v)\|_{X_0}&\lesssim (1+\|v\|_{X_{\beta}})\|v\|_{X_{\beta}},\\
\|F(\cdot,v)-F(\cdot,v')\|_{X_0}&\lesssim (1+\|v\|_{X_{\beta}}+\|v'\|_{X_{\beta}})\|v-v'\|_{X_{\beta}},\\
\|G(\cdot,v)\|_{\g(\ell^2,X_{1/2})}&\lesssim (1+\|v\|_{X_{\beta}})\|v\|_{X_{\beta}},\\
\|G(\cdot,v)-G(\cdot,v')\|_{\g(\ell^2,X_{1/2})}&\lesssim (1+\|v\|_{X_{\beta}}+\|v'\|_{X_{\beta}})\|v-v'\|_{X_{\beta}}.
\end{align*}
\end{lemma}

\begin{proof}
For exposition convenience, we divide the proof into two steps. In both cases we prove the estimate for the differences, as  $\|F(\cdot,0)\|_{X_0}<\infty$ and $\|G(\cdot,0)\|_{\g(\ell^2,X_{1/2})}<\infty$ by Assumption \ref{ass:primitive}\eqref{it:ass_primitive7}.

\emph{Step 1: Estimates for $F$}. For $v,v'\in X_1$, we set
$$
B_0(v,v')=-\p[(v\cdot\nabla_{x,y})v'] \quad \text{ and }\quad B_1(v,v')=-\p[w(v)\partial_z v'].
$$
Note that $F(\cdot,v)= b_0(v,v)+b_1(v,v)+\p[f(\cdot,v)]$. Since $f$ is globally Lipschitz by Assumption \ref{ass:primitive}, it is enough to estimate $b_0$ and $b_1$.
We begin by looking at $b_1$ as its analysis will lead to all the restrictions of the current lemma.

\emph{Substep 1a: $\|b_1(v,v)\|_{X_0}\lesssim\|v\|_{X_{\beta}}\|v'\|_{X_{\beta}}$ for all $v,v'\in X_1$.}
Let $r$ be either [$\frac{2}{r}=\delta+\frac{2}{q}$ if $\s>0$] or [$r=q$ otherwise]. Note that, if $\s>1$, then $q>\frac{2}{2-\s}$ and therefore $r\in (1,\infty)$. By Lemma \ref{l:iterated_H} and the choice of $r$, we have $
L^r_{x,y}(L^{\zeta}_z)\embed H^{-\s,(q,\zeta)}$. Thus, for all $v,v\in X_1$,
\begin{align*}
\|b_1(v,v')\|_{X_0}
&\lesssim \|w(v)\partial_z v'\|_{L^r_{x,y}(L^{\zeta}_z)}\\
&\lesssim\|w(v)\|_{L^{2r}_{x,y}(L^{\infty}_z)}\|\partial_z v'\|_{L^{2r}_{x,y}(L^{\zeta}_z)}\\
&\stackrel{(i)}{\lesssim}\|\nabla  v\|_{L^{2r}_{x,y}(L^{\zeta}_z)}\|\partial_z v'\|_{L^{2r}_{x,y}(L^{\zeta}_z)}\\
&\lesssim \|v\|_{H_{x,y}^{1,2r}(L^{\zeta}_z)}\|v'\|_{L_{x,y}^{2r}(H^{1,\zeta}_z)}\\
&\stackrel{(ii)}{\lesssim} \|v\|_{X_{\beta}}\|v'\|_{X_{\beta}},
\end{align*}
where in $(i)$ we use \eqref{eq:def_w}. For $(ii)$, we distinguish two cases:
\begin{itemize}
\item \emph{Case 1: $q<\frac{2}{\s}$}. In this case $\beta=\frac{1}{2}+\frac{\s}{4}+\frac{1}{2q}<1$ satisfies $-\s+2\beta>1$ and $-\s+2\beta-\frac{2}{q}=1-\frac{2}{2r}$. Thus Lemma \ref{l:iterated_H} and the Sobolev embeddings yield
\begin{align*}
X_{\beta}
&\embed H_{x,y}^{-\s+2\beta,q}(L^{\zeta}_z)\cap H^{-\s+2\beta-1,q}_{x,y}(H^{1,\zeta}_z)\\
&\embed H_{x,y}^{1,2r}(L^{\zeta}_z)\cap L^{2r}_{x,y}(H^{1,\zeta}_z).
\end{align*}
\item \emph{Case 2: $q\geq \frac{2}{\s}$}. In this situation $q\geq 2r$ and $-\s+2\beta=1$. The conclusion now follows as in Step 1 by using the trivial embedding $L^q_{x,y}\embed L^{2r}_{x,y}$.
\end{itemize}  

\emph{Substep 1b: $\|b_0(v,v)\|_{X_0}\lesssim\|v\|_{X_{\beta}}\|v'\|_{X_{\beta}}$ for all $v,v'\in X_1$.}
Let $r\in (1,\infty)$ be as in Substep 1a. 
Again, by $L^r_{x,y}(L^{\zeta}_z)\embed H^{-\s,(q,\zeta)}$, for all $v,v'\in X_1$,
\begin{align*}
\|b_0(v,v')\|_{X_0}
&\lesssim \|(v\cdot\nabla_{x,y})v'\|_{L^{r}_{x,y}(L^{\zeta}_z)}\\
&\lesssim \|v\|_{L^{2r}_{x,y}(L^{\infty}_z)}\| \nabla_{x,y}v'\|_{L^{2r}_{x,y}(L^{\zeta}_z)}\\
&\lesssim \| v\|_{L^{2r}_{x,y}(H^{1,\zeta}_z)}\| v'\|_{H^{1,2r}_{x,y}(L^{\zeta}_z)},
\end{align*}
where in the last step we used that $H^{1,\zeta}_z\embed L^{\infty}_z$. The conclusion follows from the above estimate with the one of $
\|b_1(v,v')\|_{X_0}$ in Substep 1a.

\emph{Substep 1c: $\|f(\cdot,v)-f(\cdot,v')\|_{X_0}\lesssim \|v-v'\|_{X_{\beta}}$}. Follows from Assumption \ref{ass:primitive}\eqref{it:ass_primitive8}, the arguments used in Substeps 1a and 1b, and the fact that $X_{\beta}\embed H^{1,(q,\zeta)}$ as $\beta\geq \frac{1}{2}+\frac{\s}{2}$ by construction.

\emph{Step 2: Estimate for $G$}. We claim that $\|G(\cdot,v)-G(\cdot,v')\|_{\g(\ell^2,X_{1/2})}\lesssim \|v-v'\|_{X_{\beta}}$. The latter follows from Assumption \ref{ass:primitive}\eqref{it:ass_primitive8},  \eqref{eq:gamma_identification}, the chain rule and $X_{\beta}\embed H^{1,(q,\zeta)}$ as in Substep 1c.
\end{proof}

\subsection{Stochastic maximal $L^p(L^q)$-regularity for the linearized problem}
\label{sss:local_smr}
In this subsection we investigate maximal $L^p$-regularity estimates for the following linearization of \eqref{eq:primitive}:
\begin{equation}
\label{eq:primitive_linear_zero}
\left\{
\begin{aligned}
		&\partial_t v   
		=\nabla\cdot(a\cdot\nabla v) + (b\cdot\nabla)v + \fts(\cdot) v + f	\\
		&\qquad \qquad \qquad \qquad 
		+\sum_{n\geq 1}\big(\p [(\sigma_{n}\cdot\nabla) v]+ g_{n}\big)\, \dot{ \beta}_t^n, &\text{ on }&\Tor^3,\\
&v(0,\cdot)=0,&\text{ on }&\Tor^3.
\end{aligned}
\right.
\end{equation}
where $\fts$ is as in \eqref{eq:def_f_tilde1} and, for $T\in (0,\infty)$,
\begin{align}
\label{eq:f_g_assumptions_1}
f&\in L^p((0,T)\times \O,w_{\a};\Hs^{-s,(q,\zeta)}),\\
\label{eq:f_g_assumptions_2}
g=(g_n)_{n\geq 1}&\in L^p((0,T)\times \O,w_{\a};\Hs^{1-s,(q,\zeta)}(\ell^2)),
\end{align}
and the parameters $(p,s,\a,q)$ satisfy
\begin{equation}
\label{eq:parameters_smr}
\begin{aligned} 
s\in (-\infty,1) \ \text{ and either }& \ [p=q=\zeta=2\text{ and }\a=0] \\
\text{ or }&\ \Big[ q,\zeta\in [2,\infty),  \ p\in (2,\infty)\text{ and } \a\in [0,\tfrac{p}{2}-1)\Big].
\end{aligned}
\end{equation}
Here we also consider the case $s<0$ which is needed to prove Theorem \ref{t:regularity2}.

We study \eqref{eq:primitive_linear_zero} under the following assumption.

\begin{assumption} Let $h\in [-1,\infty)$ and assume that  the following hold.
\label{ass:primitive_linear}
\begin{enumerate}[{\rm(1)}]
\item\label{it:ass_primitive_linear1} For all $n\geq 1$ the following mapping are $\Progress\otimes \Borel(\T^3)$-measurable
\begin{align*}
 a=(a_{\eta,\xi})_{\eta,\xi\in \{x,y,z\}}&:\R_+\times \O\times \T^3\to \R^{3\times 3},\\
b=(b_{\xi})_{\xi\in \{x,y,z\}},\ \sigma_n=(\sigma_{n,\xi})_{\xi\in \{x,y,z\}}&:\R_+\times \O\times \T^3\to \R^3,\\
 b_{0,n}&:\R_+\times \O\times \T^3\to \R.
\end{align*}
\item\label{it:ass_primitive_linear2} For all $n\geq 1$, the mappings
$$
(a_{\eta,\xi})_{\eta,\xi\in \{x,y\}},  (\sigma_{n,\xi})_{\xi\in \{x,y\}}   \text{ are independent of }z\in \T_z.
$$
\item\label{it:ass_primitive_linear3} There exist $M,\g>0$ such that $\g> 1-s$, and for a.a.\ $(t,\om)\in  \R_+\times \O$,
$$
\|a(t,\om,\cdot)\|_{C^{\g}(\Tor^3;\R^{3\times 3})}
+\|(\sigma_n(t,\om,\cdot))_{n\geq 1}\|_{C^{\g}(\Tor^3;\ell^2)}\leq M   .
$$
\item\label{it:ass_primitive_linear4} There exists $\nu\in (0,1)$ such that for all $\lambda=(\lambda_{\xi})_{\xi\in \{x,y,z\}}\in \R^3$
\begin{equation*}
\sum_{\eta,\xi\in \{x,y,z\}} \Big(a_{\eta,\xi}-\frac{1}{2}\sum_{n\geq 1} \sigma_{n,\eta} \sigma_{n,\xi}\Big) 
\lambda_{\eta}\lambda_{\xi}  \geq  \ellip|\lambda|^2 \ \ \text{ a.e.\ on }\R_+\times \O.
\end{equation*}
 \item\label{it:ass_primitive_linear5} There exists $N>0$ such that,  for a.a.\ $(t,\om)\in \R_+\times \O$,
 \begin{align*}
 \|b(t,\om,\cdot)\|_{L^{\infty}(\T^3;\R^3)}+
 \|(b_{0,n}(t,\om,\cdot))_{n\geq 1}\|_{L^{\infty}(\T^3;\ell^2)}&\leq N \ \ \text{ if }s\leq 0,\\
 \|b(t,\om,\cdot)\|_{C^{\g-1}(\T^3;\R^3)}+
 \|(b_{0,n}(t,\om,\cdot))_{n\geq 1}\|_{C^{\g-1}(\T^3;\ell^2)}&\leq N\ \  \text{ if }s>0.
 \end{align*}
\end{enumerate}
\end{assumption}

Strong solutions $v\in L^p((0,T)\times \O,w_{\a};\Hs^{-s,(q,\zeta)})$ to \eqref{eq:primitive_linear_zero} can be defined similarly to Definition \ref{def:solution} by considering the integrated version of \eqref{eq:primitive_linear_zero}.

For a brief overview of stochastic maximal $L^p$-regularity estimates for parabolic problems, we refer to \cite[Subsection 3.2]{AV19_QSEE_1}.  The parabolic version of the result below is given in Theorem \ref{t:smr_anisotropic} in Appendix \ref{app:smr}.

\begin{lemma}
\label{l:smr}
Assume that Assumption \ref{ass:primitive_linear} and  \eqref{eq:parameters_smr} hold. 
Then for all progressively measurable processes as in \eqref{eq:f_g_assumptions_1}-\eqref{eq:f_g_assumptions_2} 
there exists a unique strong and progressively measurable solution $v\in L^p((0,T)\times \O,w_{\a};\Hs^{2-s,(q,\zeta)})$ to \eqref{eq:primitive_linear_zero} such that for all $\vartheta\in [0,\frac{1}{2})$ one has $v\in L^p(\O,H^{\vartheta,p}(0,T,w_{\a};\Hs^{2-s-2\vartheta,(q,\zeta)}))$ and 
\begin{align}
\label{eq:estimate_max_reg_theta_v}
\E\|v\|_{H^{\vartheta,p}(0,T,w_{\a};H^{2-s-2\vartheta,(q,\zeta)})}^p 
&\lesssim_{\vartheta,p} 
\E\|f\|_{L^p(0,T,w_{\a};H^{-s,(q,\zeta)})}^p\\
\nonumber
&\quad + \E\|g\|_{L^p(0,T,w_{\a};H^{1-s,(q,\zeta)}(\ell^2))}^p,
\end{align}
where in the case $p=2$ and $\vartheta>0$ the space $H^{\vartheta,p}(0,T,w_{\a};H^{2-s-2\vartheta,(q,\zeta)})$ is replaced by
$C([0,T];H^{1-s,q})$. 
\end{lemma}

Note that the choice $\zeta=2$ is allowed. 
As the proof below shows we may replace $(0,T)$ by $(0,\tau)$ where $\tau$ is a stopping time taking value in a compact set of $[0,\infty)$.
By \cite[Proposition 3.10]{AV19_QSEE_1}, the above result also implies stochastic maximal $L^p$-regularity estimates for solutions to \eqref{eq:primitive_linear_zero} with non-trivial initial data $u(0)\in L^p_{\F_0}(\O;\Bs^{2-s-2\frac{1+\a}{p}}_{(q,\zeta),p})$.

The proof of Lemma \ref{l:smr} simplifies in the case we assume $z$-independence of $\sigma_{n,z}$, as in the condition \eqref{eq:ass_intro_2} of Theorem \ref{t:global_intro}, instead of using Assumption \ref{ass:global}.
Indeed, in the latter case, there are no coupling terms between the barotropic and baroclinic evolutions \eqref{eq:primitive_linear_overlinev}-\eqref{eq:primitive_linear_wtv}.

\begin{proof}
Let $T<\infty$ be fixed. 
Let us begin by noticing that, by Assumption \ref{ass:primitive_linear}\eqref{it:ass_primitive_linear5}, 
\begin{align*}
\|(b\cdot\nabla)v \|_{H^{-s,(q,\zeta)}}+ 
\Big\|\sum_{n\geq 1}b_{0,n} \qq[(\sigma_n\cdot\nabla)v ] \Big\|_{H^{-s,(q,\zeta)}}
&\lesssim
\left\{
\begin{aligned}
& \|v\|_{H^{1,(q,\zeta)}} & \text{ if } &s\geq 0,\\
& \|v\|_{H^{2-s,(q,\zeta)}} & \text{ if } & s< 0,
\end{aligned}
\right.
\end{align*}
Indeed, in the case $s\geq 0$ (reps.\ $s<0$), the above follows from $L^{(q,\zeta)}\embed H^{-\s,(q,\zeta)}$ (resp.\ Lemma \ref{l:pointwise}). 
Hence the above operators are lower order, and by the perturbation result of \cite[Theorem 3.2]{AV_torus} it is enough to consider the case $b\equiv b_{0,n}\equiv 0$.

Note that in the case $\sigma\equiv 0$ and $a\equiv\mathrm{Id}$, then the content of Lemma \ref{l:smr} is a special case of Theorem \ref{t:smr_anisotropic}. Hence, arguing as in \cite[Subsection 5.2]{AV_torus} (cf.\ Lemma 5.4 there), due to the method of continuity, it is enough to show an a priori estimate of the form 
\begin{align}
\label{eq:primitive_linear_estimate}
\E\|v\|_{L^{p}(t,t+t_{\star},w_{\ell}^{t};H^{2-s,(q,\zeta)})}^p &\leq C_{0}\big(
\E\|f\|_{L^p(t,t+t_{\star},w_{\ell}^{t};H^{-s,(q,\zeta)})}^p\\
\nonumber
&\quad\quad + \E\|g\|_{L^p(t,t+t_{\star},w_{\ell}^{t};H^{1-s,(q,\zeta)}(\ell^2))}^p\big),
\end{align}
where $\ell\in \{0,\a\}$ and $w_{\ell}^t (\tau)\stackrel{{\rm def}}{=}|\tau-t|^{\ell}$, and $v\in L^p((0,T)\times \O,w_{\a};\Hs^{2-s,(q,\zeta)})$ is a strong solutions to \eqref{eq:primitive_linear_zero} with initial condition at time $t\geq 0$ and on a time interval $[t,(t+t_{\star})\wedge T]$:
\begin{equation}
\label{eq:primitive_linear}
\left\{
\begin{aligned}
		&\partial_t v   =\nabla (a\cdot \nabla v)+ \fts(\cdot)v +f +\sum_{n\geq 1}\big(\p [(\sigma_{n}\cdot\nabla) v]+ g_{n}\big)\, \dot{ \beta}_t^n, &\text{ on }&\Tor^3,\\
&v(t,\cdot)=0,&\text{ on }&\Tor^3.
\end{aligned}
\right.
\end{equation}
To extend the results of \cite{AV_torus}, it is important that $t_{\star}$ and the constant $K_0$ in \eqref{eq:primitive_linear_estimate} depend only on the parameters appearing in Assumption \ref{ass:primitive_linear}, i.e.\ elements of the set $\set \stackrel{{\rm def}}{=}\{p,q,\s,\a,M,\g\}$. To emphasize this we will write $C=C(\set)>0$ if the constant $C$ depends only on parameters in $\set$.

The idea now is to split $v$ into barotropic and baroclinic modes, i.e.\ 
\begin{equation*}
\overline{v}\stackrel{{\rm def}}{=}\fint_{\T_z} v(\cdot,z)\,\dd z \quad \text{ and } \quad
\wt{v}\stackrel{{\rm def}}{=} v-\overline{v}.
\end{equation*}
Recall that $\overline{\cdot}$ commutes with $\p$, i.e.\ $\overline{\p f}=\p_{x,y}[\, \overline{f} \,]$ for all $f\in L^2(\T^3;\R^2)$ where $\p$ and $\p_{x,y}$ are Helmholtz type projections, see Subsection \ref{ss:projections}. 
By  Assumption \ref{ass:primitive_linear}\eqref{it:ass_primitive_linear2}, one can check that $\overline{v}$ and $\wt{v}$ solve
\begin{equation}
\label{eq:primitive_linear_overlinev}
\left\{
\begin{aligned}
		&\partial_t \overline{v}   =\sum_{\eta,\xi\in \{x,y\}} \p[(\partial_{\eta}(a_{\eta,\xi} \partial_{\xi}\overline{v})
		+ \fover 	\overline{v}]
		+ \sum_{\eta\in \{x,y\}} \partial_{\eta}(\overline{a_{\eta,z}\partial_z v}) +\overline{f}	\\
		&\qquad \qquad\qquad\qquad\ \
		 +\sum_{n\geq 1}\big(\p [(\sigma_{n}\cdot\nabla_{x,y}) \overline{v}]+ \overline{\sigma_{n,z}\partial_z v}+ \overline{g}_{n}\big)\, \dot{ \beta}_t^n, &\text{on }&\Tor^3,\\
&\overline{v}(s,\cdot)=0,&\text{on }&\Tor^3,
\end{aligned}
\right.
\end{equation}
and 
\begin{equation}
\label{eq:primitive_linear_wtv}
\left\{
\begin{aligned}
		&\partial_t \wt{v}   =\nabla (a\cdot\nabla \wt{v})+\fwt v + 
		\sum_{\eta\in \{x,y\}}\big[ \partial_{z}(a_{z,\eta}\partial_{\eta} \overline{v})  
		-\partial_{\eta}(\overline{a_{\eta,z}\partial_z v}) \big]
		\\
		&\qquad \qquad\qquad\qquad
		\ \ +\wt{f} +\sum_{n\geq 1}\big((\sigma_{n}\cdot\nabla) \wt{v}-\overline{\sigma_{n,z}\partial_z v}+ \wt{g}_{n}\big)\, \dot{ \beta}_t^n, &\text{on }&\Tor^3,\\
&\wt{v}(s,\cdot)=0,&\text{on }&\Tor^3,
\end{aligned}
\right.
\end{equation}
respectively; where
\begin{align}
\nonumber
\fover\overline{v} &\stackrel{{\rm def}}{=} -\frac{h}{2}\sum_{n\geq 1}\sum_{\eta\in \{x,y\}} \partial_{\eta}( \qq_{x,y} [(\sigma_n\cdot\nabla)\overline{v}] \sigma_{n,\eta}),\\
\label{eq:def_fwt}
\fwt v&\stackrel{{\rm def}}{=} -\frac{h}{2}\sum_{n\geq 1} \qq [(\sigma_n\cdot\nabla) v] \partial_z\sigma_{n,z}.
\end{align}
We now split the proof of \eqref{eq:primitive_linear_estimate} into two steps.

\emph{Step 1: Let $M,\g>0$ be as in Assumption \ref{ass:primitive}. Then there exist $\varepsilon=\varepsilon(\set)>0$  and a constant $C_1=C_1(\set)>0$ such that, a.e.\ on $\R_+\times \O$ and for all $v\in H^{2-s,(q,\zeta)}$,}
\begin{align*}
\sup_{\eta\in \{x,y\}}\|\partial_{\eta}(\overline{a_{\eta,z}\partial_z v}) \|_{H^{-s,(q,\zeta)}} +
\|(\overline{\sigma_{n,z}\partial_z v})_{n\geq 1}\|_{H^{1-s,q}(\ell^2)}
&\leq C_1 \|v\|_{H^{2-s-\varepsilon,(q,\zeta)}}^p,\\
\|\fwt v\|_{H^{-s,(q,\zeta)}}&\leq C_1 \|v\|_{H^{2-s-\varepsilon,(q,\zeta)}}^p.
\end{align*}

We start by proving the first estimate of Step 1. Below we only consider the estimate of $(\overline{\sigma_{n,z}\partial_z v})_{n\geq 1}$, the other follows similarly.
Fix $\eta\in (0,\g-1+\s)$ and $\g_0\in (1-\s,\g)$ such that $\g_0+\varepsilon<\g$. Note that 
$$
\overline{\sigma_{n,z} \partial_z v}=\overline{\, [(1-\partial^2_z)^{\varepsilon/2} \sigma_{n,z}] [(1-\partial^2_z)^{-\varepsilon/2} \partial_z v] \, }.
$$
By the pointwise multiplication result of Lemma \ref{l:pointwise} we obtain
\begin{align*}
\|(\overline{\sigma_{n,z} \partial_z v})_{n\geq 1}\|_{H^{1-s,q}(\ell^2)}
&\lesssim \|(\, [(1-\partial^2_z)^{\varepsilon/2} \sigma_{n,z}][ (1-\partial^2_z)^{-\varepsilon/2}\partial_z v]\, )_{n\geq 1}\|_{H^{1-s,(q,\zeta)}(\ell^2)}\\
&\lesssim \|((1-\partial^2_z)^{\varepsilon/2} \sigma_{n,z})_{n\geq 1} \|_{C^{\g_0}(\ell^2)}
\|(1-\partial^2_z)^{-\varepsilon/2} \partial_z v\|_{H^{1-s,(q,\zeta)}}\\
&\lesssim \|(\sigma_{n,z})_{n\geq 1} \|_{C^{\g}(\ell^2)}
\| v\|_{H^{2-s-\varepsilon,(q,\zeta)}}\\
&\lesssim_M
\| v\|_{H^{2-s-\varepsilon,(q,\zeta)}}
\end{align*}
where the last step follows from Assumption \ref{ass:primitive}.

To prove the second estimate of Step 1, let us first focus on the case $s\leq 0$. Recall that $L^{q}_{x,y}(H^{-s,\zeta})\embed H^{-s,(q,\zeta)}$ by Lemma \ref{l:anisotropic_B}\eqref{it:anisotropic_B3}. Thus, since $\qq [(\sigma_n\cdot\nabla)v]$ is $z$-independent by \eqref{eq:def_p_q}, we have
\begin{align*}
\| \fwt (\cdot)v\|_{H^{-s,(q,\zeta)}}
\lesssim \Big\|\|( \qq [(\sigma_n\cdot\nabla)v] )_{n \geq 1}\|_{\ell^2} \|(\partial_z \sigma_{n,z})_{n\geq 1}\|_{H^{-s,\zeta}_z(\ell^2)}\Big\|_{L^q_{x,y}}
\lesssim \|\nabla v\|_{H^{1,(q,\zeta)}}
\end{align*}
where in the last step we used that, by Assumption \ref{ass:primitive}\eqref{it:ass_primitive3} we have $\g>1-s$, 
$$
\sup_{ \T^2_{x,y}} \|(\partial_z \sigma_{n,z})_{n\geq 1}\|_{H^{-s,\zeta}_z(\ell^2)}\lesssim 
\sup_{ \T^2_{x,y}}  \|( \sigma_{n,z})_{n\geq 1}\|_{C^{\g}_z(\ell^2)}\leq M.
$$
Hence the second estimate of Step 1 follows as $s< 1$. The case $s\geq 0$ is analogue.

\emph{Step 2: Conclusion}. 
The stochastic maximal $L^p$-regularity estimate of \cite[Theorem 3.2]{AV21_NS} applied to \eqref{eq:primitive_linear_overlinev} and Step 1, we get, for all $t_1\in (t,T)$ and $\ell\in \{0,\a\}$,
\begin{align*}
\E\|\overline{v}\|_{L^{p}(t,t_1,w_{\ell}^t;H^{2-\s,q})}^p 
&\leq \overline{C}_2\big(  
\E\|v\|_{L^{p}(t,t_1,w_{\ell}^t;H^{2-s-\varepsilon,(q,\zeta)})}^p \\
&\ +\E\|f\|_{L^p(t,t_1,w_{\a};H^{-s,(q,\zeta)})}^p
+ \E\|g\|_{L^p(t,t_1,w_{\a};H^{1-s,(q,\zeta)}(\ell^2))}^p\big)
\end{align*}
where $\overline{C}_2=\overline{C}_2(\set)>0$.

By Theorem \ref{t:smr_anisotropic} applied to \eqref{eq:primitive_linear_overlinev} there exists $\wt{C}_{2}=\wt{C}_2(\set)>0$ such that
\begin{align*}
\E\|\overline{v}\|_{L^{p}(t,t_1,w_{\ell}^t;H^{2-s,q})}^p 
&\leq \wt{C}_2\big(  
\E\|\overline{v}\|_{L^{p}(t,t_1,w_{\ell}^t;H^{2-s,q})}^p +
\E\|v\|_{L^{p}(t,t_1,w_{\ell}^t;H^{2-s-\varepsilon,(q,\zeta)})}^p \\
&\ +\E\|f\|_{L^p(t,t_1,w_{\ell}^t;H^{-s,(q,\zeta)})}^p
+ \E\|g\|_{L^p(t,t_1,w_{\ell}^t;H^{1-s,(q,\zeta)}(\ell^2))}^p\big).
\end{align*}
Hence, putting together the previous estimates and using $v=\overline{v}+\wt{v}$, we have obtained the existence of a constant $C_2=C_2(\set)$ such that
\begin{align*}
\E\|v\|_{L^{p}(t,t_1,w_{\ell}^t;H^{2-s,(q,\zeta)})}^p 
&\leq C_2\big(  
\E\|v\|_{L^{p}(t,t_1,w_{\ell}^t;H^{2-s-\varepsilon,(q,\zeta)})}^p \\
&\ +\E\|f\|_{L^p(t,t_1,w_{\ell}^t;H^{-s,(q,\zeta)})}^p
+ \E\|g\|_{L^p(t,t_1,w_{\ell}^t;H^{1-s,(q,\zeta)}(\ell^2))}^p\big)
\end{align*}
By interpolation we have $\|v\|_{H^{2-s-\varepsilon,(q,\zeta)}}\leq \frac{1}{2C_2^{1/p}}\|v\|_{H^{2-s,(q,\zeta)}}+K \|v\|_{H^{-s,(q,\zeta)}}$. Hence the above and \cite[Proposition 2.5(2)]{AV_torus} yields the existence of a constant  $R_{|t_1-t|}(\set)>0$ satisfying $\lim_{r\downarrow 0}R_r=0$ and
\begin{align*}
\E\|v\|_{L^{p}(t,t_1,w_{\ell}^t;H^{2-s,(q,\zeta)})}^p 
&\leq 2C_2  R_{|t_1-t|} 
\E\|v\|_{L^{p}(t,t_1,w_{\ell}^t;H^{-s,(q,\zeta)})}^p \\
&+ 2C_2\big(\E\|f\|_{L^p(t,t_1,w_{\ell}^t;H^{-s,(q,\zeta)})}^p
+ \E\|g\|_{L^p(t,t_1,w_{\ell}^t;H^{1-s,(q,\zeta)}(\ell^2))}^p\big)
\end{align*}
Now \eqref{eq:primitive_linear_estimate} follows from the previous estimates by choosing $t_{\star}=t_{\star}(\set)>0$ such that $R_{t_{\star}}\leq (4K_2)^{-1}$.
\end{proof}

\subsection{Proof of Proposition \ref{prop:local}}
\label{sss:proof_local}

\begin{proof}[Proof of Proposition \ref{prop:local}]
As discussed before \eqref{eq:SEE}, the stochastic PEs \eqref{eq:primitive} can be recast as a stochastic evolution equation \eqref{eq:SEE} with the choice \eqref{eq:choice_X0X1}-\eqref{eq:def_ABFG}.    
The claim of Proposition \ref{prop:local} therefore follows from \cite[Theorem 4.8]{AV19_QSEE_1} (see also \cite[Remark 5.6]{AV19_QSEE_2}),
where the assumptions are satisfied due to Lemmas \ref{l:nonlinearities} and \ref{l:smr}. In particular, let us note that the conditions \eqref{eq:local_1_assumption_general1}-\eqref{eq:local_1_assumption_general2} are equivalent to the (sub)criticality conditions \cite[eq.\ (4.2)-(4.3)]{AV19_QSEE_1}, i.e.\ $\beta\leq 1-\frac{1}{2}\frac{1+\a}{p}$ where $\beta$ is as in Lemma \ref{l:nonlinearities}.
\end{proof}

\section{Blow-up criteria and regularity -- Proofs of Theorems \ref{t:regularity1}, \ref{t:serrin} and \ref{t:regularity2}}
The aim of this section is to prove Theorems \ref{t:regularity1}, \ref{t:serrin} and \ref{t:regularity2}. The proofs rely on the blow-up criteria and regularization phenomena proven in \cite[Sections 4 and 6]{AV19_QSEE_2} for stochastic evolution equations of the form \eqref{eq:SEE}. The results of \cite{AV19_QSEE_2} have been already applied to stochastic Navier-Stokes equations in \cite[Subsections 2.3 and 2.4]{AV21_NS} and reaction-diffusion equations in \cite[Subsections 2.2 and 2.3]{AV22_localRD} to obtain similar results for solutions of the corresponding problems. As in \cite{AV21_NS,AV22_localRD},  
 criticality plays a fundamental role and as discussed in Remark \ref{r:regularity_necessary_anisotropic} the anisotropic setting in the spatial variable is fundamental to prove  Theorem \ref{t:regularity1} for the well-known case $p=q=2$ and $\s=1$ analyzed in \cite{Primitive2,Primitive1,BS21,Debussche_2012}.  
 
 \subsection{Proof of Theorem \ref{t:regularity1}}
\label{ss:proof_regularity1}

\begin{proof}[Proof of Theorem \ref{t:regularity1}]
Let $(v,\tau)$ be $(p,\s,\a_{\crit},q)$-solution to \eqref{eq:primitive}. 
 It is convenient to divide the proof into two cases: 
\begin{enumerate}[{\rm(a)}]
\item\label{it:reg_p_equal_2} $q=p=2$ and $\s=1$.
\item\label{it:reg_p_grater_than_2} $\displaystyle{\s\in \Big(0,\frac{1}{2}\Big)}$,  $\displaystyle{ q\in \Big(\frac{2}{2-\s},\frac{2}{1-\s}\Big)} $ and $\displaystyle{
 \frac{1}{p}+\frac{1}{q}+\frac{\delta}{2}\leq 1}$.
\end{enumerate}
Note that the cases \eqref{it:reg_p_equal_2} and \eqref{it:reg_p_grater_than_2} cover the conditions in 
\eqref{eq:assumptions_critical_setting}.
Due to Lemmas \ref{l:nonlinearities}, \ref{l:smr} and the results in Appendix \ref{app:anisotropic} (see in particular Lemma \ref{l:anisotropic_B}\eqref{it:anisotropic_B4} and Corollary \ref{cor:sob_emb_besov}), the proof of Theorem \ref{t:regularity1} in the case \eqref{it:reg_p_grater_than_2} is very similar to the proof of  Part (B) of the proof of \cite[Proposition 3.1]{AV22_localRD}, therefore we only give some details for the case \eqref{it:reg_p_equal_2}. 
More precisely, we prove the following claim:

There exist $\s_{\star}\in (0,\frac{1}{2})$, $q_{\star}\in (\frac{2}{2-\s_{\star}},\frac{2}{1-\s_{\star}})$ and $p_{\star}>2$ independent of $(v_0,v)$ such that 
\begin{equation}
 \label{eq:claim_regularization_integrability_2}
 \frac{1}{p_{\star}}+\frac{1}{q_{\star}}+\frac{\delta_{\star}}{2}< 1\quad \text{and}\quad
v\in \bigcap_{\vartheta\in [0,1/2)} H^{\vartheta,p_{\star}}_{{\rm loc}} (0,\tau;H^{2-\s_{\star}-2\vartheta,(q_{\star},2)}) \ \text{ a.s.}
\end{equation}
Once \eqref{eq:claim_regularization_integrability_2} is proved, then to obtain the claim of Theorem \ref{t:regularity1} in the case \eqref{it:reg_p_equal_2}, it is enough to repeat the argument of Part (B) in the proof of \cite[Proposition 3.1]{AV22_localRD} or equivalently applying \cite[Lemma 6.10]{AV19_QSEE_2} and the case \eqref{it:reg_p_grater_than_2} of Theorem \ref{t:regularity1} (cf.\ the proof of \cite[Proposition 7.2]{AV22_localRD}).

Hence it remains to prove \eqref{eq:claim_regularization_integrability_2}. 
For future convenience, let us recall that the $(2,1,0,2)$-solution $(v,\tau)$ to \eqref{eq:primitive} satisfies:
\begin{equation}
\label{eq:claim_regularization_integrability_2_reg0}
v\in L^2_{\loc}([0,\tau);H^2)\cap C([0,\tau);H^1)\  \text{ a.s.\ }
\end{equation} 
To prove \eqref{eq:claim_regularization_integrability_2}, we adapt the arguments in the proof of \cite[Proposition 6.8]{AV19_QSEE_1} to the current situation. Unfortunately, the latter result is not directly applicable as if $\s_{\star}>0$, then $q_{\star}>\frac{2}{2-\s_{\star}}>2$. Hence it is \emph{not} possible to modify the spatial smoothness (i.e.\ from $2$ in \eqref{eq:claim_regularization_integrability_2_reg0} to $2-\s_{\star}$ in \eqref{eq:claim_regularization_integrability_2}) while keeping the integrability fixed (in contrast with the case of Navier-Stokes equations, see Part (C) in the proof of \cite[Theorem 4.1]{AV21_NS}).

To begin, let us note that Assumption \ref{ass:primitive}$(2,1,2)$ holds as we are in the case \eqref{it:reg_p_equal_2}. Therefore, Assumption \ref{ass:primitive}$(p,\s,q)$ also holds for all $\s\in (0,\frac{1}{2})$ and $q,p\in (2,\infty)$. Fix $\s_{\star}\in (0,\frac{1}{2})$ and $q_{\star}\in (\frac{2}{2-\s_{\star}},\frac{2}{1-\s_{\star}})$ for which there exists $p_{\star}\in (2,4)$ satisfying \eqref{eq:claim_regularization_integrability_2} (e.g.\ by choosing $\s_{\star}=\frac{3}{8}$ and $q_{\star}=\frac{8}{3}$).

Let us begin by noticing that Corollary \ref{cor:sob_emb_besov} implies $H^1\embed B^{2/q_{\star}}_{(q_{\star},2),p_{\star}}$ as $q_{\star},p_{\star}>2$. Thus, Theorem \ref{t:local} ensures the existence of a $(p_{\star},\s_{\star},\a_{\star},q_{\star})$-solution  $(v_{\star},\tau_{\star})$ to \eqref{eq:primitive} where $\a_{\star}\stackrel{{\rm def}}{=}p_{\star}(1-\frac{1}{q_{\star}}-\frac{\delta_{\star}}{2})-1>0$ and $v_{\star}$ satisfies
\begin{equation}
\label{eq:claim_regularization_integrability_2_reg0_star}
v_{\star}\in \bigcap_{\vartheta\in [0,1/2)} H^{\vartheta,p_{\star}}([0,\tau_{\star}),w_{\a_{\star}};H^{2-\s_{\star}-2\vartheta,(q_{\star},2)})\text{ a.s. }
\end{equation}
As in \cite[Proposition 6.8]{AV19_QSEE_1}, now the idea is to show 
\begin{equation}
\label{eq:tau_star_equal_tau_proof_regularity}
\tau_{\star}= \tau\  \text{ a.s.}  \ \ \ \text{and} \ \ \  v_{\star}=v \ \text{ a.e.\ on }[0,\tau_{\star})\times \O.
\end{equation}
Indeed, if \eqref{eq:tau_star_equal_tau_proof_regularity}, then \eqref{eq:claim_regularization_integrability_2} follows from \eqref{eq:claim_regularization_integrability_2_reg0_star}. We now split the proof of \eqref{eq:tau_star_equal_tau_proof_regularity} into three steps.

\emph{Step 1: Let $(v_{\star},\tau_{\star})$ be $(p_{\star},\s_{\star},\a_{\star},q_{\star})$-solution  $(v_{\star},\tau_{\star})$ to \eqref{eq:primitive} as above. Then }
\begin{align}
\label{eq:v_star_regularity_1}
v_{\star}
&\in L^2_{\loc}([0,\tau_{\star});H^1),\\
\label{eq:v_star_regularity_2}
(v_{\star}\cdot\nabla_{x,y})v_{\star},\  w(v_{\star})\partial_z v_{\star}
&\in L^2_{\loc}([0,\tau_{\star});L^2).
\end{align} 

Proof of \eqref{eq:v_star_regularity_1}: Note that $L^{p_{\star}}(0,t;w_{\a_{\star}})\embed L^2(0,t)$ for all $t<\infty$ as  $\frac{1}{2}>\frac{1+\a_{\star}}{p_{\star}}$ due to $q_{\star}<\frac{2}{1-\s_{\star}}$ and the H\"older inequality. Hence, as $\s_{\star}<1$, \eqref{eq:v_star_regularity_1} follows from \eqref{eq:claim_regularization_integrability_2_reg0_star}.

Proof of \eqref{eq:v_star_regularity_2}: The estimates in Step 1 of Lemma \ref{l:nonlinearities} and Lemma \ref{l:iterated_H} imply 
\begin{align*}
\|( \phi\cdot\nabla_{x,y})\phi\|_{L^2} +\| w(\phi)\partial_z \phi\|_{L^2}
\lesssim \|\phi\|_{H^{1,(4,2)}}^2\   \text{ for all }\phi\in H^{1,(4,2)}.
\end{align*}
Hence to prove \eqref{eq:v_star_regularity_2} it is enough to show that \eqref{eq:claim_regularization_integrability_2_reg0_star} implies
\begin{equation}
\label{eq:v_star_L4}
v_{\star}\in L^4_{\loc}([0,\tau_{\star});H^{1,(4,2)})\text{ a.s. }
\end{equation}
To this end, note that for all $t<\infty$ and $\vartheta_{\star}\stackrel{{\rm def}}{=}\frac{3}{2}-\frac{2}{q_{\star}}+\delta_{\star}<\frac{1}{2}$, 
\begin{align*}
H^{\vartheta_{\star},p_{\star}}(0,t,w_{\a_{\star}};H^{2-\vartheta_{\star},(q_{\star},2)})
\stackrel{(i)}{\embed }
H^{\vartheta_{\star},p_{\star}}(0,t,w_{\a_{\star}};H^{1,(4,2)})
\stackrel{(ii)}{\embed }
L^4(0,t;H^{1,(4,2)}),
\end{align*}
where $(i)$ and $(ii)$ follow from the above choice of the parameters as well as from Lemma \ref{l:anisotropic_B}\eqref{it:anisotropic_B4} and \cite[Proposition 2.7]{AV19_QSEE_1}, respectively.
Therefore \eqref{eq:v_star_L4} follows from \eqref{eq:claim_regularization_integrability_2_reg0_star} as $\vartheta_{\star}<\frac{1}{2}$.

\emph{Step 2: 
$
v_{\star}\in C([0,\tau_{\star});H^1)\cap L^2_{\loc}([0,\tau_{\star});H^2)
$, and moreover 
\begin{equation}
\label{eq:tau_star_equal_to_v}
\tau_{\star}\leq \tau\  \text{ a.s.}  \ \ \ \text{and} \ \ \  v_{\star}=v \ \text{ a.e.\ on }[0,\tau')\times \O.
\end{equation}
}
By Step 1 and Lemma \ref{l:nonlinearities}, 
\begin{align*}
-(v_{\star}\cdot\nabla_{x,y})v_{\star} - w(v_{\star})\partial_z v_{\star}
+\ftg(\cdot,v_{\star})+f(\cdot,v_{\star},\nabla v_{\star}) &\in L^2_{\loc}([0,\tau_{\star});L^2),\\
(g_n(\cdot,v_{\star}))_{n\geq 1} &\in L^2_{\loc}([0,\tau_{\star});H^{1}(\ell^2)).
\end{align*}
Due to \eqref{eq:claim_regularization_integrability_2_reg0_star}, by Lemma \ref{l:smr} and repeating the argument in Step 1 of \cite[Proposition 6.8]{AV19_QSEE_2} we obtain $v_{\star}\in C([0,\tau_{\star});H^1)\cap L^2_{\loc}([0,\tau_{\star});H^2)$. Therefore, due to Definition \ref{def:solution}, $(v_{\star},\tau_{\star})$ is a \emph{local} $(2,1,0,2)$-solution to \eqref{eq:primitive}. Now \eqref{eq:tau_star_equal_to_v} follows from the maximality of $(2,1,0,2)$-solution to \eqref{eq:primitive} (cf.\ Definition \ref{def:solution}\eqref{it:def_sol2}).

\emph{Step 3: \eqref{eq:tau_star_equal_tau_proof_regularity} holds.} 
Due to \eqref{eq:tau_star_equal_to_v}, it remains to prove $\P(\tau_{\star}<\tau)=0$. As in \cite[Step 2 of Proposition 6.8]{AV19_QSEE_2}, we employ a blow-up criterion for $(v_{\star},\tau_{\star})$. More precisely, by Lemma \ref{l:nonlinearities} and \cite[Theorem 4.11]{AV19_QSEE_2}, for all $T\in (0,\infty)$,
\begin{equation}
\label{eq:blow_up_star}
\P\big( \tau_{\star}<T,\, \|v_{\star}\|_{L^{p_{\star}}(0,\tau_{\star};H^{\mu_{\star},(q_{\star},2)})}<\infty\big)=0 \ \  \text{ where }\ \ \mu_{\star}\stackrel{{\rm def}}{=}
\frac{2}{q_{\star}}+\frac{2}{p_{\star}},
\end{equation}
and we used that $[H^{-\s_{\star},(q_{\star},2)},H^{2-\s_{\star},(q_{\star},2)}]_{1-\frac{\a_{\star}}{p}}= H^{\mu_{\star},(q_{\star},2)}$ by Lemma \ref{l:anisotropic_B}\eqref{it:anisotropic_B2} and $2-\s_{\star}-2\frac{\a_{\star}}{p_{\star}}=\mu_{\star}$.

Note that, by Step 2 and \eqref{eq:claim_regularization_integrability_2_reg0}, a.s.\ on $\{\tau_{\star}<\tau\}$,
\begin{align}
\label{eq:v_vstar_regularity_bootstrap}
v_{\star} 
= v 
&\in L^2(0,\tau_{\star};H^2)\cap L^{\infty}(0,\tau_{\star};H^1)  & &\\ 
\nonumber
& \embed L^{p_{\star}}(0,\tau_{\star};H^{1+2/p_{\star}})&&\text{(Interpolation)}\\
\nonumber
& \embed L^{p_{\star}}(0,\tau_{\star};H^{\mu_{\star},(q_{\star},2)})&&\text{(Sobolev emb.\ -- Lemma \ref{l:anisotropic_B}\eqref{it:anisotropic_B4})}. 
\end{align}
Hence
\begin{align*}
\P(\tau_{\star}<\tau)
=\ &\lim_{T\uparrow \infty}
\P(\tau_{\star}<T,\, \tau_{\star}<\tau)\\
\stackrel{\eqref{eq:v_vstar_regularity_bootstrap}}{=}&\lim_{T\uparrow \infty} 
\P\big(\tau_{\star}<T,\, \tau_{\star}<\tau, \, \|v_{\star}\|_{L^{p_{\star}}(0,\tau_{\star};H^{\mu_{\star},(q_{\star},2)}}<\infty\big)\\
\leq\  & \limsup_{T\uparrow \infty}  \P\big(\tau_{\star}<T, \, \|v_{\star}\|_{L^{p_{\star}}(0,\tau_{\star};H^{\mu_{\star},(q_{\star},2)}}<\infty\big)\stackrel{\eqref{eq:blow_up_star}}{=}0.
\end{align*}
Hence, Step 2 and the above imply that \eqref{eq:tau_star_equal_tau_proof_regularity} holds.
\end{proof}

\begin{remark}[Necessity of $\x$-anisotropy for instantaneous regularization]
\label{r:regularity_necessary_anisotropic}
The $\x$-anisotropy of function spaces used in Theorem \ref{t:local} plays a central role in the first regularity improvement   \eqref{eq:claim_regularization_integrability_2_reg0_star} of the $(2,1,0,2)$-solution $(v,\tau)$ to \eqref{eq:primitive}.
The latter is visible in the last the embedding of 
\eqref{eq:v_vstar_regularity_bootstrap}, which is \emph{not} true if one replaces $H^{\mu_{\star},(q_{\star},2)}$ by $H^{\mu_{\star},(q_{\star},\zeta)}$ for any $\zeta>2$. In particular, the isotropic setting $\zeta=q_{\star}>2$ does not work for improving the regularity of $(2,1,0,2)$-solutions, which have already been studied but for which Theorem \ref{t:regularity1} seems new, see e.g.\ \cite{Primitive2,Primitive1,BS21,Debussche_2012}. 
\end{remark}

\subsection{Proof of Theorems  \ref{t:serrin} and  \ref{t:regularity2}}
\label{ss:proof_serrin_regularity}
We begin with Theorem  \ref{t:serrin}. As the proof is similar to the one of 
\cite[Theorem 2.10]{AV22_localRD} we only provide a sketch.

\begin{proof}[Proof of Theorem  \ref{t:serrin} -- Sketch]
As commented below the statement of Theorem \ref{t:serrin}, \eqref{it:serrin2} is a consequence of \eqref{it:serrin1}. Hence we only prove the latter. 
As recalled in \eqref{eq:blow_up_star}, by \cite[Theorem 4.11]{AV19_QSEE_2} any $(p,\s,\a_{\crit},q)$-solution $(v,\tau)$ to \eqref{eq:primitive} with $\a_{\crit}=p(1-\frac{1}{q}-\frac{\delta}{2})-1$ satisfies, for all $T\in (0,\infty)$,
\begin{equation}
\label{eq:blow_up_normal}
\P\big( \tau<T,\, \|v\|_{L^{p}(0,\tau;H^{\mu,(q,2)})}<\infty\big)=0 \ \  \text{ where }\ \ \mu =
\frac{2}{q}+\frac{2}{p}.
\end{equation}
Now the extrapolation of \eqref{eq:blow_up_normal} to the $(p_0,\s_0,q_{0})$-setting follows by Theorem \ref{t:serrin} as in the proof of
\cite[Theorem 2.10]{AV22_localRD} (or equivalently \cite[Lemma 6.10]{AV19_QSEE_2}).
\end{proof}

It remains to prove Theorem \ref{t:regularity2}. The proof follows as the one of \cite[Theorem 2.7]{AV21_NS} and therefore again we only provide a sketch. Parallel to \cite[Lemma 4.3]{AV21_NS}, we extend Lemma \ref{l:nonlinearities} to the case of high-order smoothness $s>0$.

\begin{lemma}
\label{l:nonlinearities_high_order}
Let the conditions \eqref{it:reg_f_highreg1}-\eqref{it:reg_f_highreg3} of Theorem \ref{t:regularity2} be satisfied. 
Let $F $ and $G$ be as in \eqref{eq:def_ABFG}. Assume that $p,q\in (2,\infty)$, $s\geq 0$ and $\a\in [0,\frac{p}{2}-1)$ satisfy 
\begin{equation}
\label{eq:condition_high_order_regularity}
2\frac{1+\a}{p}+\frac{d}{q}<1+s.
\end{equation}
Moreover,
 set 
$X_{j}=\Hs^{2+s,q}$ for $j\in \{0,1\}$ and $\Xap=(X_0,X_1)_{1-\frac{1+\a}{p},p}=\Bs^{2+s-2\frac{1+\a}{p}}_{q,p}$. Then, for all $n\geq 1$ and $v,v'\in X_1$ such that $\|v\|_{\Xap},\|v'\|_{\Xap}\leq n$,
\begin{align*}
\|F(\cdot,v)\|_{X_0}&\lesssim_n \|v\|_{\Xap},\\
\|F(\cdot,v)-F(\cdot,v')\|_{X_0}&\lesssim_n \|v-v'\|_{\Xap},\\
\|G(\cdot,v)\|_{\g(\ell^2,X_{1/2})}&\lesssim_n \|v\|_{\Xap},\\
\|G(\cdot,v)-G(\cdot,v')\|_{\g(\ell^2,X_{1/2})}&\lesssim_n \|v-v'\|_{\Xap}.
\end{align*}
\end{lemma}

\begin{proof}
To prove the lemma is it enough to show that 
\begin{equation}
\label{eq:nonlinearity_in_high_order_spaces}
\|w(v)\partial_z v'\|_{H^{s,(q,\zeta)}}\lesssim \|v\|_{\Xap}\|v'\|_{\Xap}\ \ \text{ for all }v,v'\in\Xap.
\end{equation}
The remaining terms can be estimated as in \cite[Lemma 4.3]{AV21_NS}.

To prove \eqref{eq:nonlinearity_in_high_order_spaces} note that, by Sobolev embeddings,  \eqref{eq:condition_high_order_regularity} implies 
$$
\Xap \embed C^{1+s+\varepsilon} \ \  \text{ where }\ \ \varepsilon=2\frac{1+\a}{p}+\frac{d}{q}-1>0.
$$
Hence, by Lemma \ref{l:anisotropic_B}\eqref{it:anisotropic_B2} and complex bilinear interpolation (see e.g.\ \cite[Theorem 4.4.4]{BeLo}), to prove \eqref{eq:nonlinearity_in_high_order_spaces}, it is enough to show that for all $k\in \N_{\geq 1}$
$$
\|w(v)\partial_z v'\|_{H^{k,(q,\zeta)}}\lesssim \|v\|_{C^{k+1}}\|v'\|_{C^{k+1}},  \ \text{ if }\ \  \int_{\T_z} \nabla_{x,y}\cdot \phi (\cdot,z)\,\dd z\equiv 0 \text{ for } \phi\in \{v,v'\}.
$$
The previous readily follows from the fact that $\partial_z (w(v))=-\nabla_{x,y}\cdot v $ in $\D'(\T^3)$ as 
$\int_{\T_z} \nabla_{x,y}\cdot v (\cdot,z)\,\dd z=0$.
\end{proof}

\begin{proof}[Proof of Theorem  \ref{t:regularity2} -- Sketch]
Due to Lemmas \ref{l:smr} with $s<0$ and \ref{l:nonlinearities_high_order}, the proof of Theorem \ref{t:regularity2} is a straightforward adaptation of the one of \cite[Theorem 2.7]{AV21_NS}.
\end{proof}

\section{Global well-posedness --  Proofs of Theorems \ref{t:global} and  \ref{t:continuous_dependence}}
\label{s:global_proofs}
As usual in parabolic (S)PDEs, global well-posedness follows by combining energy estimates and blow-up criteria such as Theorem \ref{t:serrin}. The following is well suited to obtain Theorems \ref{t:global} and \ref{t:continuous_dependence}.

\begin{lemma}[Strong energy estimate]
\label{l:strong_energy_estimate}
Let the Assumptions of Theorem \ref{t:global} be satisfied. Let $(v,\tau)$ be the $(p,\a,\s,q)$-solution provided by Theorem \ref{t:local}. Let $\pz\in (2,\infty)$ be such that $ \frac{1}{\pz}+\frac{1}{q}+\frac{\delta}{2}= 1$. 
Then, for all $0<s<T<\infty$,
\begin{equation}
\label{eq:energy_estimate1}
\sup_{t\in [s,\tau\wedge T)}\|v(t)\|_{B^{2/q}_{(q,2),\pz}}^{\pz}+\int_{s}^{\tau\wedge T} \|v(t)\|^{\pz}_{H^{2-\s,(q,2)}}\,\dd t <\infty \ \ \text{ a.s. }
\end{equation}
More precisely, there exists a constant $C_0>0$ independent of $(s,v_0)$ such that, for all $\lambda>e$, $L>0$ and $ \Gamma\in \F_s $ satisfying $\tau>s$ a.s.\ on $\Gamma$,  
it holds that
\begin{align}
\label{eq:energy_estimate2_vz}
\P\Big(\Gamma\cap \Big\{
\sup_{r\in [s,\tau\wedge T)}\|v(t)\|_{B^{2/q}_{(q,2),\pz}}^{\pz}
+\int_{s}^{\tau\wedge T} \|v(t)\|^{\pz}_{H^{2-\s,(q,2)}}\,\dd t \geq \lambda\Big\}\Big)&\\
\nonumber
\leq \frac{C_0}{\log\log(\lambda)} \Big(1+\E\big[\one_{\Gamma} \|v(s)\|_{B^{2/q}_{(q,2),\pz}}^{\pz}\big]+
\E\big[\one_{\Gamma} \|v(s)\|_{H^1}^{4}\big]\Big)&.
\end{align}
\end{lemma}

Note that the choice of $\pz\in (2,\infty)$ is possible as $q\in (\frac{2}{2-\s},\frac{2}{1-\s})$. 
Moreover, Theorem \ref{t:regularity1} and $s>0$ ensure that all the norms in \eqref{eq:energy_estimate2_vz} are well-defined (note that, if $p>p_0$, then $B^{2/q}_{(q,2),\pz}\not\subseteq B^{2/q}_{(q,2),p}$). 

The estimate \eqref{eq:energy_estimate1} follows from \eqref{eq:energy_estimate2_vz}. To see this, let
\begin{equation}
\label{eq:Gamma_L_smooth}
\Gamma_L=\big\{\tau>s,\,  \|v(s)\|_{B^{2/q}_{(q,2),\pz}}+ \| v(s)\|_{H^1}\leq L\big\}\in \F_s \text{ for some }L\geq 1.
\end{equation} 
Note that $\lim_{L\to \infty}\P(\{\tau>s\}\setminus \Gamma_L)=0$ by Theorem \ref{t:regularity1}. Thus, taking $\Gamma=\Gamma_L$ in \eqref{eq:energy_estimate2_vz}, \eqref{eq:energy_estimate1} follows by  
letting $\lambda\to \infty$ and afterwards $L\uparrow \infty$.

This section is organized as follows. In Subsection \ref{ss:proof_global} we prove Theorem \ref{t:global} by using \eqref{eq:energy_estimate1}. In Subsection \ref{ss:proof_main_estimate_conclusion} we prove Lemma \ref{l:strong_energy_estimate} assuming an a priori estimate $\partial_z v$ whose proof is postponed in Section \ref{s:energy_estimate_proof}. Finally, in Subsection \ref{ss:proof_continuous_dependence} we derive Theorem \ref{t:continuous_dependence} from the tail estimate of \eqref{eq:energy_estimate2_vz}.

\subsection{Proof of Theorem \ref{t:global}}
\label{ss:proof_global}

\begin{proof}[Proof of Theorem \ref{t:global}]
Let $(v,\tau)$ be the $(p,\a_{\crit},\s,q)$-solution to \eqref{eq:primitive} provided by Theorem \ref{t:local}. 
Recall that $\tau>0$ a.s. Let $\pz\in (2,\infty)$ be such that $\frac{1}{\pz}+\frac{1}{q}+\frac{\s}{2}=1$. Then, for all $0<s<T<\infty$,
\begin{align*}
\P(s<\tau<T)\stackrel{(i)}{=}\P\Big(s<\tau<T,\ \int_{s}^{\tau\wedge T} \|v(t)\|^{\pz}_{H^{2-\s,(q,2)}}\,\dd t <\infty \Big)\stackrel{(ii)}{=}0
\end{align*}
where in $(i)$ we used the energy inequality \eqref{eq:energy_estimate1} of Lemma \ref{l:energy_estimate} and in $(ii)$ the blow-up criterion of Theorem \ref{t:serrin}\eqref{it:serrin1} with $(p_0,\s_0,q_0)=(\pz,\s,q)$.

By letting $T\uparrow \infty$, we obtain $\P(s<\tau<\infty)=0$ for all $s>0$. Since $\tau>0$ a.s.\ by Theorem \ref{t:local}, sending $s\downarrow 0$ we get $\P(\tau<\infty)=0$ as desired.
\end{proof}

\subsection{Proof of Lemma \ref{l:strong_energy_estimate}}
\label{ss:proof_main_estimate_conclusion}
The proof of Lemma \ref{l:strong_energy_estimate} is based on the following intermediate estimate which will be proven in Section \ref{s:energy_estimate_proof}, see Lemma \ref{l:main_intermediate_estimate_wt} there. 

\begin{lemma}[Intermediate estimate]
\label{l:main_intermediate_estimate_vz}
Let the assumptions of Lemma \ref{l:strong_energy_estimate} be satisfied. 
Let $(v,\tau)$ be as in Lemma \ref{l:strong_energy_estimate} and fix $T\in (0,\infty)$. 
Then 
\begin{equation}
\label{eq:regularity_vz_L2}
\partial_z v \in L^2_{\loc}(0,\tau;H^1) \text{ a.s.\ }
\end{equation}
and there exists a constant $c_0>0$ such that, for all  $j\in \{0,1\}$, $\lambda>1$, $s\in (0,T]$ and $ \Gamma \in \F_s $ such that $\tau>s$ a.s.\ on $\Gamma$, it holds that 
\begin{align}
\label{eq:energy_estimate_vz}
\P\Big(\Gamma\cap \Big\{ \sup_{t\in [s,\tau\wedge T)}\|\partial_z^j v(t)\|^2_{L^2}+ \int_s^{\tau\wedge T}\|\partial_z^j v(t)\|^2_{H^1}\,\dd t \geq \lambda \Big\}\Big)&\\
\nonumber
\leq  \frac{c_0}{\log(\lambda)}\big(1+
\E\big[\one_{\Gamma}\| v(s)\|_{H^1}^{4}\big]\big)&.
\end{align}
\end{lemma}

Note that \eqref{eq:regularity_vz_L2} and \eqref{eq:regularity_C1} ensure that the LHS\eqref{eq:energy_estimate_vz} and RHS\eqref{eq:energy_estimate_vz} are well-defined, respectively.
In the case $j=0$ the estimate \eqref{eq:energy_estimate_vz} can be improved, see Lemma \ref{l:energy_estimate} below.

\begin{proof}[Proof of Lemma \ref{l:strong_energy_estimate}]
Let $\delta\in [0,\frac{1}{2})$ and $q\in (\frac{2}{2-\s},\frac{2}{1-\s})$ be as in Theorem \ref{t:global}. 
Below we focus on the case $\s<0$ as the case $\s=0$ follows as in the proof of \cite[Proposition 5.1]{Primitive1}.
Let
$p_0	$ be defined via the relation 
\begin{equation}
\label{eq:choose_p0_global}
\frac{1}{\pz}+\frac{1}{q}+\frac{\delta}{2}=1.
\end{equation}
As above, $\pz\in (2,\infty)$ due to $q\in (\frac{2}{2-\s},\frac{2}{1-\s})$.
Next let us fix $0<s<T<\infty$. Without loss of generality, we may assume that $\Gamma$ satisfies
$$
\E\big[\one_{\Gamma} \|v(s)\|_{B^{2/q}_{(q,2),\pz}}^{\pz}\big]+
\E\big[\one_{\Gamma} \|v(s)\|_{H^1}^{4}\big]<\infty.
$$
In particular, $\|v(s)\|_{B^{2/q}_{(q,2),\pz}}<\infty$ a.s.\ on $\Gamma$. 
To proceed further we need a localization argument. More precisely, for all $j\geq 1$ define the stopping time
\begin{equation}
\label{eq:localization_strong_estimate}
\tau_j\stackrel{{\rm def}}{=}
\left\{
\begin{aligned}
&\inf\Big\{t\in [s,\tau)\,:\, 
\big(\|v(t)\|_{B^{2/q}_{(q,2),\pz}}+ 
\|v\|_{L^{p_0}(s,t;H^{2-\s,(q,2)}}\big)\geq j
\Big\} &  &\text{on }\Gamma,\\
&s & &\text{on }\O\setminus\Gamma.
\end{aligned}
\right.
\end{equation}

Next, we conclude the proof of Lemma \ref{l:strong_energy_estimate} by using the stochastic Grownall lemma of \cite[Lemma A.1]{AV_variational}. To this end, let $\eta,\xi$ be two stopping time such that $0\leq \eta\leq \xi\leq \tau_j$ a.s.\ for some $j\geq 1$. 
By the stochastic maximal $L^{p}$-regularity estimate of Lemma \ref{l:smr} and \cite[Proposition 3.12]{AV19_QSEE_1}  (here we are also using that $\a_{\crit}=p_0(1-\frac{1}{q}-\frac{\s}{2})=0$ by \eqref{eq:choose_p0_global}) 
ensure the existence of a constant $C>0$ independent of $(\eta,\xi,j,\Gamma,v_0)$ such that 
\begin{align}
\label{eq:proof_global_smr_gronwall}
&\E\sup_{t\in [\eta,\xi]}\one_{\Gamma}\|v(t)\|_{B^{2/q}_{(q,2),\pz}}^{\pz}
+\E\int_{\eta}^{\xi}\one_{\Gamma}\|v\|_{H^{2-\s,(q,2)}}^{\pz}\,\dd t\\
\nonumber
&\leq C\Big(1+\E\big[\one_{\Gamma} \|v(s)\|_{B^{2/q}_{(q,2),\pz}}^{\pz}\big]\Big)
\\
\nonumber
&+ C \E\int_{\eta}^{\xi} \one_{\Gamma}\underbrace{\big(\|\ftg(\cdot,v)\|_{H^{-\delta,(q,2)}}^{\pz}
+ \|f(\cdot,v,\nabla v)\|_{H^{-\delta,(q,2)}}^{\pz}
+ \|g(\cdot,v)\|_{H^{1-\delta,(q,2)}(\ell^2)}^{\pz}\big)}_{I_1\stackrel{{\rm def}}{=}}\,\dd t  \\
\nonumber
&+ C \E\int_{\eta}^{\xi}\one_{\Gamma}\underbrace{
\big(\|(v\cdot\nabla_{x,y})v\|_{H^{-\delta,(q,2)}}^{\pz}
+\|w(v)\partial_z v\|_{H^{-\delta,(q,2)}}^{\pz}\big)}_{I_2\stackrel{{\rm def}}{=}}\,\dd t.
\end{align}
Arguing as in Substep 1c and Step 2 of Lemma \ref{l:nonlinearities}, by interpolation, Young's inequality and the fact that 
$\s<1$ we have, for some constant $C_1,C_2$ independent of $(\eta,\xi,j,\Gamma,v_0)$ and a.e.\ on $(0,\tau)\times \O$,
\begin{equation}
\label{eq:estimate_I1_global}
I_1\leq C_1( 1+ \|v\|_{H^{1,(q,2)}}^{\pz} )\leq \frac{1}{4C} \|v\|_{H^{2-\delta,(q,2)}}^{\pz}
+
 C_2 (1+\|v\|_{H^{-\delta,(q,2)}}^{\pz}).
\end{equation}
The estimate for $I_2$ is slightly more involved. Indeed, the estimates in Substeps 1a-1b of Lemma \ref{l:nonlinearities} yield, a.e.\ on $(0,\tau)\times \O$, 
\begin{align}
\label{eq:nonlinearity_proof_global}
I_2
 \lesssim \|v\|_{H^{1,2r}_{x,y} (L_z^2)}^{\pz}\| v\|_{L^{2r}_{x,y}(H^1_z)}^{\pz},
\end{align}
where $r\stackrel{{\rm def}}{=} \frac{2q}{\s q+2}$ (as we assumed $\s>0$). Note that $r>1$ as $q>\frac{2}{2-\s}$, cf.\ \eqref{eq:assumptions_critical_setting}.
From the definition of anisotropic Besov spaces via interpolation (see Subsection \ref{ss:notation}), we obtain
\begin{align}
\label{eq:interpolation_B_H}
&(B^{2/q}_{(q,2),\pz},H^{2-\s,(q,2)})_{1/2,1}\\
\nonumber
&\qquad \stackrel{(i)}{=}
(L^{(q,2)},H^{2-\s,(q,2)})_{\vartheta,1}\qquad \big[\text{with $\vartheta=(\tfrac{1}{q}+1-\tfrac{\s}{2})/(2-\s)\in (0,1)$}\big]\\ 
\nonumber
&\qquad \stackrel{(ii)}{\embed}
[L^{(q,2)},H^{2-\s,(q,2)}]_{\vartheta}\\
\nonumber
&\qquad \stackrel{(iii)}{= }H^{1/q+1-\s/2,(q,2)}\\
\nonumber
&\qquad \stackrel{(iv)}{\embed} H^{1/q+1-\s/2,q}_{x,y}(L^2_z)
\stackrel{(v)}{\embed} H^{1,2r}_{x,y}(L^2_z)
\end{align}
where in $(i)$ we used the reiteration theorem for real interpolation \cite[Theorem 3.5.3]{BeLo}, in $(ii)$ \cite[Theorem 4.7.1]{BeLo},  in $(iii)$ Lemma \ref{l:anisotropic_B}\eqref{it:anisotropic_B2}, and in $(iv)$ Lemma \ref{l:iterated_H}. Finally, $(v)$ follows from the Sobolev embeddings as 
$$
-\frac{1}{q}+1-\frac{\s}{2}=1-\frac{2}{2r}.
$$
The chain of embeddings \eqref{eq:interpolation_B_H} in particular yield 
$$
\|\phi\|_{H^{1,2r}_{x,y}(L^2_z)}\lesssim \|\phi\|_{B^{2/q}_{(q,2),\pz}}^{1/2}\|\phi\|_{H^{2-\s,(q,2)}}^{1/2}
$$ 
whenever the RHS is finite. Combining the above with \eqref{eq:nonlinearity_proof_global}, we obtain
\begin{equation}
\label{eq:estimate_I2_global}
I_2\leq
 C_{3} \|v\|_{B^{2/q}_{(q,2),\pz}}^{\pz}\|v \|_{L^{2r}_{x,y}(H^1_z)}^{2\pz}+ \frac{1}{4C} \|v\|_{H^{2-\s,(q,2)}}^{\pz},
\end{equation}
where $C$ is as in \eqref{eq:proof_global_smr_gronwall} and $C_3$ is independent of $(\eta,\xi,j,\Gamma,v_0)$.

Using the estimates \eqref{eq:estimate_I1_global} and \eqref{eq:estimate_I2_global} in \eqref{eq:proof_global_smr_gronwall}, for some constant $K=K(C,C_1,\dots,C_3)$,
\begin{align}
\label{eq:grownall_strong_application}
&\E\sup_{t\in [\eta,\xi]}\one_{\Gamma}\|v(t)\|_{B^{2/q}_{(q,2),\pz}}^{\pz}
+\E\int_{\eta}^{\xi}\one_{\Gamma} \|v\|_{H^{2-\s,(q,2)}}^{\pz}\,\dd t\\
\nonumber
&\quad\quad 
\leq K\Big(1+\E\big[\one_{\Gamma} \|v(s)\|_{B^{2/q}_{(q,2),\pz}}^{\pz}\big]\Big) \\
\nonumber
 &\quad\quad
 + K \E\int_{\eta}^{\xi} \one_{\Gamma}\Big[\big(1+  \| v \|_{L^{2r}_{x,y}(H^1_z)}^{2\pz}\big)
 \big(1+ \|v\|_{B^{2/q}_{(q,2),\pz}}^{\pz}\big)\Big]\,\dd t.
\end{align}

We claim that there exists a constant $K_0>0$  independent of $(\eta,\xi,j,\Gamma,v_0)$ such that, for all $\lambda>e$,
\begin{equation}
\label{eq:partialz_v_anisotropic_space}
\P\Big(\Gamma\cap \Big\{\int_s^{\tau}\| v(t) \|^{2\pz}_{L^{2r}_{x,y}(H^1_z)}\,\dd t\geq \lambda \Big\}\Big)
\leq  \frac{K_0}{\log(\lambda)} \Big(1+
\E\big[\one_{\Gamma} \|v(s)\|_{H^1}^{4}\big]\Big).
\end{equation}
Next we first show how \eqref{eq:partialz_v_anisotropic_space} implies \eqref{eq:energy_estimate2_vz}. The proof of \eqref{eq:partialz_v_anisotropic_space} is given at the end of this proof.
The stochastic Gronwall lemma of \cite[Lemma A.1]{AV_variational} applied to \eqref{eq:grownall_strong_application} (with stopping times $s\leq \eta\leq \xi\leq \tau_j$ a.s.\ for some $j\geq 1$) yield, for all $R,\lambda$ large,
\begin{align*}
&\P\Big(
\Gamma\cap \Big\{
\sup_{r\in [s,\tau_j\wedge T)}\|v(t)\|_{B^{2/q}_{(q,2),\pz}}^{\pz}
+\int_{s}^{\tau_j\wedge T} \|v\|^{\pz}_{H^{2-\s,(q,2)}}\,\dd t \geq \lambda \Big\}
\Big)\\
&\ \ \leq 8K \frac{e^{8K R} }{\lambda}   \Big(1+\E\big[\one_{\Gamma} \|v(s)\|_{B^{2/q}_{(q,2),\pz}}^{\pz}\big]\Big)
+
\P\Big(\Gamma\cap \Big\{\int_{s}^{\tau_j}\| v \|_{L^{2r}_{x,y}(H^1_z)}^{2\pz}\,\dd t\geq \frac{R}{2K} \Big\}\Big)\\
&\ \ \lesssim_{K_0,K}  
\Big(\frac{e^{8K R} }{\lambda}   
+ \frac{1}{\log R}\Big)\Big(1+\E\big[\one_{\Gamma} \|v(s)\|_{B^{2/q}_{(q,2),\pz}}^{\pz}\big]+
\E\big[\one_{\Gamma} \|v(s)\|_{H^1}^{4}\big]\Big).
\end{align*}
where in the last step we used \eqref{eq:partialz_v_anisotropic_space} and $\tau_j\leq \tau$. Since the constants on the RHS of the above inequality are independent of $j\geq 1$, we may let $j\uparrow \infty$ and obtain the above estimate with $\tau_j$ replaced by $\tau$. Hence, choosing $R=\frac{1}{8K}\log(\frac{\lambda}{\log(\lambda)})$ for $\lambda$ large, one obtains \eqref{eq:energy_estimate2_vz}.

To conclude the proof it remains to show \eqref{eq:partialz_v_anisotropic_space}. To this end, we employ Lemma \ref{l:main_intermediate_estimate_vz}. To begin, note that, for all $t\in (0,\infty)$,
\begin{align}
\label{eq:embedding_global_proof_anisotropic}
L^{\infty}(0,t;L^2)\cap L^2(0,t;H^1)
&\stackrel{(i)}{\embed} L^{2\pz}(0,t;H^{1/\pz})\\
\nonumber
&\stackrel{(ii)}{\embed} L^{2\pz}(0,t;L^{(2r,2)})=L^{2\pz}(0,t;L^{2r}_{x,y}(L^2_z))
\end{align}
where in $(i)$ follows from the H\"older inequality, and $(ii)$ from Lemma \ref{l:anisotropic_B}\eqref{it:anisotropic_B4} and
$$
\frac{1}{\pz}-1\stackrel{\eqref{eq:choose_p0_global}}{=}-\frac{\s}{2}-\frac{1}{q}=-\frac{1}{r}, \quad \text{ as }\quad 
r=\frac{2q}{\s q+2}.
$$
Hence \eqref{eq:partialz_v_anisotropic_space} follows from the embedding \eqref{eq:embedding_global_proof_anisotropic} and Lemma \ref{l:main_intermediate_estimate_vz}. The embedding constant in \eqref{eq:embedding_global_proof_anisotropic} does not depend on $t>0$ as $(i)$ follows from the H\"older inequality. Hence, as commented below \eqref{eq:partialz_v_anisotropic_space}, this completes the proof of 
 Lemma \ref{l:strong_energy_estimate}.
\end{proof}

\begin{remark}[Necessity of the  $\x$-anisotropy for global well-posedness]
\label{r:necessity_anisotropic}
Similar to Remark \ref{r:regularity_necessary_anisotropic}, the central role of the $\x$-isotropic function spaces used in Theorem \ref{t:local} (and therefore in Lemma  \ref{l:strong_energy_estimate}) 
is visible from the last two embeddings of \eqref{eq:embedding_global_proof_anisotropic}.
Indeed, the embedding $
L^{\infty}(0,t;L^2)\cap L^2(0,t;H^1)\embed L^{2\pz}(0,t;L^{2r}_{x,y}(L^{\zeta}_z))$ holds if and only if $\zeta\leq 2$.
\end{remark}

\subsection{Proof of Theorem \ref{t:continuous_dependence}}
\label{ss:proof_continuous_dependence}
To prove Theorem \ref{t:continuous_dependence}, as in \cite[Section 6]{AV23_reactionII}, we combine the energy estimate \eqref{eq:energy_estimate2_vz} at time $s>0$ and the local continuity w.r.t.\ the initial data of Proposition \ref{prop:continuity}.

The following is a version of \cite[Lemma 6.7]{AV23_reactionII} adapted to our situation.

\begin{lemma}[Stability estimate]
\label{l:estimate_difference}
Let the assumptions of Theorem \ref{t:continuous_dependence} be satisfied and fix $0<s<T<\infty$. 
For $v_0,v_0'\in L^p(\O;\Bs^{2/q}_{(q,2),p})$, let $v$ and $v'$ be the global $(p,\a_{\crit},q,\s)$-solution provided by Theorem \ref{t:global} with initial data $v_0$ and $v_0'$, respectively. 
Then there exist a constant $C_s$ and a map $\psi_s:\R_+\to \R_+$ satisfying $\displaystyle{\lim_{L\to \infty}\psi_s (L)=0}$, both independent of $v_0,v_0'$ such that, for all $\Gamma\in \F_s$ for which $\tau>s$ a.s.\ on $\Gamma$ and $\ell,L\geq 1$,
\begin{align*}
&\P\Big(\Gamma \cap \Big\{\sup_{t\in [s,T]}\|v(t)-v'(t)\|_{B^{2/q}_{(q,2),p_0}}^{p_0}\geq \ell  \Big\}\Big) \\
&\leq C_s\frac{e^{C_s L}}{\ell}\E\big[\one_{\Gamma}\|v(s)-v'(s)\|_{B^{2/q}_{(q,2),p_0}}^{p_0}\big]
+\psi_s (L)\big[1+ \norm_{v,v'}(s)\big],
\end{align*}
where $p_0\in (2,\infty)$ satisfies $\frac{1}{p_0}+\frac{1}{q}+\frac{\delta}{2}=1$ and 
\begin{equation}
\label{eq:def_N_mathcal_proof_continuity}
\norm_{v,v'}(s) \stackrel{{\rm def}}{=}\max_{\phi\in \{v,v'\}}
\E\Big[\one_{\Gamma}\big( \|\phi(s)\|_{B^{2/q}_{(q,2),p_0}}^{p_0}+\|\phi(s)\|_{H^1}^{4}\big)\Big].
\end{equation}
\end{lemma}

\begin{proof}
To short-hand the formulas below, we use the notation introduced in \eqref{eq:choice_X0X1}-\eqref{eq:def_ABFG}. In particular, the stochastic PEs \eqref{eq:primitive} are seen as a stochastic evolution equation of the form \eqref{eq:SEE} on the Banach space $\Hs^{-\s,(q,2)}$.
Without loss of generality we assume that $\Gamma$ is such that $\E\big[\one_{\Gamma}\|v(s)-v'(s)\|_{B^{2/q}_{(q,2),p_0}}^{p_0}\big]<\infty$ and $\norm_{v,v'}(s)<\infty$.

The process $V\stackrel{{\rm def}}{=}\one_{\Gamma}( v-v')$ satisfies, on $[s,\infty)$,
\begin{equation}
\label{eq:SEE_V}
\left\{
\begin{aligned}
\dd V &+ A(t) V\, \dd t = \one_{\Gamma}( F(t,
\one_{\Gamma}v)-F(t,
\one_{\Gamma}v'))\,\dd t\\
& + (B(t)V+\one_{\Gamma}[G(t,
\one_{\Gamma}v)-G(t,
\one_{\Gamma}v')])\,\dd \Br_{\ell^2}(t),\quad t\in\R_+,\\
 V(s)&=\one_{\Gamma}(v(s)-v'(s)).
\end{aligned}
\right.
\end{equation}
where we also used that
$
\one_{\Gamma}F(\cdot,\phi)= \one_{\Gamma} F(\cdot,\one_{\Gamma}\phi)
$ and $
\one_{\Gamma}G(\cdot,\phi)= \one_{\Gamma} G(\cdot,\one_{\Gamma}\phi)$ a.e.\ on $\R_+\times \O$ for $\phi\in \{v,v'\}$. 

As for Lemma \ref{l:strong_energy_estimate}, we prove the claim by applying the stochastic Gronwall lemma of \cite[Lemma A.1]{AV_variational}. To this end, for all $j\geq 1$, we set 
$$
\tau_j\stackrel{{\rm def}}{=}
\inf\Big\{t\in [s,T]\,:\, 
\max_{\phi\in \{v,v'\}}\big(\|\phi(t)\|_{B^{2/q_0}_{(q_0,2),p_0}}+ 
\|\phi\|_{L^{p_0}(s,t;B^{2/q_0}_{(q_0,2),p_0})}\big)\geq j
\Big\} .
$$
Now let $\eta,\xi$ be stopping times satisfying $s\leq \eta\leq \xi\leq \tau_j$ a.s.\ for some $j\geq 1$. Arguing as in the proof of Lemma \ref{l:strong_energy_estimate}, by the stochastic Gronwall lemma \cite[Lemma A.1]{AV_variational}, one can readily check that the claim of Lemma \ref{l:estimate_difference} follows from the following assertion: 

There exist constants $K_s,\lambda_s>0$ independent of $(j,\eta,\xi,\Gamma)$ such that 
\begin{align}
\label{eq:grownall_V_continuity}
\E\sup_{t\in [s,T]}\one_{\Gamma} \|V(t)\|_{B^{2/q}_{(q,2),p_0}}^{p_0}
&+ \E\int_{s}^T \one_{\Gamma} \|V(t)\|_{H^{2-\s,(q,2)}}^{p_0}\,\dd t\\
\nonumber
& \leq K_s \E\big[\one_{\Gamma}\|v^{(1)}(s)-v^{(2)}(s)\|_{B^{2/q}_{(q,2),p_0}}^{p_0}\big]\\
\nonumber
&+ K_s \E\int_{\eta}^{\xi}\one_{\Gamma} \big(1+\JJ_{v,v'}(t)\big) \|V(t)\|_{B^{2/q}_{(q,2),p_0}}^{p_0}\,\dd t
\end{align}
where $\JJ_{v,v'}$ is a random variable satisfying with paths in $L^1(s,T)$ and
\begin{equation}
\label{eq:tail_estimate_J_gronwall_continuity}
\P\Big(\int_s^{T}\JJ_{v,v'}(t)\, \dd t \geq \lambda\Big)\leq  \frac{K_s}{\log(\log(\lambda))}\norm_{v,v'}(s) \ \ \text{ for all }\lambda \geq \lambda_s.
\end{equation}

\emph{Step 1: Proof of \eqref{eq:grownall_V_continuity} with}
\begin{equation}
\label{eq:def_J_gronwlall_continuity}
\JJ_{v,v'} (t)\stackrel{{\rm def}}{=}\one_{\Gamma}\big(\|v(t)\|^{2p_0}_{H^{1/q+1-\s/2,(q,2)}}
+\|v'(t)\|^{2p_0}_{H^{1/q+1-\s/2,(q,2)}}\big).
\end{equation}
As in \eqref{eq:choice_X0X1} in this step we let $X_{\vartheta}=\Hs^{2-\s+2\vartheta,(q,2)}$ for $\vartheta\in (0,1)$. 
Let $\beta=\frac{1}{2}+\frac{\delta}{4}+\frac{1}{2q}$ be as in Lemma \ref{l:nonlinearities} (recall that the condition $q<\frac{2}{1-\s}\leq \frac{2}{\s}$ as \eqref{eq:assumptions_critical_setting} is in force). By $(i)$-$(ii)$ in \eqref{eq:interpolation_B_H}  and $2-\s+2\beta=1/q+1-\s/2$, we obtain 
\begin{equation}
\label{eq:interpolation_estimate_phi_continuity}
\|\phi\|_{X_{\beta}}=\|\phi\|_{H^{1/q+1-\s/2,(q,2)}}\lesssim \|\phi\|_{B^{2/q}_{(q,2),p_0}}^{1/2}\|\phi\|_{H^{2-\s,(q,2)}}^{1/2}.
\end{equation}
Hence, by Lemma \ref{l:nonlinearities} we have, a.e.\ on $\R_+\times \O$,
\begin{align*}
&\|F(\cdot,\one_{\Gamma}v)-F(\cdot,\one_{\Gamma}v')\|_{X_0}
+\|G(\cdot,\one_{\Gamma}v)-G(\cdot,\one_{\Gamma}v')\|_{\g(\ell^2;X_{1/2})}\\
&\lesssim  (1+\one_{\Gamma}\|v\|_{X_{\beta}}+\one_{\Gamma}\|v'\|_{X_{\beta}})\|V\|_{X_{\beta}}\\
&\leq C_{\eta} (1+\one_{\Gamma}\|v\|_{X_{\beta}}^2+\one_{\Gamma}\|v'\|_{X_{\beta}}^2)\|V\|_{B^{2/q}_{(q,2),p_0}}+ \eta\|V\|_{X_1},
\end{align*}
where $\eta>0$ is arbitrary.

Now the claim follows by arguing as below \eqref{eq:proof_global_smr_gronwall}, by applying the stochastic maximal $L^{p}$-regularity estimate of Lemma \ref{l:smr} and \cite[Proposition 3.12]{AV19_QSEE_1} with weight $\a=0$ to \eqref{eq:SEE_V} and absorbing the term $\eta\E\int_{\eta}^{\xi}\one_{\Gamma}\|V(t)\|_{X_1}^{p_0}\,\dd t$ by choosing $\eta>0$ sufficiently small (depending only on the constant in the estimate of Lemma \ref{l:smr}).

\emph{Step 2: Proof of \eqref{eq:tail_estimate_J_gronwall_continuity} with $\JJ_{v,v'}$ as in \eqref{eq:def_J_gronwlall_continuity}}.
By \eqref{eq:interpolation_estimate_phi_continuity} and the H\"older inequality, 
$$
L^{\infty}(0,t;B^{2/q}_{(q,2),p_0})\cap
L^{p_0}(0,t;H^{2-\s,(q,2)})\embed 
L^{2p_0}(0,t;H^{1/q+1-\s/2,(q,2)})\ \text{ for }\ t>0
$$
with embedding constant independent of $t$.
Hence \eqref{eq:tail_estimate_J_gronwall_continuity} follows from the above estimate and \eqref{eq:energy_estimate2_vz} in Lemma \ref{l:strong_energy_estimate}.
\end{proof}

Now we are ready to prove Theorem \ref{t:continuous_dependence}. The strategy is roughly as follows. Firstly, we decompose $[0,T]$ into $[0,t]$ and $[t,T]$ where $t>0$ is small. In the first interval, we use Proposition \ref{prop:continuity}, which already gives continuity in $[0,t]$. Secondly, we apply Lemma \ref{l:estimate_difference} to obtain an estimate depending only on differences of solutions at time $t\sim 0$. At this stage, one is tempted to use again  Proposition \ref{prop:continuity} since $t\sim 0$. However, the RHS\eqref{eq:def_N_mathcal_proof_continuity} contains norms which are (typically) \emph{stronger} than the one estimated in Proposition \ref{prop:continuity} (e.g.\ in case $\a_{\crit}=0$). To circumvent this problem, we use the extra flexibility coming from the \emph{weighted} setting and the instantaneous regularization of weighted functions (cf.\ \cite[Theorem 1.2]{ALV23} and the text below Proposition \ref{prop:continuity}).

\begin{proof}[Proof of Theorem \ref{t:continuous_dependence}]
The proof is similar to the one of \cite[Theorem 6.2]{AV23_reactionII}. 
Let us fix a sequence $(v_{0,n})_{n\geq 1}$ such that $v_{0,n}\to v_0$ in probability in $\Bs^{2/q}_{(q,2),p}$.
Arguing as in \cite[Theorem 6.2, Step 1]{AV23_reactionII}, it is enough to consider a sequence such that $\sup_{n\geq 1}\|v_{0,n}\|_{B^{2/q}_{(q,2),p}}\leq R$ a.s.\ for some deterministic constant $R>0$.

Next we collect some useful facts. Firstly, by Proposition \ref{prop:continuity} (see also the comments below it) and the Chebyshev inequality, there exist constants $T_0,K_0,N_0>0$ and stopping times $\tau_0,\tau_1\in (0,\tau]$ a.s.\ for which the following holds: 
For all $t\in (0,T_0]$ and $n\geq N_0$ there exists a stopping time $\tau^{(n)}_0$ such that, for all $\ell>0$,
\begin{align}
\label{eq:continuity1_proof}
\P\Big(\sup_{r\in [0,t]}\|v(r)-v^{(n)}(r)\|_{B^{2/q}_{(q,2),p}}\geq \ell,\ \tau_0\wedge \tau_0^{(n)}\geq t\Big)\leq 
\frac{K_0}{\ell^p}\E\|v_0-v_0^{(n)}\|_{B^{2/q}_{(q,2),p}}^p,&\\
\label{eq:continuity2_proof}
\P(\tau_0\wedge \tau_0^{(n)}\leq t)\leq K_0 \big[\E\|v_0-v_0^{(n)}\|_{B^{2/q}_{(q,2),p}}^p+\P(\tau_1\leq t)\big].&
\end{align}

Secondly, we exploit the weighted setting and the instantaneous regularization of weighted function spaces see e.g.\ \cite[Theorem 1.2]{ALV23}. To this end, let $p_1\in (p\vee 4,\infty)$ so large that 
\begin{equation}
\label{eq:embedding_r_large}
B^{2-\s-2/p_1}_{(q,2),p_1}\embed B^{2/q}_{(q,2),p_0} \quad \text{ and }\quad
B^{2-\s-2/p_1}_{(q,2),p_1}\embed H^1,
\end{equation}
where $p_0\in (2,\infty)$ satisfies $ \frac{1}{p_0}+\frac{1}{q}+\frac{\delta}{2}=1$ (and hence $p_0\leq p$). Since $q\geq 2$ and  $\s<1$, the choice of $p_1$ for which \eqref{eq:embedding_r_large} holds is a consequence of  \eqref{eq:elementary_emb0}-\eqref{eq:elementary_emb1}. 

Since $B^{2/q}_{(q,2),p}\embed B^{2/q}_{(q,2),p_1}$ by \eqref{eq:elementary_emb0} as $p_1\geq p$,  we can apply Theorem \ref{t:local} and Proposition \ref{prop:continuity} with $p$ replaced by $p_1$. Hence, the estimates \eqref{eq:continuity0}-\eqref{eq:continuity2} (see also the comments below Proposition \ref{prop:continuity}) with $p$ replaced by $p_1$ imply the existence of constants $K_1,T_1,N_1>0$ and stopping times $\lambda_0,\lambda_1\in (0,\tau]$ a.s.\ for which the following is satisfied:

For all $t\in (0,T_1]$ and $n\geq N_1$ there exists a stopping time $\lambda^{(n)}_0$ such that, for all $\ell> 0$,
\begin{align}
\label{eq:continuity1_proof2}
\E\Big[\one_{\{\lambda_0>t\}}\sup_{r\in [0,t]}\big(r^{\alpha}\|v(r)\|_{B^{2-\s-2/p_1}_{(q,2),p_1}}^{p_1}\big)\Big]
\leq 
K_1,&\\
\label{eq:continuity2_proof2}
\E \Big[ \one_{\{\lambda_0\wedge \lambda_0^{(n)}>t\}}
\sup_{r\in [0,t]}\big(r^{\alpha}\|v(r)-v^{(n)}(r)\|_{B^{2-\s-2/p_1}_{(q,2),p_1}}^{p_1}\big) 
\Big]\qquad \qquad\qquad&\\
\nonumber
\leq 
K_1\E\|v_0-v_0^{(n)}\|_{B^{2/q}_{(q,2),p_1}}^{p_1},&\\
\label{eq:continuity3_proof2}
\P(\lambda_0\wedge \lambda_0^{(n)}\leq t)\leq K_1 \big[\E\|v_0-v_0^{(n)}\|_{B^{2/q}_{(q,2),p_1}}^{p_1}+\P(\lambda_1\leq t)\big],&
\end{align}
where $\alpha\stackrel{{\rm def}}{=} p_1(1-\frac{1}{q}-\frac{\delta}{2})-1>0$ as $p_1>p\geq p_0$ and $ \frac{1}{p_0}+\frac{1}{q}+\frac{\delta}{2}=1$.

Let us remark that we also used Corollary \ref{cor:compatibility} which ensures that the solutions provided by Theorem \ref{t:global} with $p$ replaced by $p_1$ are the same as the original one.

Now we may turn to the main proof. Set
$$
\mu_0\stackrel{{\rm def}}{=} \tau_0\wedge \lambda_0 \quad \text{ and }\quad 
\mu_0^{(n)}\stackrel{{\rm def}}{=} \tau_0^{(n)}  \ \text{ for  }\ n\geq 1. 
$$
Using that $B^{2/q}_{(q,2),p_0}\embed B^{2/q}_{(q,2),p}$ contractively as $p_0\leq p$  we have, for all $t\in (0,T_0\wedge T_1]$, $n\geq N_0\vee N_1$ and $\ell>0$,
\begin{align}
\label{eq:main_estimate_continuity}
&\P\Big(\sup_{r\in [0,T]}\|v(r)-v^{(n)}(r)\|_{B^{2/q}_{(q,2),p}}\geq \ell\Big)\\
\nonumber
&\qquad\quad\leq \P(\tau_0\wedge \tau_0^{(n)}\leq t)+\P(\lambda_0\wedge \lambda_0^{(n)}\leq t)\\
\nonumber
&\qquad\quad+\P\Big(\sup_{r\in [0,t]}\|v(r)-v^{(n)}(r)\|_{B^{2/q}_{(q,2),p}}\geq \ell,\ \mu_0\wedge \mu_0^{(n)}> t\Big)\\
\nonumber
&\qquad\quad+\P\Big(\sup_{r\in [t,T]}\|v(r)-v^{(n)}(r)\|_{B^{2/q}_{(q,2),p_0}}\geq \ell,\ \mu_0\wedge \mu_0^{(n)}> t\Big)\\
\nonumber
&\qquad\qquad\quad\quad\leq K_0 \P(\tau_1\leq t)+K_1\P(\lambda_1\leq t)\\
\nonumber
&\qquad\qquad\quad\quad+2 (1+\ell^{-p})K_0\E\|v_0-v_0^{(n)}\|_{B^{2/q}_{(q,2),p}}^{p}
+K_1\E\|v_0-v_0^{(n)}\|_{B^{2/q}_{(q,2),p_1}}^{p_1}\\
\nonumber
&\qquad\qquad\quad\quad+\underbrace{\P\Big(\sup_{r\in [t,T]}\|v(r)-v^{(n)}(r)\|_{B^{2/q}_{(q,2),p_0}}\geq \ell,\ \mu_0\wedge \mu_0^{(n)}> t\Big)}_{\mathcal{R}^{(n)}\stackrel{{\rm def}}{=}}
\end{align}
where in the last step we used \eqref{eq:continuity1_proof}-\eqref{eq:continuity2_proof} and \eqref{eq:continuity3_proof2}.

To estimate the remainder $\mathcal{R}^{(n)}$ we apply Lemma \ref{l:estimate_difference}. Indeed, from the latter result applied to $\Gamma= \{\mu_0\wedge \mu_0^{(n)}> t\}\in \F_s$ (note that $\tau>s $ a.s.\ on $\Gamma$ as $\mu_0\leq \tau_0\leq \tau$ a.s.), we obtain for all $L\geq 1$ (recall that $p_1\geq p_0$ and  $\norm_{v,v^{(n)}}$ is as in \eqref{eq:def_N_mathcal_proof_continuity})
\begin{align*}
\mathcal{R}^{(n)} 
&\leq C_s \frac{e^{C_s L}}{\ell^{p_0}}\E\big[\one_{\Gamma}\|v(t)-v^{(n)}(t)\|_{B^{2/q}_{(q,2),p_0}}^{p_0}\big]
+\psi_t (L)\big[1+ \norm_{v,v^{(n)}}(t)\big]\\
&\stackrel{\eqref{eq:embedding_r_large}}{\lesssim} C_s \frac{e^{C_s L}}{\ell^{p_0}}
\Big(\E[\one_{\Gamma}\|v(t)-v^{(n)}(t)\|_{B^{2-\s-2/p_1}_{(q,2),p_1}}^{p_1}]\Big)^{p_0/p_1}
+\psi_t (L)\big[1+\norm_{v,v^{(n)}}(t)\big]\\
&\stackrel{\eqref{eq:continuity2_proof2}}{\leq} C_s\frac{e^{C_s L} t^{-\alpha}}{\ell^{p_0}}
\Big(\E[\|v_0-v^{(n)}_0\|_{B^{2/q}_{(q,2),p_1}}^{p_1}]\Big)^{p_0/p_1}
+\psi_t (L)\big[1+ \norm_{v,v^{(n)}}(t)\big].
\end{align*}
Similarly, by using \eqref{eq:embedding_r_large}-\eqref{eq:continuity2_proof2}, $p\geq 4$ and the second embedding in \eqref{eq:embedding_r_large}, we have   
$
\norm_{v,v^{(n)}}(t)\lesssim  t^{-\alpha}.
$
Hence, using the previous estimates in \eqref{eq:main_estimate_continuity}, for all $\ell>0$,
$$
\limsup_{n\to \infty}
\P\Big(\sup_{r\in [0,T]}\|v(r)-v^{(n)}(r)\|_{B^{2/q}_{(q,2),p}}\geq \ell\Big)
\lesssim   \P(\tau_1\leq t)+\P(\lambda_1\leq t)+\psi_t (L) (1+t^{-\alpha}) .
$$
Letting $L\to \infty$ and using $\displaystyle{\lim_{L\to \infty}}\psi_t(L)\to 0$ (cf.\ Lemma \ref{l:estimate_difference}), the claim of Theorem \ref{t:continuous_dependence} then follows by taking $t\downarrow 0$ as $\tau_1\wedge \lambda_1>0$ a.s.\ by construction.
\end{proof}

\section{Energy estimates -- Proof of Lemma \ref{l:main_intermediate_estimate_vz}}
\label{s:energy_estimate_proof}
The aim of this section is to prove Lemma \ref{l:main_intermediate_estimate_vz}. Its proof is based on the well-known decomposition into baroclinic and barotropic modes (see e.g.\ \cite{CT07,HH20_fluids_pressure}). More precisely, the barotropic mode $\overline{v}$ and baroclinic mode $\wt{v}$ are given by:
\begin{equation}
\label{eq:def_ubar_utilde}
\overline{v}\stackrel{{\rm def}}{=} \int_{\T_z} v(\cdot,z)\,\dd z \qquad \text{ and }\qquad 
\wt{v}\stackrel{{\rm def}}{=}v-\overline{v}.
\end{equation}

\begin{lemma}[Main energy estimate]
\label{l:main_intermediate_estimate_wt}
Let the Assumptions of Theorem \ref{t:global} be satisfied. Let $(v,\tau)$ be the $(p,\a,\s,q)$-solution provided by Theorem \ref{t:local}. Then
\begin{equation}
\label{eq:regularity_of_partialz}
\partial_z v \in L^2_{\loc}(0,\tau;H^1)\ \text{ a.s.\ }
\end{equation}
Moreover, for all $T\in (0,\infty)$, 
there exists $c_0>0$ such that for all $\lambda>1$, $s\in (0,T]$ and $\Gamma\in \F_s$ such that, $\tau>s$ a.s.\ on $\Gamma$,
\begin{align}
\label{eq:energy_estimate2}
&\P\Big(\Gamma\cap \Big\{\sup_{t\in [s,\tau\wedge T)} \X_{t}+ \int_{s}^{ \tau\wedge T}\Y_t\,\dd t \geq \lambda\Big\}\Big)\\
\nonumber
&\leq \frac{c_0}{\log(\lambda)} \Big(1+
\max_{j\in \{0,1\}}\E\big[\one_{\Gamma} \|\partial_z^j v(s)\|_{L^2}^{2}\big]+\E\big[\one_{\Gamma} \|\wt{v}(s)\|_{L^4}^{4}\big]\Big)
\end{align}
where 
\begin{align*}
\X_t &\stackrel{{\rm def}}{=}\|\partial_z v(t)\|_{L^2}^2+\|\wt{v}(t)\|_{L^4}^4,\\
\Y_t  &\stackrel{{\rm def}}{=}
 \|\partial_z v(t)\|_{H^1}^2+\big\||\wt{v}(t)||\nabla \wt{v}(t)|\big\|_{L^2}^2.
\end{align*}
\end{lemma}

Note that \eqref{eq:regularity_of_partialz} it is used to make sense of $\int \Y_t\,\dd t$, since it does \emph{not} follow from the instantaneous regularization result of Theorem \ref{t:regularity1}. This is in contrast with the r.v.\ $\sup_t \X_t$ which is well-defined due to \eqref{eq:regularity_C1}.

Compared to Lemma \ref{l:main_intermediate_estimate_vz} we also estimate the baroclinic mode $\wt{v}$. As Step 1 in the proof of Lemma \ref{l:main_intermediate_estimate_wt} below shows, it seems not possible to derive a priori estimates for $\sup_t\|\partial_z v\|_{L^2_{\x}}$ and $\|\partial_z v\|_{L^2_t(H^1_{\x})}$ without proving $L^4$-type estimates for the baroclinic mode $\wt{v}$. 
Moreover, compared to usual strategy in the context of PEs (see e.g.\ \cite{CT07}, \cite[Subsection 1.4]{HH20_fluids_pressure}  or \cite[Subsection 5.2]{Primitive1}), we do not prove any estimate for the barotropic mode $\overline{v}$. The reason behind this is that the barotropic mode is typically estimated in strong norms (i.e.\ $\sup_t\|\overline{v}\|_{H^1_{\x}}$ and $\|\overline{ v}\|_{L^2_t(H^2_{\x})}$), which are not available due to the roughness of the noise (see the discussion around \eqref{eq:transport_noise_mapping}-\eqref{eq:transport_noise_mapping_2} in Subsection \ref{ss:global_intro}). Fortunately, as we have seen in Section \ref{s:global_proofs} such estimates are not needed to prove global well-posedness of the stochastic PEs \eqref{eq:primitive}.

The proofs of \eqref{eq:regularity_of_partialz} and \eqref{eq:energy_estimate2} are given in Subsections \ref{ss:intermezzo} and \ref{ss:main_intermediate_estimates}, respectively. 
Before proceeding further, we state the following $L^2$-energy estimate for solutions of the stochastic PEs \eqref{eq:primitive}. We do not provide a proof as it readily follows from the usual cancellation $\int_{\T^3} [(u\cdot\nabla)v]\cdot v\,\dd \x=0$ as $\nabla\cdot u=0$ (cf.\ 
\cite[Lemma 5.2]{Primitive1}). 

\begin{lemma}[$L^2$-energy estimate]
\label{l:energy_estimate}
Let the assumptions of Theorem \ref{t:local} be satisfied, and let $(v,\tau)$ be the $(p,\a_{\crit},\s,q)$-solution to \eqref{eq:primitive}. Then, for all $T\in (0,\infty)$ there exists $C>0$ such that, for all $s\in [0,T]$ and $\Gamma\in \F_s$ satisfying $\tau>s$ a.s.\ on $\Gamma$,
\begin{equation}
\label{eq:velocity_L2_balance_proof}
\E\Big[ \sup_{r\in [s,\tau\wedge T)} \one_{\Gamma} \|v(r)\|_{L^2}^2\Big] 
+\E\int_0^{\tau\wedge T}\one_{\Gamma} \|v(r)\|_{H^1}^2\,\dd r\leq C\big(1+\E[\one_{\Gamma}\|v(s)\|_{L^2}^2]\big).
\end{equation}
\end{lemma}

In the above we do not need Assumption \ref{ass:global}. 
Taking $s=0$ and $\Gamma=\O$, Lemma \ref{l:energy_estimate} coincides with the usual energy estimate for Navier-Stokes equations.

In light of Lemmas \ref{l:main_intermediate_estimate_wt} and \ref{l:energy_estimate}, we can prove Lemma \ref{l:main_intermediate_estimate_vz}.

\begin{proof}[Proof of Lemma \ref{l:main_intermediate_estimate_vz}]
The case $j=0$ of \eqref{eq:energy_estimate_vz} follows from Lemma \ref{l:energy_estimate} and the Chebyshev inequality. The case $j=1$ of \eqref{eq:energy_estimate_vz} is a consequence of Lemma \ref{l:main_intermediate_estimate_wt} and the embedding $H^1=H^1(\T^3)\embed L^4$.
\end{proof}

\subsection{Preparation to the proof of Lemma \ref{l:main_intermediate_estimate_wt}}
\label{ss:preparation_estimate}

Next we introduce a suitable localization for the $(p,\a_{\crit},\s,q)$-solution $(v,\tau)$ provided by Theorem \ref{t:local}. 
Throughout this section we fix $0<s<T<\infty$.
For each integer $j\geq 1$, we define the following stopping time
\begin{equation}
\label{eq:tau_j_localization_main_proof}
\tau_j \stackrel{{\rm def}}{=}
\left\{
\begin{aligned}
&\inf\{t\in [s,\tau)\,:\, \|v(t)-v(s)\|_{C^1(\T^3;\R^2)}\geq j\}\wedge T & &\text{ on }
\Gamma_j,\\
&s  & &\text{ otherwise.}
\end{aligned}
\right.
\end{equation}
where $\Gamma_j\stackrel{{\rm def}}{=} \Gamma\cap \big\{\tau>s,\, \|v(s)\|_{C^1}\leq j\big\}\in \F_s$.

By Theorem \ref{t:regularity1} it follows that
\begin{align}
\label{eq:approximation1}
\P(\Gamma\setminus ( \cup_{j}\Gamma_j))=1,&\\
\label{eq:approximation2}
\lim_{j\to \infty}\tau_j =\tau&\text{ a.s.\ on }\Gamma ,\\
\label{eq:approximation3}
\|(v,\nabla v)\|_{L^{\infty}([s,\tau_j )\times \T^3;\R^2\times \R^{2\times3 })}\lesssim j \  &\text{ a.s.\ for all }j\geq 1.
\end{align} 
Roughly speaking, the conditions \eqref{eq:approximation1}-\eqref{eq:approximation2} ensure that the stochastic intervals $([s,\tau_j)\times\Gamma_j)_{j\geq 1}$ covers the interval $[s,\tau)\times \Gamma$ up to a $\P\otimes \dd t$-null set. 

\subsection{Instantaneous regularization in the vertical direction} 
\label{ss:intermezzo}
The aim of this subsection is to prove \eqref{eq:regularity_of_partialz}. 
The latter result can be proven (even in an $L^p$-setting) by using the bootstrap results of  \cite[Section 6]{AV19_QSEE_2}. However, this requires working with spaces having anisotropic smoothness (more precisely, spaces of functions having more differentiability in the vertical direction than in the horizontal ones). Here we use a more standard argument using difference quotients. Let us stress that the latter strategy is possible as we have already bootstrapped some regularity in Theorem \ref{t:regularity1} using the results of \cite[Section 6]{AV19_QSEE_2}, which are essential to bootstrap regularity in the critical setting.

\begin{proof}[Proof of \eqref{eq:regularity_of_partialz} in Lemma \ref{l:main_intermediate_estimate_wt}]
Let $\tau_j$ be as in \eqref{eq:tau_j_localization_main_proof}. Below we prove that 
\begin{equation}
\label{eq:boundedness_L2_partialzv}
\E\int_s^{\tau_j} \one_{\Gamma} \|\nabla \partial_z v \|_{L^2}^2\,\dd t<\infty \ \ \text{ for all $s>0$ and $j\geq 1$}.
\end{equation}
This in particular implies 
 \eqref{eq:regularity_of_partialz}. 
  Here we prove \eqref{eq:boundedness_L2_partialzv} by using the difference quotient method and exploiting that $v|_{[s,\tau_j)\times \Gamma}$ is smooth, see \eqref{eq:approximation1}. 
For all $r\in[-\frac{1}{2},\frac{1}{2}]\setminus \{0\}$, consider the difference quotient operator in $z$, i.e.\ for all map $\phi:\T^3\to \R$,
$$
\diff_r \phi \stackrel{{\rm def}}{=} \frac{\tr_r \phi - \phi}{|r|} \quad \text{ where } \quad 
\tr_r \phi\stackrel{{\rm def}}{=}\phi(\cdot+r e_z), \ \ e_z=(0,0,1).
$$

Fix $s>0$ and $j\geq 1$. To prove \eqref{eq:boundedness_L2_partialzv}, 
it suffices to prove the existence of $C_0>0$ such that
\begin{equation}
\label{eq:bound_diff_uniformh}
\E \Big[\one_{\Gamma} 
\int_s^{\tau_j} \| \nabla \diff_r v(r)\|_{L^2}^2 \,\dd t\Big] \leq C_0 \ \text{ for all }\ r\in [-\tfrac{1}{2},\tfrac{1}{2}]\setminus \{0\}.
\end{equation}
 
\emph{Step 1: Sufficiency of \eqref{eq:bound_diff_uniformh}}. 
We begin this step with a preliminary observation. If $\phi\in L^2$ satisfies $C_{\phi}\stackrel{{\rm def}}{=}\liminf_{r\to \infty}\|\diff_r \phi\|_{L^2}<\infty$, then $\partial_z \phi\in L^2$ and $\|\partial_z \phi\|_{L^2}\leq C_{\phi}$. To see this, note that, for all $\psi\in \D(\T^3)$, 
$$
\Big|\int_{\T^3} \phi\partial_z \psi\,\dd \x\Big|\leq \liminf_{n\to \infty} 
\Big|\int_{\T^3} \phi( \diff_{-1/n} \psi)\,\dd \x\Big|
=\liminf_{n\to \infty} 
\Big|\int_{\T^3} (\diff_{1/n}\phi) \psi\,\dd \x\Big|\leq C_{\phi} \|\psi\|_{L^2}.
$$
Now $\partial_z \phi\in L^2$ and $\|\partial_z \phi\|_{L^2}\leq C_{\phi}$ follow from the Riesz representation theorem and the above inequality. Now we turn to the proof of the sufficiency of \eqref{eq:bound_diff_uniformh}. 
By Fatou's lemma, it follows that $\displaystyle{\liminf_{n\to \infty}\| \nabla \diff_{1/n} v(r)\|_{L^2}^2<\infty}$ a.e.\ on $\Gamma\times [s,\tau_j)$. Hence $\partial_z\nabla v\in L^2$ a.e.\ on $\Gamma\times [s,\tau_j)$ and satisfies $\E\int_s^{\tau_j}\one_{\Gamma} \| \nabla \partial_z v(r)\|_{L^2}^2 \,\dd t \leq C_0$ due to the preliminar observation.  This and the arbitrariness of $j\geq 1$ yield \eqref{eq:boundedness_L2_partialzv}.

\emph{Step 2: Proof of \eqref{eq:bound_diff_uniformh}}. 
To prove the above estimate, we obtain an SPDE for $\diff_r v$. To this end, let us recall the following Leibniz type rule:
For all maps $\psi,\phi:\T^3\to \R$, 
\begin{equation}
\label{eq:Leibnitz_rule_discrete}
\diff_r (\psi \phi)= (\diff_r \psi)\, \tr_r \phi + \psi \, \diff_r \phi\quad \text{ where } \ r\in  [-\tfrac{1}{2},\tfrac{1}{2}]\setminus \{0\}.
\end{equation}
Hence, one can readily check that $(\diff_r v |_{[s,\tau_j)\times \Gamma},\tau_j)$ is a local $(p,\a_{\crit},\s,q)$-solution to 
 \begin{align*}
		\partial_t \diff_r v   -\nabla\cdot(a\cdot \nabla \diff_r v )
		= F_r	+\sum_{n\geq 1} 
		\big[(\sigma_{n}\cdot\nabla) \diff_r v +   G_{n,r} \big]\, \dot{ \beta}_t^n\quad \text{ on }\Tor^3,
\end{align*} 
where
\begin{align*}
F_r&\stackrel{{\rm def}}{=}
\one_{[s,\tau_j)\times \Gamma}\big\{
\diff_r\big[ b(v,v)+\ftg(\cdot,v)  +f(\cdot,v,\nabla v)\big] +\nabla \cdot ([\diff_r a]\cdot \nabla [s_r v])\big\},\\
G_{n,r}&\stackrel{{\rm def}}{=}
\one_{[s,\tau_j)\times \Gamma}
\big\{g_n(\cdot,v)+([\diff_r \sigma_{n}]\cdot\nabla) [\tr_r v]\big\} , \quad \text{ for }n\geq1,
\end{align*}
 $\ft$ is defined in \eqref{eq:def_f_tilde} and $b(v,v)=- (v\cdot \nabla_{x,y}) v- w(v)\partial_z v$. Note that the term $\fts$ does not give any contribution as it is $z$-independent.

We claim that there exists $C_1=C_1(j)$ such that, for all 
$r\in [-\tfrac{1}{2},\tfrac{1}{2}]\setminus \{0\}$,
\begin{equation}
\label{eq:F_G_truncation_smooth}
\|F_r\|_{L^2((0,T)\times \O;H^{-1}) }
+\|(G_{n,r})_{n\geq 1}\|_{ L^2((0,T)\times \O;L^2(\ell^2))}\leq C_1.
\end{equation}
Note that, if \eqref{eq:F_G_truncation_smooth} holds, then \eqref{eq:bound_diff_uniformh} follows from the It\^o formula (see e.g.\ \cite[Theorem 4.2.5]{LR15}), the parabolicity condition of Assumption \ref{ass:primitive}\eqref{it:ass_primitive2} and a standard integration by parts argument.

In the remaining part of the proof, we show the validity of \eqref{eq:F_G_truncation_smooth}.
As the estimate for $(G_{n,r})_{n\geq 1}$ is analog to the one of $F_r$, we prove only the required estimate for the latter. We split the estimate for $F_r$ in several cases:
\begin{itemize}
\item
\emph{Proof of $\|\diff_r (b(v,v))\|_{L^2((s,\tau_j)\times \Gamma;H^{-1})}\leq C_1(j)$}. By \eqref{eq:def_w} and \eqref{eq:approximation3}, it follows that
$
b(v,v)= \nabla \cdot \big[ v\otimes (v,w(v))\big] \text{ a.e.\ on }(s,\tau_j)\times \Gamma. 
$
Hence
\begin{align*}
&\|\diff_r (b(v,v))\|_{L^2((s,\tau_j)\times \Gamma;H^{-1})}\\
&\quad \lesssim  
\|\diff_r(v\otimes v)\|_{L^2((s,\tau_j)\times \Gamma;L^2)}
+
\|\diff_r(w(v)v)\|_{L^2((s,\tau_j)\times \Gamma;L^2)}\\
&\quad \lesssim_{j}(1 +
\|\diff_r w(v)\|_{L^2((s,\tau_j)\times \Gamma;L^2)})
\end{align*}
where in the last step we used \eqref{eq:approximation3} and \eqref{eq:Leibnitz_rule_discrete}.

The boundedness of the maximal function in $L^2(\T_z)$ and \eqref{eq:def_w} yield
$$
\|\diff_r w(v)\|_{L^2((s,\tau_j)\times \Gamma;L^2)}\lesssim 
\|\nabla_{x,y} v\|_{L^2((s,\tau_j)\times \Gamma;L^2)}\stackrel{\eqref{eq:approximation3}}{\lesssim} j.
$$
\item 

\emph{$\|\nabla \cdot ([\diff_r a]\cdot \nabla [s_r v])\|_{L^2((s,\tau_j)\times \Gamma;H^{-1})}\leq C_1(j)$}. Note that, by Assumption \ref{ass:global} we have $\|\diff_r a\|_{L^2((0,T)\times \T^3;\R^{3\times 3})}\leq K_1$ where $K_1$ is independent of $r$. Thus
\begin{align*}
\|\nabla \cdot ([\diff_r a]\cdot \nabla [s_r v])\|_{L^2((s,\tau_j)\times \Gamma;H^{-1})}
\lesssim K_1\|\nabla v\|_{L^{\infty}((s,\tau_j)\times \Gamma\times \T^3)} \stackrel{\eqref{eq:approximation3}}{\lesssim}  K_1j .
\end{align*}

\item 
\emph{Estimate of the remaining terms of $F_r$}. The remaining terms can be estimated as in the previous items, where one also uses Assumption \ref{ass:primitive}\eqref{it:ass_primitive7}-\eqref{it:ass_primitive8} for the $f$-part and that 
$
\|\diff_r \phi\|_{H^{-1}}\lesssim \|\phi\|_{L^2}
$
for all $\phi\in L^2$.
\end{itemize}
This concludes the proof of Step 2.
\end{proof}

\subsection{Estimate of $\wt{v}$ and $\partial_z v$ -- Proof of \eqref{eq:energy_estimate2}}
\label{ss:main_intermediate_estimates}
We begin by collecting some useful facts. 
As explained in Subsections \ref{ss:reformulation}-\ref{ss:assumption_solution}, $(v,\tau)$ is a $(p,\a_{\crit},\s,q)$-solution to 
\begin{equation}
\label{eq:primitive_localized}
\left\{
\begin{aligned}
		&\partial_t v  		
		= \p\big[\nabla\cdot(a\cdot \nabla v)- (v\cdot \nabla_{x,y})v- w(v)\partial_z v\big]\\
		&\qquad \quad +\p \big[ \ftss v + F\big]
		+\sum_{n\geq 1} \p\big[ (\sigma_{n}\cdot\nabla) v+ G_n\big]\, \dot{ \beta}_t^n, &\text{on }&\Tor^3,\\
&v(\cdot,0)=v_0,&\text{on }&\Tor^3.
\end{aligned}
\right.
\end{equation}
where 
\begin{align*}
\ftss(\cdot) v &\stackrel{{\rm def}}{=} -\frac{h}{2}\sum_{n\geq 1} \nabla \cdot (\qq[(\sigma_n\cdot\nabla) v]\otimes \sigma_n),\\
F&\stackrel{{\rm def}}{=} (b\cdot\nabla) v+f(\cdot,v,\nabla v)+\ftg(\cdot,v) +b_{0,n} \qq[(\sigma_n\cdot\nabla) v],\\
G_n &\stackrel{{\rm def}}{=} g_{n}(\cdot,v).
\end{align*}

Next, we derive an SPDE for the barotropic mode $\wt{v}$. Arguing as in \eqref{eq:primitive_linear_wtv}, by the $z$-independence assumptions in Assumption \ref{ass:primitive}\eqref{it:ass_primitive4} and the fact that $\wt{\p [f]} = \wt{f}$ by \eqref{eq:def_p_q}, 
$(\wt{v}|_{[s,\tau)\times \O},\one_{\Gamma}\tau+ \one_{\O\setminus\Gamma} s)$ is a local $(p,\a_{\crit},\s,q)$-solution to 
\begin{equation}
\label{eq:primitive_vt}
\left\{
\begin{aligned}
		&\partial_t \wt{v}   =\nabla (a\cdot\nabla \wt{v})+\fwt v + 
		\sum_{\eta\in \{x,y\}}\big[ \partial_{z}(a_{z,\eta}\partial_{\eta} \overline{v})  
		-\partial_{\eta}(\overline{a_{\eta,z}\partial_z v}) \big]\\
		&\qquad \qquad \qquad \qquad \qquad
		- (\wt{v} \cdot \nabla_{x,y}) \wt{v} - w(\wt{v})\partial_z \wt{v}+ \mathcal{F}(\wt{v}) + \wt{F}\\
		&\qquad \qquad \qquad \qquad  \ \   
		  +\sum_{n\geq 1}\big(\p [(\sigma_{n}\cdot\nabla) \wt{v}]-\overline{\sigma_{n,z}\partial_z v}+ \wt{G}_{n}\big)\, \dot{ \beta}_t^n, &\text{on }&\Tor^3,\\
&\wt{v}(s,\cdot)=\one_{\Gamma}\wt{v}(s),&\text{on }&\Tor^3,
\end{aligned}
\right.
\end{equation}
where 
\begin{align*}
\fwt v&\stackrel{{\rm def}}{=} -\frac{h}{2}\sum_{n\geq 1} \qq [(\sigma_n\cdot\nabla) v] \partial_z\sigma_{n,z}
\qquad \text{(cf.\ \eqref{eq:def_fwt})},
\\
\mathcal{F}(\wt{v})& \stackrel{{\rm def}}{=} -(\overline{v}\cdot\nabla_{x,y} )\wt{v}-(\wt{v}\cdot\nabla_{x,y} )\overline{v}+ \overline{(\wt{v}\cdot\nabla_{x,y})\wt{v}}+\overline{(\nabla_{x,y}\cdot \wt{v})\wt{v}}.
\end{align*}
In deriving the SPDE \eqref{eq:primitive_vt} we also used that, due to \eqref{eq:incompressibility} and \eqref{eq:def_w}, $w(v)=w(\wt{v})$, $\partial_z v=\partial_z \wt{v}$ and integrating by parts
$$
\overline{(v\cdot\nabla_{x,y})v+w(v)\partial_z v} = 
(\overline{v}\cdot \nabla_{x,y})\overline{v}+\overline{(\wt{v}\cdot\nabla_{x,y})\wt{v}+(\nabla_{x,y}\cdot \wt{v})\wt{v}}.
$$

By Assumption \ref{ass:primitive}\eqref{it:ass_primitive7}-\eqref{it:ass_primitive8} and Lemma \ref{l:energy_estimate}, it readily follows that 
\begin{align}
\label{eq:def_nn}
\nn(r)
&\stackrel{{\rm def}}{=} (1+\| v(r)\|_{H^1}^2+\|F(r)\|_{L^2}^2+ \|G(r)\|_{H^1(\ell^2)}^2),\\
\label{eq:nn_estimate_proof}
\E \Big[\one_{\Gamma} \int_s^{\tau\wedge T}\one_{\Gamma} \nn(r) \,\dd r\Big] 
&\lesssim \big(1+\E\big[\one_{\Gamma}\|v(s)\|_{L^2}^2\big]\big),&
\end{align}
where the implicit constant in \eqref{eq:nn_estimate_proof} is independent of $\Gamma$ and $v_0$.

\begin{proof}[Proof of \eqref{eq:energy_estimate2} in Lemma \ref{l:main_intermediate_estimate_wt}]
The proof follows is similar to the one of \cite[Lemma 5.3]{Primitive1}. Therefore we only give some comments on the changes needed. 

Again, to prove \eqref{eq:energy_estimate2}, we use the stochastic Gronwall lemma of \cite[Lemma A.1]{AV_variational}. Therefore, let us fix two stopping times $\eta,\xi$ such that $s\leq \eta\leq \xi\leq \tau_j$ a.s.\ for some $j\geq 1$, where $\tau_j$ is as in \eqref{eq:tau_j_localization_main_proof}. It remains to prove the existence of a constant $R_0>0$, independent of $(\Gamma,j,\eta,\xi,v_0)$
 such that 
\begin{align}
\label{eq:gronwall_claim_intermediate_estimate}
&\E\sup_{r\in [\eta,\xi]}\one_{\Gamma}\big( \|\partial_z v(r)\|_{L^2}^2 + \|\wt{v}(r)\|_{L^4}^4\big)  \\
\nonumber
&\qquad \qquad
+ \E\int_{\eta}^{\xi}\one_{\Gamma}  (\|\partial_z v(r)\|_{H^1}^2+\big\||\wt{v}(r)||\nabla \wt{v}(r)|\big\|_{L^2}^2)\,\dd r\\
\nonumber
&\qquad\qquad\qquad\qquad \leq  R_0\Big(1+ \E[\one_{\Gamma}( \|\partial_z v(\eta)\|_{L^2}^2 + \|\wt{v}(\eta)\|_{L^4}^4)]\Big)\\
\nonumber
 &\qquad\qquad\qquad\qquad+R_0\E\int_{\eta}^{\xi}\one_{\Gamma}  \nn(r) \big(1+ \|\partial_z v(r)\|_{L^2}^2 + \|\wt{v}(r)\|_{L^4}^4\big)\,\dd r.
\end{align}

Indeed, if \eqref{eq:gronwall_claim_intermediate_estimate} holds, then \cite[Lemma A.1]{AV_variational} yields for all $\ell,L>0$
\begin{align*}
\P(\Gamma \cap \{\X_{\tau_j}\geq \ell\}) &
\leq  8R_0\frac{e^{8R_0 L}}{\ell} (1+\E[\one_{\Gamma} \|\partial_z v(s)\|_{L^2}^2 + \|\wt{v}(s)\|_{L^4}^4])\\
&\qquad\qquad + \P\Big(\Gamma\cap \Big\{\int_{s}^{\tau_j} \nn(r)\,\dd r \geq \frac{L}{2R_0}  \Big\}\Big)\\
&\stackrel{\eqref{eq:nn_estimate_proof}}{\lesssim_{R_0}}  \frac{e^{8R_0 L}}{\ell} 
(1+\E[\one_{\Gamma} \|\partial_z v(s)\|_{L^2}^2 + \|\wt{v}(s)\|_{L^4}^4])\\
&\qquad \qquad+ \frac{1}{L}(
1+\E\big[\one_{\Gamma}\|v(s)\|_{L^2}^2\big]),
\end{align*}
where $\X$ is as in the statement of Lemma \ref{l:main_intermediate_estimate_wt}.
Now,  \eqref{eq:energy_estimate2} follows by
taking $L=\frac{1}{8R_0}\log(\frac{\ell}{\log(\ell)})$ for $\ell$ large, and letting $j\to \infty$ in the previous.

Hence, it remains to prove \eqref{eq:gronwall_claim_intermediate_estimate}. We split the proof into three steps.

\emph{Step 1: (Estimates of $\partial_z v$). There exists $C_1\geq 1$ independent of $(\Gamma,j,\eta,\xi,v_0)$ such that}
\begin{align}
\label{eq:estimate_vz_proof_inequality}
\E \Big[\sup_{r\in [\eta,\xi]} \one_{\Gamma}\|\partial_z v(r)\|_{L^2}^2\Big]
&+\E\int_{\eta}^{\xi}\one_{\Gamma}\|\nabla \partial_z v (r)\|_{L^2}^2\,\dd r\\
\nonumber
& \leq C_1\big(1+\E[\one_{\Gamma} \|\partial_z v(\eta)\|_{L^2}^2]\big)\\
\nonumber
&+ C_1 \E\int_{\eta}^{\xi} \one_{\Gamma}\nn(r)(1+\|\partial_z v(r)\|_{L^2}^2)\,\dd r\\
\nonumber
& + C_1  \E\int_{\eta}^{\xi} \one_{\Gamma}\big\||\wt{v}(r)||\nabla \wt{v}(r)|\big\|_{L^2}^2\,\dd r.
\end{align}

The estimate \eqref{eq:estimate_vz_proof_inequality} follows as in \cite[Step 1, Lemma 5.3]{Primitive1} by applying the It\^o formula to $v\mapsto \|\partial_z v\|_{L^2}^2$. Note that the applications of the formula is allowed as \eqref{eq:regularity_of_partialz} holds. Here one also uses that $\E\int_{\eta}^{\xi}\one_{\Gamma}\|\nabla \partial_z v (r)\|_{L^2}^2\,\dd t<\infty$ by \eqref{eq:boundedness_L2_partialzv} and $\xi\leq \tau_j$ a.s.\ for some $j\geq1$. Note that the condition $(\partial_z \sigma_n)_{n\geq 1}\in L^{\infty}(\ell^2)$ in Assumption \ref{ass:global} can be used to estimate lower-order terms of the form $(\partial_z \sigma_{n}\cdot \nabla) v=(\partial_z \sigma_{n,z}) (\partial_z v)$, cf.\ the arguments between \cite[eq.\ (5.41)-(5.42)]{Primitive1}. Similar comments hold for the conditions $\partial_z a_{z,\eta},\partial_z a_{\eta,z}\in L^{\infty}$ in Assumption \ref{ass:global}, cf.\ \cite[Section 7]{Primitive1}.

\emph{Step 2: (Estimates of $\wt{v}$). Let $C_1$ be as in Step 1. Then there exists $C_2\geq 1$ independent of $(\Gamma,j,\eta,\xi,v_0)$ such that}
\begin{align}
\label{eq:estimate_wt_proof_inequality}
\E \Big[\sup_{r\in [\eta,\xi]} \one_{\Gamma}\|\wt{v}(r)\|_{L^4}^4\Big]
&+\E\int_{\eta}^{\xi}\one_{\Gamma}\big\||\wt{v}(r)||\nabla \wt{v}(r)|\big\|_{L^2}^2\,\dd t\\
\nonumber
 &\leq C_2\big(1+\E[\one_{\Gamma} \|\wt{ v}(\eta)\|_{L^4}^4]\big)\\
\nonumber
&+ C_2 \E\int_{\eta}^{\xi} \one_{\Gamma}\nn(r)(1+ \|\wt{v}(r)\|_{L^4}^4)\,\dd t \\
 \nonumber
&
+ \frac{1}{4C_1}\E\int_{\eta}^{\xi} \one_{\Gamma}\|\nabla\partial_z v\|_{L^2}^2\,\dd t.
\end{align}

As in the previous steps, the proof is analog to the one of \cite[Step 3, Lemma 5.3]{Primitive1} (see also \cite[Step 3, Subsection 1.4.3]{HH20_fluids_pressure}). We content ourselves to estimate the terms deriving from the terms $\partial_{z}(a_{z,\eta}\partial_{\eta} \overline{v}) $ and $-\partial_{\eta}(\overline{a_{\eta,z}\partial_z v})$, for which we need a different argument compared to \cite{Primitive1}. 

Recall that, in \cite[Step 3, Lemma 5.3]{Primitive1}, one applies the It\^o formula to $\wt{v}\mapsto \|\wt{v}\|_{L^4}^4$. Therefore, the above terms give rise to the contributions: For $\eta\in \{x,y\}$, 
\begin{align*}
A_1 &\stackrel{{\rm def}}{=}\E \int_{\eta}^{\xi}\one_{\Gamma}
\int_{\T^3} |\wt{v}|^2 \wt{v} \cdot \big[\partial_{z}(a_{z,\eta}\partial_{\eta} \overline{v}) \big]\,\dd x\dd t ,\\
A_2 &\stackrel{{\rm def}}{=}\E \int_{\eta}^{\xi}\one_{\Gamma}
\int_{\T^3} \partial_{\eta} \big[ |\wt{v}|^2 \wt{v}\big] \cdot \overline{a_{\eta,z}\partial_z v} \,\dd x\dd t.
\end{align*}

An inspection of the proof of \cite[Step 3, Lemma 5.3]{Primitive1} shows that to prove \eqref{eq:estimate_wt_proof_inequality} is it enough to obtain the following claim: 

For each $\nu_0>0$ there exists $C_{\nu_0}>0$ independent of $(\eta,\xi,j,v_0,\Gamma)$ such that 
\begin{align}
\label{eq:A1_A2_claim}
A_1 + A_2 
&\leq \nu_0 \E\int_{\eta}^{\xi} \big\||\wt{v}(r)| |\nabla \wt{v}(r)|\big\|_{L^2}^2\,\dd t \\
\nonumber
 &+ C_{\nu_0}\E\int_{\eta}^{\xi}\nn (r)\big(1+ \|\wt{v}(r)\|_{L^4}^4\big)\, \dd t.
\end{align}

To prove \eqref{eq:A1_A2_claim}, we split the arguments into three substeps. Below $\eta\in \{x,y\}$ and $\nu_0>0$ are fixed and $K$ is as in Assumption \ref{ass:global}.

\emph{Substep 2a: Estimate of $A_1$}. To begin note that $\partial_{z}(a_{z,\eta}\partial_{\eta} \overline{v}) =[\partial_{z}a_{z,\eta} ]\partial_{\eta} \overline{v} $ as $\overline{v}$ is $z$-independent. Hence, the H\"older inequality shows
\begin{align*}
\Big|\int_{\T^3} |\wt{v}|^2 \wt{v} \cdot \big[\partial_{z}(a_{z,\eta}\partial_{\eta} \overline{v}) \big]\,\dd x\Big|
&\leq \big\||\wt{v}|^3 \big\|_{L^2_{x,y}(L^1_z)} \|\partial_z a_{z,\eta}\|_{L^{\infty}} \|\partial_{\eta}\overline{v}\|_{L^2_{x,y}}.
\end{align*}
Note that the first term on the RHS of the above estimate can be estimated by:
\begin{align*}
\big\||\wt{v}|^3 \big\|_{L^2_{x,y}(L^{4/3}_z)}
&= \big\||\wt{v}|^2\big\|_{L^3_{x,y}(L^2_z)}^{3/2}		\\
& \stackrel{(i)}{\lesssim} 
\big\||\wt{v}|^2\big\|_{H^{1/3}_{x,y}(L^2_z)}^{3/2}\\
&\stackrel{(ii)}{\lesssim}  
\big\||\wt{v}|^2\big\|_{L^2}^{3/2}+ \big\||\wt{v}|^2\big\|_{L^2}\big\|\nabla |\wt{v}|^2\big\|_{L^2}^{1/2}\\
&\lesssim 
\|\wt{v}\|_{L^4}^{3}+ \|\wt{v}\|_{L^4}^2\big\||\wt{v}| |\nabla \wt{v}|\big\|_{L^2}^{1/2}
\end{align*}
where in $(i)$ we used Sobolev embeddings, and in $(ii)$ that $H^{1/3}\embed H^{1/3}_{x,y}(L^2_z)$ (see Lemma \ref{l:anisotropic_B}\eqref{it:anisotropic_B3}) and standard interpolation inequalities.

Since $\|\partial_{\eta}\overline{v}\|_{L^2_{x,y}}\lesssim \|\nabla v\|_{L^2}$,  the Young inequality and the above estimates yield, for all $\mu_0>0$,
\begin{align}
\label{eq:A1_estimate}
\Big|\int_{\T^3} |\wt{v}|^2 \wt{v} \cdot \big[\partial_{z}(a_{z,\eta}\partial_{\eta} \overline{v}) \big]\,\dd x\Big|
\leq \frac{\nu_0}{2}  \big\||\wt{v}| |\nabla \wt{v}|\big\|_{L^2}^2 + C_{\nu_0}' \|\nabla v\|_{L^2}^{4/3}\|\wt{v}\|_{L^4}^{8/3},
\end{align}
where $C_{\nu_0}'$ is independent of $(\eta,\xi,j,v_0,\Gamma)$.

\emph{Substep 2b: Estimate of $A_2$}. We begin with a preliminary observation. Note that, integrating by parts, $\overline{a_{\eta,z}\partial_z v}=-\overline{[\partial_z a_{\eta,z}] v}$. Hence, by the Sobolev embedding $H^{1}_{x,y}\embed L^4_{x,y}$ and Lemma \ref{l:anisotropic_B}\eqref{it:anisotropic_B3}, 
$$
\|\overline{a_{\eta,z}\partial_z v}\|_{L^4_{x,y}}\lesssim_K \|v\|_{L^4_{x,y}(L^2_z)}\lesssim \|v\|_{H^1}.
$$
Moreover, since $|\partial_{\eta} [ |\wt{v}|^2 \wt{v}] |\lesssim |\wt{v}|^2 |\nabla \wt{v}|$ pointwise, we have
\begin{align}
\label{eq:A2_estimate}
\Big|\int_{\T^3} \partial_{\eta} \big[ |\wt{v}|^2 \wt{v}\big] \cdot \overline{a_{\eta,z}\partial_z v} \,\dd x \Big|
&\leq \big\| |\wt{v}| |\nabla\wt{v}|\big\|_{L^2} \|\wt{v}\|_{L^4} \|\overline{[\partial_z a_{\eta,z}] v}\|_{L^4}\\
\nonumber
&\leq \frac{\nu_0}{2} \big\| |\wt{v}| |\nabla\wt{v}|\big\|_{L^2}^2 + C_{\nu_0}''\|v\|_{H^1}^2\|\wt{v}\|_{L^4}^2.
\end{align}
where $C_{\nu_0}''$ is independent of $(\eta,\xi,j,v_0,\Gamma)$.

\emph{Substep 2c: Proof of \eqref{eq:A1_A2_claim}}. It immediately follows from \eqref{eq:A1_estimate} and \eqref{eq:A2_estimate}, the definition of $\nn$ in \eqref{eq:def_nn}, and Young's inequality.

\emph{Step 3: Proof of \eqref{eq:gronwall_claim_intermediate_estimate}}. The estimate \eqref{eq:gronwall_claim_intermediate_estimate} follows by multiplying the equation \eqref{eq:estimate_vz_proof_inequality} by $\frac{1}{2C_1}$ and then summing the corresponding inequalities, i.e., more schematically
$$
\frac{1}{2C_1}\eqref{eq:estimate_vz_proof_inequality} + \eqref{eq:estimate_wt_proof_inequality} \Longrightarrow \eqref{eq:gronwall_claim_intermediate_estimate}.
$$

Indeed, the terms $\frac{1}{2}  \E\int_{\eta}^{\xi} \one_{\Gamma}\big\||\wt{v}||\nabla \wt{v}|\big\|_{L^2}^2\,\dd t$ and $\frac{1}{4C_1}\E\int_{\eta}^{\xi} \one_{\Gamma}\|\nabla \partial_z v\|_{L^2}^2\,\dd t$ appearing on the RHS of the estimate $
\frac{1}{2C_1}\eqref{eq:estimate_vz_proof_inequality} + \eqref{eq:estimate_wt_proof_inequality}$ can be absorbed on the LHS of the corresponding estimates, as on the LHS the terms $  \E\int_{\eta}^{\xi} \one_{\Gamma}\big\||\wt{v}||\nabla \wt{v}|\big\|_{L^2}^2\,\dd t$ and $\frac{1}{2C_1}\E\int_{\eta}^{\xi} \one_{\Gamma}\|\nabla \partial_z v\|_{L^2}^2\,\dd t$ are present. Note that the previous argument is valid also because $  \E\int_{\eta}^{\xi} \one_{\Gamma}\big\||\wt{v}||\nabla \wt{v}|\big\|_{L^2}^2\,\dd t<\infty$ and $\E\int_{\eta}^{\xi} \one_{\Gamma}\|\nabla \partial_z v\|_{L^2}^2\,\dd t<\infty$ since $s\leq \eta\leq \xi\leq \tau_j$ a.s.\ for some $j\geq 1$, cf.\ \eqref{eq:tau_j_localization_main_proof} and \eqref{eq:boundedness_L2_partialzv}.
\end{proof}

\section{The non--isothermal case}
\label{s:non_isothermal}
As announced below Theorem \ref{t:global_intro}, the results stated in Section \ref{s:statement} have a fairly straightforward extension to non-isothermal PEs of the form on the three-dimensional torus: 
\begin{equation}
\label{eq:primitive_temperature}
\left\{
\begin{aligned}
		&\partial_t v  -\big[\nabla\cdot(a\cdot \nabla v) + (b\cdot\nabla) v\big]\\
		&\qquad \qquad = \big[-\ph p - (v\cdot \nabla_{x,y})v- w\partial_z v+\ps\wt{p} \big] \\ 
		&\qquad \qquad \qquad\qquad\quad
				+\sum_{n\geq 1} [-\ph \wt{p}+ (\sigma_{n}\cdot\nabla) v]\, \dot{ \beta}_t^n, &\text{ on }&\Tor^3,\\
						&\partial_t \theta   -\nabla(d\cdot\nabla \theta)=
						\big[(k\cdot\nabla) \theta -(v\cdot\nabla_{x,y})\theta -w \partial_z \theta	\big]\\
						&\qquad \qquad \qquad\qquad\qquad\qquad\qquad\ \ 
						+\sum_{n\geq 1}(\nt_{n}\cdot\nabla) \theta\, \dot{ \beta}_t^n, \quad &\text{ on }&\Tor^3,\\
						&\partial_z p + \theta=0, \qquad  \partial_z \wt{p}_n=0,&\text{ on }&\Tor^3,\\
&\nabla_{x,y} \cdot v +\partial_z w=0,&\text{ on }&\Tor^3,\\
&v(0,\cdot)=v_0,\qquad \theta(0,\cdot)=\theta_0,&\text{ on }&\Tor^3.
\end{aligned}
\right.
\end{equation}
As in Section \ref{s:statement}, the above system is complemented with the boundary conditions $w|_{\T^2_{x,y}\times \{0,1\}}=0$ and
\begin{equation}
\label{eq:def_h_statement_temp}
\ps \wt{p}\stackrel{{\rm def}}{=}\frac{h}{2}\sum_{n\geq 1} 
\big[\nabla \cdot (\ph \wt{p}_n\otimes \sigma_n) + b_{0,n} \ph \wt{p}_n\big] \ \ \text{ where }\ \ h\in [ -1,\infty),
\end{equation} 
and $(a,b,b_0,d,k)$ are specified in Assumption \ref{ass:primitivetemp} below.
As commented below Corollary \ref{cor:stratonovich} (see also \cite[Section 8]{Primitive1}), the coefficients $(d,k)$ appear in case of transport noise in the Stratonovich form. As for \eqref{eq:primitive}, we can allow (globally) Lipschitz nonlinear lower order terms in \eqref{eq:primitive_temperature}. We omit this for brevity.


From a physical point of view, transport noise in the context of reaction-diffusion equations can be motivated via separation of scale, see e.g.\ in \cite[Subsection 1.3]{AV22_localRD}. Let us recall that transport noise also allows for maximum principles, see e.g.\ \cite[Subsection 2.4 and Appendix A]{AV22_localRD} and \cite[Section 4]{Kry13}, which is of physical interest as  $\theta$ represents the temperature.

\subsection{Reformulation, main assumption and $(p,\a,\s_0,\s_1,q)$-solutions}
As for \eqref{eq:primitive}, the problem \eqref{eq:primitive_temperature} can be reformulated only in terms of $(v,\theta)$. 
To begin, note that, integrating the hydrostatic balance $\partial_z p+\theta=0$, we obtain
$$
p(t,\x)= p_s(t,x,y) + \int_{0}^z \theta(t,x,y,z')\,\dd z' \ \ \text{ a.s.\ for all }t\in \R_+\text{ and }\x\in \T^3.
$$
Note that $p_s$ depends only on the horizontal variables, and therefore it plays the role of the surface pressure in \eqref{eq:primitive}. In light of the previous observations, the reformulation of \eqref{eq:primitive} is similar to the one of \eqref{eq:primitive} given in Subsection \ref{ss:reformulation}. Indeed, letting $\p, w(v)$ and $\ft$ be as in Subsection \ref{ss:reformulation}, one can readily check that \eqref{eq:primitive_temperature} is formally equivalent to the following generalization version of \eqref{eq:primitive2}:
\begin{equation}
\label{eq:primitive2_temp}
\left\{
\begin{aligned}
		&\partial_t v  		
		= \p\big[\nabla\cdot(a\cdot \nabla v)+(b\cdot\nabla) v- (v\cdot \nabla_{x,y})v- w(v)\partial_z v\big]\\
		&\  \  +\p\Big[\ft(\cdot,v) +\int_0^{\cdot}\ph \theta(\cdot,z')\,\dd z'\Big]
		+\sum_{n\geq 1} \p\big[ (\sigma_{n}\cdot\nabla) v\big]\, \dot{ \beta}_t^n, &\text{on }&\Tor^3,\\
		&\partial_t \theta
		=\big[ \nabla\cdot(d\cdot \nabla\theta)+ (k\cdot\nabla)\theta - (v\cdot \nabla_{x,y})\theta- w(v)\partial_z \theta\big]		\\
		&\qquad \qquad \qquad \qquad\qquad \qquad \qquad\qquad \ \  
		+\sum_{n\geq 1} (\nt_{n}\cdot\nabla) \theta\, \dot{ \beta}_t^n, &\text{on }&\Tor^3,\\
&v(0,\cdot)=v_0, \qquad \theta(0,\cdot)=\theta_0,&\text{on }&\Tor^3.
\end{aligned}
\right.
\end{equation}
We investigate the above reformulation of \eqref{eq:primitive_temperature} under the following conditions which extend the one in Assumption \ref{ass:primitive} to the current situation.

\begin{assumption} 
\label{ass:primitivetemp}
We say that Assumption \ref{ass:primitivetemp}$(p,\s_0,\s_1,q)$ holds if the following conditions are satisfied.
\begin{enumerate}[{\rm(1)}]
\item $h\in [-1,\infty)$.
\item $\delta_0\in [0,1)$, $\delta_1\in [0,2)$ and either $[\, p\in (2,\infty)\text{ and }q\in [2,\infty)\, ]$ or $[\, p=q=2\, ]$.
\item\label{it:primitivetemp1} For all $n\geq 1$ the following mapping are $\Progress\otimes \Borel(\T^3)$-measurable
\begin{align*}
 a=(a_{\eta,\xi})_{\eta,\xi\in \{x,y,z\}}, 
 d=(d_{\eta,\xi})_{\eta,\xi\in \{x,y,z\}}&:\R_+\times \O\times \T^3\to \R^{3\times 3},\\
b=(b_{\xi})_{\xi\in \{x,y,z\}},\  k=(k_{\xi})_{\xi\in \{x,y,z\}}	, &\\
 \sigma_n=(\sigma_{n,\xi})_{\xi\in \{x,y,z\}},\ \nt_n=(\nt_{n,\xi})_{\xi\in \{x,y,z\}}
&:\R_+\times \O\times \T^3\to \R^3,\\
 b_{0,n} =&:\R_+\times \O\times \T^3\to \R.
\end{align*}
\item\label{it:primitivetemp2} For all $n\geq 1$, the mappings
$$
(a_{\eta,\xi})_{\eta,\xi\in \{x,y\}},  (\sigma_{n,\xi})_{\xi\in \{x,y\}}   \text{ are independent of }z\in \T_z.
$$
\item\label{it:primitivetemp3} There exist $M,\g_0,\g_1>0$ such that $\g_0> 1-\s_0$, $\g_1>|1-\s_1|$, and for a.a.\ $(t,\om)\in  \R_+\times \O$,
\begin{align*}
\|a(t,\om,\cdot)\|_{C^{\g_0}(\Tor^3;\R^{3\times 3})}
+\|(\sigma_n(t,\om,\cdot))_{n\geq 1}\|_{C^{\g_0}(\Tor^3;\ell^2)}\leq M   ,&\\
\|d(t,\om,\cdot)\|_{C^{\g_1}(\Tor^3;\R^{3\times 3})}
+\|(\nt_n(t,\om,\cdot))_{n\geq 1}\|_{C^{\g_1}(\Tor^3;\ell^2)}\leq M   .&
\end{align*}

\item\label{it:primitivetemp4} There exists $\nu\in (0,1)$ such that for all $\lambda=(\lambda_{\xi})_{\xi\in \{x,y,z\}}\in \R^3$
\begin{align*}
\sum_{\eta,\xi\in \{x,y,z\}} \Big(a_{\eta,\xi}-\frac{1}{2}\sum_{n\geq 1} \sigma_{n,\eta} \sigma_{n,\xi}\Big) 
\lambda_{\eta}\lambda_{\xi}  \geq  \ellip|\lambda|^2 \ \ \text{ a.e.\ on }\R_+\times \O,&\\
\sum_{\eta,\xi\in \{x,y,z\}} \Big(d_{\eta,\xi}-\frac{1}{2}\sum_{n\geq 1} \nt_{n,\eta} \nt_{n,\xi}\Big) 
\lambda_{\eta}\lambda_{\xi}  \geq  \ellip|\lambda|^2 \ \ \text{ a.e.\ on }\R_+\times \O.&
\end{align*}
 \item\label{it:primitivetemp5} There exists $N>0$ such that,  for a.a.\ $(t,\om)\in \R_+\times \O$,
 \begin{equation*}
 \|b(t,\om,\cdot)\|_{L^{\infty}(\T^3;\R^3)}+
 \|(b_{0,n}(t,\om,\cdot))_{n\geq 1}\|_{L^{\infty}(\T^3;\ell^2)}\leq N .
 \end{equation*}
 \end{enumerate}
\end{assumption} 

Compared to Assumption \ref{ass:primitive}, the additional parameter $\delta_1$ rules the space regularity of $\theta$,  measured in spaces of the form $L^{p}_t (H_{\x}^{2-\s_1,(q,2)})$, while $\s_0$ plays the role of $\s$ in the isothermal case. Let us remark that no further restrictions on $\s_1$ are enforced below, while we shall assume that $\s_0<\frac{1}{2}$ as in Section \ref{s:statement}. In particular, if $\delta_1\geq 1$, then $\theta$ has regularity $\leq 1$ and in this scenario, the weak setting for the temperature equation is preferable (cf.\ \cite{Primitive1}). 
Hence, depending on $\s_1$, we rewrite the nonlinearity in the $\theta$-equation of \eqref{eq:primitive_temperature} as follows:
\begin{equation}
\label{eq:def_b_delta_one}
b_{\delta_1}(v,\theta)\stackrel{{\rm def}}{=}
\left\{
\begin{aligned}
&(v\cdot\nabla_{x,y})\theta +w(v)\partial_z \theta  \quad & \text{ if } \ &\delta_1 \in [\delta_0,1],\\
&\nabla_{x,y}(v\,\theta)+\partial_z(w(v)\,\theta)  \quad & \text{ if } \ &\delta_1 \in [1,1+\delta_0].
\end{aligned}
\right.
\end{equation}
The (formal) equivalence of $b_{\delta_1}(v,\theta)$ with $(v\cdot\nabla_{x,y})\theta +w(v)\partial_z \theta $ follows from $\nabla \cdot u=0$ where $u=(v,w(v))$.

Similar to Definition \ref{def:solution}, we define $(p,\a,\s_0,\s_1,q)$-solutions to \eqref{eq:primitive_temperature}.
Below $\Br_{\ell^2}$ is the $\ell^2$-cylindrical Brownian motion associated to $(\beta^n)_{n\geq 1}$, see \eqref{eq:def_Br}.
 
\begin{definition}[$(p,\a,\s_0,\s_1,q)$-solution]
\label{def:solution_temp}
Suppose that Assumption \ref{ass:primitivetemp}$(p,\s_0,\s_1,q)$ holds. Let $\a\in[0,\frac{p}{2}-1)$ if $p>2$ or $\a=0$ otherwise.
Let $\tau$ be a stopping time, and let $v:[0,\tau)\times \O\to \Hs^{2-\s_0,(q,2)}(\T^3)$, $\theta:[0,\tau)\times \O\to H^{2-\s_1,(q,2)}(\T^3)$ be a stochastic processes. 
\begin{enumerate}[{\rm(1)}]
\item We say that $((v,\theta),\tau)$ is a \emph{local $(p,\a,\s_0,\s_1,q )$-solution} to \eqref{eq:primitive_temperature} if the there exists a sequence of stopping times $(\tau_j)_{j\geq 1 }$ for which the following hold for all $j\geq 1$:
\begin{itemize}
\item $\tau_j\leq \tau$ a.s.;
\item $\one_{[0,\tau_j]} v$ is progressively measurable, and a.s.\ 
\begin{align*}
v\in L^p(0,\tau_j,w_{\a};H^{2-\s_0,(q,2)}), \qquad 
\theta\in L^p(0,\tau_j,w_{\a};H^{2-\s_1,(q,2)}),&\\
-(v\cdot\nabla_{x,y})v - w(v)\partial_z v+\ftg(\cdot,v)+\int_0^{\cdot}\ph \theta(\cdot,z) \,\dd z 
\in L^p(0,\tau_j,w_{\a};H^{-\s_0,(q,2)}),&\\
b_{\delta_1}(v,\theta)
\in L^p(0,\tau_j,w_{\a};H^{-\s_1,(q,2)});&
\end{align*}
\item for all $t\in [0,\tau_j]$ and a.s.\ 
\begin{align*}
v(t)-v_0 
&=\int_0^t  \p\Big[\nabla\cdot(a\cdot\nabla v)+ (b\cdot\nabla) v-(v\cdot\nabla_{x,y})v- w(v)\partial_z v\\
&\qquad \qquad \qquad\qquad \qquad
+\ft(\cdot,v)+\int_0^{\cdot}\ph \theta(\cdot,z') \,\dd z' \Big]\,\dd s \\
&\qquad \qquad\qquad \qquad \qquad 
+\int_0^t \one_{[0,\tau_j]} \big(\p[(\sigma_n\cdot\nabla) v\big]\big)_{n\geq 1}\,\dd \Br_{\ell^2}(s),\\
\theta(t)-\theta_0 
&=\int_0^t  \big[\nabla\cdot(d\cdot\nabla v)+ (k\cdot\nabla) v+ b_{\delta_1}(v,\theta) \big]\,\dd s \\
&\qquad \qquad \qquad\qquad \qquad 
+\int_0^t \one_{[0,\tau_j]} \big((\nt_n\cdot\nabla) \theta\big)_{n\geq 1}\,\dd \Br_{\ell^2}(s),
\end{align*}
\end{itemize}
where $w(v)$ and $\ft(\cdot,v), \ftg(\cdot,v) $ are as in \eqref{eq:def_w} and \eqref{eq:def_f_tilde}-\eqref{eq:def_f_tilde2}, respectively.
\item A local $(p,\a,\s_0,\s_1,q)$-solution  to \eqref{eq:primitive_temperature} is said to be a \emph{$(p,\a,\s_0,\s_1,q)$-solution} if any other local $(p,\a,\s_0,\s_1,q)$-solution $(v',\tau')$  to \eqref{eq:primitive_temperature} we have $\tau'\leq \tau$ a.s.\ and $v=v'$ a.e.\ on $[0,\tau')\times \O$.
\item A $(p,\a,\s_0,\s_1,q)$-solution $(v,\tau)$  to \eqref{eq:primitive_temperature} is said to be \emph{global} if $\tau=\infty$ a.s. 
\end{enumerate}
\end{definition}

\subsection{Statement of main results -- Non-isothermal case}
\label{ss:statement_non_isothermal}
In this subsection, we state some extensions of the results in Section \ref{s:statement}. We begin with the following local existence which extends Theorem \ref{t:local} to the current situation.

\begin{theorem}[Local existence and uniqueness in critical spaces -- Non-isothermal case]
\label{t:local_temp}
Assume that $\s_1\in [\s_0,1+\s_0]$ and 
\begin{equation}
\label{eq:assumptions_critical_setting_temp}
\begin{aligned}
\text{ either }\ &\big[ q=p=2 \text{ and }\s_0=0\big],\\
 \text{ or }\  & \Big[
 \s_0\in \Big(0,\frac{1}{2}\Big),\   q\in \Big(\frac{2}{2-\s_0},\frac{2}{1-\s_0}\Big) \text{ and }\
 \frac{1}{p}+\frac{1}{q}+\frac{\delta_0}{2}\leq 1\Big], 
\end{aligned}
\end{equation}
Let Assumption \ref{ass:primitive}$(p,\s_0,\s_1,q)$ be satisfied and set $\a=\a_{\crit}\stackrel{{\rm def}}{=} p(1-\frac{1}{q}-\frac{\delta_0}{2})-1$. 
Then for each 
\begin{equation}
\label{eq:regularity_data_temp}
v_0\in L^0_{\F_0}(\O;\Bs^{2/q}_{(q,2),p}(\T^3)) \quad \text{ and }\quad  \theta_0\in L^0_{\F_0}(\O;B^{2/q-\s_1+\s_0}_{(q,2),p}(\T^3))
\end{equation}
the problem \eqref{eq:primitive_temperature} has a (unique) $(p,\a_{\crit},\s,q)$-solution
$((v,\theta),\tau)$ such that  a.s.\ $\tau>0$  and for all $\eta\in [0,\frac{1}{2})$
\begin{align}
\label{eq:reg_critical1_temp}
(v,\theta)&\in  H^{\eta,p}_{\loc}([0,\tau),w_{\a_{\crit}};\Hs^{2-\s_0-2\eta,(q,2)}(\Tor^3)\times H^{2-\s_1-2\eta,(q,2)}(\T^3)) ,\\
\label{eq:reg_critical2_temp}
(v,\theta)&\in  C([0,\tau);\Bs^{2/q}_{(q,2),p}(\Tor^3)\times B^{2/q-\s_1+\s_0}_{(q,2),p}(\T^3)).
\end{align}
\end{theorem}

The choice $\s_1=1+\s_0$ leads to the weakest regularity assumption on $\theta_0$, namely $\theta_0\in B^{2/q-1}_{(q,2),p}(\T^3)$. However, such a choice is only possible if $\g_1>|1-\s_1|=\s_0$, see Assumption \ref{ass:primitivetemp}\eqref{it:primitivetemp3}. By Proposition \ref{prop:smoothness_Kraichnan} and $\s_0<\frac{1}{2}$ due to  \eqref{eq:assumptions_critical_setting_temp}, the latter is satisfied in case $(\nt_n)_{n\geq 1}$ reproduces the Kolmogorov spectrum, i.e.\ $\nt_n$ is enumeration of \eqref{eq:Kraich_1}-\eqref{eq:Kraich_2} with $d=3$ and $\alpha=\frac{4}{3}$.
For physical motivations, we refer to \cite{P91_Kolmogorov}.

The following is an extension of Theorem \ref{t:regularity1} to \eqref{eq:primitive_temperature}.

\begin{theorem}[Instantaneous regularization -- Non isothermal case]
\label{t:regularity_temp}
Let the assumptions of Theorem \ref{t:local_temp} be satisfied and let $((v,\theta),\tau)$ be the $(p,\a_{\crit},\s_0,\s_1,q)$-solution to \eqref{eq:primitive_temperature} where $\a_{\crit}\stackrel{{\rm def}}{=} p(1-\frac{1}{q}-\frac{\delta_0}{2})-1$. 
Let $(\g_0,\g_1)$ be as in Assumption \ref{ass:primitivetemp}\eqref{it:primitivetemp3}. Then $((v,\theta),\tau)$ instantaneously regularizes in time and space: \eqref{eq:inst_reg1}-\eqref{eq:inst_reg2} hold with $\g=\g_0$ and 
\begin{align}
\label{eq:inst_reg11}
\theta&\in  H^{\eta,r}_{\loc}(0,\tau; H^{1+\alpha_1-2\vartheta,r})
\text{ a.s.\ for all }\eta\in [0,\tfrac{1}{2}),\, \alpha_1<\g_1'
,\, r \in ( 2,\infty),\\
\label{eq:inst_reg22}
\theta& \in C^{\mu,\nu}_{\loc}((0,\tau)\times \T^3;\R^2) \text{ a.s.\ for all }\mu\in [0,\tfrac{1}{2}), \, \nu\in (0,1+\g_1'),
\end{align}
where $\g_1'\stackrel{{\rm def }}{=}1\wedge \g_1$.
\end{theorem}

The extension of Theorem \ref{t:regularity2} to \eqref{eq:primitive_temperature} is not straightforward and requires some modifications. For brevity, we only provide some comments in the following 

\begin{remark}[High-order regularity -- non-isothermal case]
\label{r:high_order_temp}
To prove Theorem \ref{t:regularity2} to \eqref{eq:primitive_temperature}, we need to extend Lemma \ref{l:nonlinearities_high_order} to the current situation. However, note that $\big[(x,y,z)\mapsto \int_{0}^{z}\nabla_{x,y}\theta(x,y,z')\,\dd z'\big]\in H^{s,q}(\T^3;\R^2)$ with $s>\frac{3}{q}$ implies
\begin{equation}
\label{eq:null_gradient_mean_temperature}
\int_{\T_z}\nabla_{x,y}\theta(\cdot,z)\,\dd z=0\ \ \text{ on }\ \T^2_{x,y}.
\end{equation}
The above follows from the Sobolev embedding $H^{s,q}(\T^3)\subseteq C(\T^3)$ if $s>\frac{3}{q}$.

Let us note that \eqref{eq:null_gradient_mean_temperature} is satisfied if $\theta=\theta_{{\rm const}}+\theta_{{\rm odd}}$ whenever $\theta_{{\rm const}}\in \R$ and $\theta_{{\rm odd}}$ is odd w.r.t.\ $\{z=\frac{1}{2}\}$.
The odd part of $\theta$ is consistent with the derivation of the PEs with temperature (see \cite[eq.\ (1.4)]{PZ22}). Moreover, the oddness also arises via a reflection to treat the PEs \eqref{eq:primitive_temperature} on $\T^2_{x,y}\times [0,\frac{1}{2}]$ where additionally to \eqref{eq:no_flux_bc} we enforce the following boundary conditions for the temperature:
$$
\partial_z\theta=0 \ \text{ on }\ \T^2_{x,y}\times\{0\} \quad \text{ and }\quad 
\theta=0 \ \text{ on }\ \T^2_{x,y}\times\{\tfrac{1}{2}\}.
$$
\end{remark}

Finally, we state the analog of Theorem \ref{t:global} for the problem \eqref{eq:primitive_temperature}. 

\begin{theorem}[Global well-posedness -- Non isothermal case]
\label{t:global_temp}
Let the assumptions of Theorem \ref{t:local_temp} be satisfied. 
Let $((v,\theta),\tau)$ be the corresponding $(p,\a_{\crit},\s_0,\s_1,q)$-solution to \eqref{eq:primitive_temperature} with initial data as in \eqref{eq:regularity_data_temp} and $\a_{\crit}\stackrel{{\rm def}}{=}p(1-\frac{1}{q}-\frac{\s_0}{2})-1$. Let Assumption \ref{ass:global} be satisfied. Then 
$(v,\tau) \text{ is \emph{global} in time, i.e.\  $\tau=\infty$ a.s.}$

Moreover, for all sequences $(v_{0}^{(n)},\theta_{0}^{(n)})_{n\geq 1}$ of initial data as in \eqref{eq:regularity_data_temp} such that $(v_{0}^{(n)},\theta_{0}^{(n)})\to (v_0,\theta_0)$ in probability in $\Bs^{2/q}_{(q,2),p}(\T^3)\times B^{2/q-\s_1+\s_0}_{(q,2),p}(\T^3)$, we have, for all $T\in (0,\infty)$,
$$
(v^{(n)},\theta^{(n)}) \to (v,\theta) \ \text{ in probability in } \ C([0,T];\Bs^{2/q}_{(q,2),p}(\T^3)\times B^{2/q-\s_1+\s_0}_{(q,2),p}(\T^3)),
$$
where $(v^{(n)},\theta^{(n)})$ is the global $(p,\a,\s_0,\s_1,q)$-solution of \eqref{eq:primitive_temperature} with data $(v_0^{(n)},\theta_0^{(n)})$.
\end{theorem}

Interestingly, no additional assumption on $(\nt_n)_{n\geq 1}$ are required for global well-posedness of \eqref{eq:primitive_temperature}. This is in accordance with the main result of \cite[Section 3]{Primitive1}.

The proofs of Theorems \ref{t:local_temp}, \ref{t:regularity_temp} and \ref{t:global_temp} are a straightforward generalization of the corresponding ones given in Section \ref{s:statement}, and are given in Subsection \ref{ss:proof_temp} below. 
Before going into the proofs,  we mention that one can also extend the following results to \eqref{eq:primitive_temperature}:
Local continuity (Proposition \ref{prop:continuity}), compatibility (Corollary \ref{cor:compatibility}), Serrin-type blow-up criteria (Theorem \ref{t:serrin}), and the Feller property (Corollary \ref{cor:Feller}). 
We omit the details for brevity. A Serrin-type blow-up criterion for \eqref{eq:primitive_temperature} is given in \eqref{eq:serrin_temp} below and is needed in the proof of Theorem \ref{t:global_temp}.

\subsection{Proof of the main results -- Non-isothermal case}
\label{ss:proof_temp}

\subsubsection{Proof of Theorems \ref{t:local_temp} and \ref{t:regularity_temp}} 
We begin by extending the proof of 
Proposition \ref{prop:local} to \eqref{eq:primitive_temperature}. Recall that the additional parameter $\zeta\geq 2$ is needed to prove Theorem \ref{t:regularity_temp}, cf.\  Subsection \ref{ss:proof_serrin_regularity}. 
To begin, 
we recast \eqref{eq:primitive_temperature} as a stochastic evolution equation on a suitable  $\XX_0=\Hs^{2-\s_0,(q,\zeta)}\times H^{2-\s_1,(q,\zeta)}$ for the unknown $\uu=(v,\theta)^{\top}$ of the form:
\begin{equation}
\label{eq:SEE_temp}
\left\{
\begin{aligned}
\dd \uu + \A(t)\uu\, \dd t &= \FF(t,\uu)\,\dd t + (\B(t)\uu+\GG(t,\uu))\,\dd \Br_{\ell^2}(t),\quad t\in\R_+,\\
 \uu(0)&=(v_0,\theta_0)^{\top},
\end{aligned}
\right.
\end{equation}
where $\zeta\in [2,\infty)$, $\Br_{\ell}$ is the $\ell^2$-cylindrical Brownian motion associated to $(\beta^n)$ as in \eqref{eq:def_Br} and, for $\uu\in \XX_1$,
\begin{equation}
\label{eq:def_ABFG_temp}
\begin{aligned}
\A(\cdot) \uu&= 
\begin{bmatrix}
\vspace{0.1cm}
-\p\big[\nabla\cdot(a\cdot \nabla v )+(b\cdot\nabla) v-\fts(\cdot)v\big] & \displaystyle{\int_{0}^{\cdot}\nabla_{x,y}\theta (\cdot,z)\,\dd z}\\
0 & -\nabla \cdot(d\cdot \nabla \theta) - (k\cdot \nabla)\theta 
\end{bmatrix}
, \\  
\B(\cdot)\uu&= 
\begin{bmatrix}
\vspace{0.1cm}
(\p[(\sigma_n\cdot\nabla) v])_{n\geq 1} & 0\\
0 & (\nt_n\cdot\nabla) \theta
\end{bmatrix},\\
\FF(\cdot,\uu)&= 
\begin{bmatrix}
\vspace{0.1cm}
\p[-(v\cdot\nabla_{x,y})v -w(v)\partial_z v ]\\
-b_{\delta_1}(v,\theta)
\end{bmatrix}, \qquad \qquad
\GG(\cdot,\uu)= 
\begin{bmatrix}
0\\
0
\end{bmatrix}.
\end{aligned}
\end{equation}
Finally, correspondingly to \eqref{eq:choice_X0X1}, we let 
\begin{equation*}
\XX_1\stackrel{{\rm def}}{=}\Hs^{2-\s_0,(q,\zeta)}\times H^{2-\s_1,(q,\zeta)}, \quad \text{and }\quad  
\XX_{\vartheta}\stackrel{{\rm def}}{=} [\XX_0,\XX_1]_{\vartheta} \ \text{ for }\ \vartheta\in (0,1).
\end{equation*}

Now, the proof of Theorems \ref{t:local_temp} and \ref{t:regularity_temp} follow the line of to the one of Theorems \ref{t:local} and \ref{t:regularity_temp} provided in Subsections \ref{ss:proof_local} and \ref{ss:proof_regularity1}, respectively. In the remaining part of this subsection, we only provide some comments on the main 
changes needed to adapt the arguments given there. More precisely, we have to extend to the current situation the nonlinear and linear estimates of Lemmas \ref{l:nonlinearities} and \ref{l:smr}.  

\begin{proof}[Proof of Theorem \ref{t:local_temp} -- Sketch]
We divide the proof into two steps.

\emph{Step 1: (Estimates of the nonlinearities). Lemma \ref{l:nonlinearities} holds with $(F,G,X_{\vartheta})$ replaced by $(\FF,\GG,\XX_{\vartheta})$ with $\beta$ as in Lemma \ref{l:nonlinearities} where $\delta=\delta_0$}. By Lemma \ref{l:nonlinearities}, to conclude this step, it is enough to show that 
\begin{equation}
\label{eq:claim_step_1_local_temp}
\|b_{\delta_1}(v,\theta)\|_{H^{-\s_1,(q,\zeta)}}\lesssim 
\|v\|_{\Hs^{-\s_0+2\beta,(q,\zeta)}}\|\theta\|_{H^{-\s_1+2\beta,(q,\zeta)}},
\end{equation}
where $\beta$ is as in Lemma \ref{l:nonlinearities}. Since $\s_1\in [\s_0,1+\s_0]$, by interpolation to prove \eqref{eq:claim_step_1_local_temp} it is enough to show that  
\begin{align}
\label{eq:claim_step_11_local_temp}
\|b_{j+\delta_0}(v,\theta)\|_{H^{-\s_1,(q,\zeta)}}&\lesssim \|v\|_{\Hs^{-\s_0+2\beta,(q,\zeta)}}
\|\theta\|_{H^{-j-\s_0+2\beta,(q,\zeta)}} \quad \text{ for }j\in \{0,1\}.
\end{align}
The case $j=1$ of \eqref{eq:claim_step_11_local_temp} follows from Lemma \ref{l:nonlinearities}. It remains to discuss the case $j=0$. Following the proof of Lemma \ref{l:nonlinearities}, we let $r\in (1,\infty)$ be either [$\frac{2}{r}=\delta+\frac{2}{q}$ if $\s<0$] or [$r=2$ if $\s=0$]. Since $L^{(q,\zeta)}\embed H^{-\s_0,(q,\zeta)}$ by Lemma \ref{l:anisotropic_B}\eqref{it:anisotropic_B4}, 
\begin{align*}
\|b_{1+\delta_0}(v,\theta)\|_{\XX_0}
&\lesssim \|v  \theta\|_{H^{-\s_0,(q,\zeta)}}+ 
\|w(v)\theta\|_{H^{-\s_0,(q,\zeta)}}\\
&\lesssim \|v  \theta\|_{L^{r}_{x,y}(L^{\zeta}_z)}+ 
\|w(v)\theta\|_{L^{r}_{x,y}(L^{\zeta}_z)}\\
&\leq \big(\|v\|_{L^{2r}_{x,y}(L^{\infty}_z)}+ 
\|w(v)\|_{L^{2r}_{x,y}(L^{\zeta}_z)}\big)\|\theta\|_{L^{2r}_{x,y}(L^{\zeta}_z)}\\	
&\lesssim \big(\|v\|_{L^{2r}_{x,y}(H^{1,\zeta}_z)}+ 
\|v\|_{W^{1,2r}_{x,y}(L^{\zeta}_z)}\big)\|\theta\|_{L^{2r}_{x,y}(L^{\zeta}_z)}.
\end{align*}
Arguing as in substeps 1a-1b in the proof of Lemma \ref{l:nonlinearities}, by Lemma \ref{l:anisotropic_B}\eqref{it:anisotropic_B4} one can check that the following estimates hold
\begin{align*}
\|v\|_{L^{2r}_{x,y}(H^{1,\zeta}_z)}+ 
\|v\|_{W^{1,2r}_{x,y}(L^{\zeta}_z)}&\lesssim \|v\|_{H^{-\s_0-2\beta,(q,\zeta)}},\\
\|\theta\|_{L^{2r}_{x,y}(L^{\zeta}_z)}&\lesssim \|\theta\|_{H^{-1-\s_0-2\beta,(q,\zeta)}},
\end{align*}
where $\beta$ is as in the statement of Lemma \ref{l:nonlinearities} with $\delta=\delta_0$.
The claim of Step 1 follows by collecting the previous estimates.

\emph{Step 2: (Stochastic maximal $L^p(L^q)$-regularity).
Let $\a\in [0,\frac{p}{2}-1)$, $\zeta\in [2,\infty)$ and $T\in (0,\infty)$.
Then for all progressively measurable processes
\begin{align*}
f=(f_v,f_{\theta})&\in L^p((0,T)\times \O,w_{\a};\Hs^{-\s_0,(q,\zeta)}\times H^{-\s_1,(q,\zeta)}),\\
g=(g_{v,n},g_{\theta,n})_{n\geq 1}&\in L^p((0,T)\times \O,w_{\a};\Hs^{1-\s_0,(q,\zeta)}(\ell^2)\times H^{1-\s_1,(q,\zeta)}(\ell^2)),
\end{align*}
there exists a unique strong solution $(v,\theta)\in L^p((0,T)\times \O,w_{\a};\Hs^{2-\s_0,(q,\zeta)}\times H^{2-\s_1,(q,\zeta)})$ to 
\begin{equation}
\label{eq:primitive_linear_zero_temp}
\left\{
\begin{aligned}
		&\partial_t v   
		=\nabla\cdot(a\cdot\nabla v) + (b\cdot\nabla)v + \fts(\cdot) v +\int_{0}^{\cdot}\nabla_{x,y}\theta (\cdot,z')\,\dd z' 	\\
		&\qquad \qquad \qquad \qquad\quad 
		+f_v+\sum_{n\geq 1}\big(\p [(\sigma_{n}\cdot\nabla) v]+ g_{v,n}\big)\, \dot{ \beta}_t^n, &\text{ on }&\Tor^3,\\
		&\partial_t \theta=\nabla (d\cdot\nabla \theta) + (k\cdot \nabla) \theta + f_{\theta}
		+ \sum_{n\geq 1}\big[(\nt_{n}\cdot\nabla) \theta+ g_{\theta,n}\big]\, \dot{ \beta}_t^n, &\text{ on }&\Tor^3,\\
&v(0,\cdot)=0,\qquad \theta(0,\cdot)=\theta_0, &\text{ on }&\Tor^3.
\end{aligned}
\right.
\end{equation}
Moreover, for all $\eta\in [0,\frac{1}{2})$, one has  
\begin{align*}
&\E\|(v,\theta)\|_{H^{\eta,p}(0,T,w_{\a};H^{2-\s_0-2\eta,(q,\zeta)}\times H^{2-\s_1-2\eta,(q,\zeta)})}^p \\
&\lesssim_{\eta,p} 
\E\|f\|_{L^p(0,T,w_{\a};H^{-\s_0,(q,\zeta)}\times H^{-\s_1,(q,\zeta)})}^p + \E\|g\|_{L^p(0,T,w_{\a};H^{1-\s_0,(q,\zeta)}(\ell^2)\times H^{1-\s_1,(q,\zeta)}(\ell^2))}^p,
\end{align*}
Finally, the above also holds true if $(0,T)$ is replaced by $(0,\tau)$ where $\tau$ is a stopping time taking values in $[0,T]$.
}
To prove the claim of Step 2 we use the method of continuity as in \cite[Proposition 3.13]{AV19_QSEE_2}. 
More precisely, let us consider the problem \eqref{eq:primitive_linear_zero_temp} where in the first SPDE the term
\begin{equation}
\label{eq:substitution_lambda}
\int_{0}^{\cdot}\nabla_{x,y}\theta (\cdot,z)\,\dd z \  \text{ is replaced by }\ \lambda 
\int_{0}^{\cdot}\nabla_{x,y}\theta (\cdot,z)\,\dd z\  \text{ where }  \ \lambda\in [0,1].
\end{equation}
Let $(v^{(\lambda)},\theta^{(\lambda)})$ be the strong solution of \eqref{eq:primitive_linear_zero_temp} in $L^p((0,T)\times \O,w_{\a};\Hs^{2-\s_0,(q,\zeta)}\times H^{2-\s_1,(q,\zeta)})$ with the above substitution. Now, in case $\lambda=0$, the statement of Step 2 is satisfied with $(v,\theta)$ replaced by $(v^{(0)},\theta^{(0)})$. Indeed, if $\lambda=0$, then the corresponding system of SPDEs is decoupled and therefore one can apply Lemma \ref{l:smr} and Theorem \ref{t:smr_anisotropic} to the first and second SPDE of the system, respectively. Now by the method of continuity \cite[Proposition 3.13]{AV19_QSEE_2} it remains to prove an a priori estimate for $(v^{(\lambda)},\theta^{(\lambda)})$ which is uniform in $\lambda$. To this end, we take advantage of the triangular structure of the system \eqref{eq:primitive_linear_zero_temp}. 

By Theorem \ref{t:smr_anisotropic}, a strong solution to the second SPDE in \eqref{eq:primitive_linear_zero_temp} satisfies
\begin{align}
\label{eq:smr_theta}
\E\|\theta^{(\lambda)}\|_{L^p(0,T,w_{\a};H^{2-\s_1,(q,\zeta)})}^p
&\lesssim 
\E\|f_{\theta}\|_{L^p(0,T,w_{\a};H^{-\s_1,(q,\zeta)})}^p\\ 
\nonumber
&+ \E\|g_{\theta}\|_{ L^p(0,T,w_{\a};H^{1-\s_1,(q,\zeta)}(\ell^2))}^p,
\end{align}
where $g_{\theta}\stackrel{{\rm def}}{=}(g_{\theta,n})_{n\geq 1}$ and the implicit constant is independent of $(\lambda,f,g)$.

Recall that $\s_1\leq 1+\s_0$. Hence, by \eqref{eq:smr_theta} and Lemma \ref{l:smr} applied to the first SPDE in \eqref{eq:primitive_linear_zero_temp}, it suffices to show that 
\begin{equation}
\label{eq:smr_theta_claim_implication}
\E\Big\|\int_{0}^{\cdot} \nabla_{x,y}\theta^{(\lambda)}(\cdot,z)\,\dd z\Big\|_{L^p(0,T,w_{\a};H^{-\s_0,(q,\zeta)})}^p
\lesssim
\E\|\theta^{(\lambda)}\|_{L^p(0,T,w_{\a};H^{1-\s_0,(q,\zeta)})}^p
\end{equation}
with implicit constant independent of $(\lambda,f,g)$. The prove the above,  note that, for all $\phi\in H^{1-\s_0,(q,\zeta)}$,
\begin{align}
\label{eq:estimate_phi_integral_temp}
\Big\|\int_{0}^{\cdot} \nabla_{x,y}\phi(\cdot,z)\,\dd z \Big\|_{H^{-\s_0,(q,\zeta)}}
&\stackrel{(i)}{\lesssim} 
\Big\|\int_{0}^{\cdot} \nabla_{x,y}\phi(\cdot,z)\,\dd z \Big\|_{H^{-\s_0,q}_{x,y}(L^\zeta_z)}\\
\nonumber
&\stackrel{(ii)}{\lesssim} \| \nabla_{x,y}\phi\|_{H^{-\s_0,q}_{x,y}(L^\zeta_z)}\\
\nonumber
&\lesssim \| \phi\|_{H^{1-\s_0,q}_{x,y}(L^\zeta_z)}\\
\nonumber
&\stackrel{(iii)}{\lesssim} \|\phi\|_{H^{1-\s_0,(q,\zeta)}},
\end{align}
where in $(i)$ and $(iii)$ we used Lemma \ref{l:anisotropic_B}\eqref{it:anisotropic_B3} as $\s_0\in [0,1)$, and in $(ii)$ that the integral operator $L^\zeta_z \ni\phi\mapsto \int_{0}^{\cdot} \phi(z) \,\dd z \in L^\zeta_z$ is bounded.
The previous estimate yields \eqref{eq:smr_theta_claim_implication}, and this concludes the proof of Step 2.
\end{proof}

\begin{proof}[Proof of Theorem \ref{t:regularity_temp}]
From Steps 1 and 2, the proof of Theorem \ref{t:regularity_temp} follows verbatim from the one of Theorem \ref{t:regularity1} in Subsection \ref{ss:proof_regularity1}.
\end{proof}

\subsubsection{Proof of Theorem \ref{t:global_temp}}
To prove Theorem \ref{t:global_temp} as a first step, we state the special case of the non-isothermal version of Theorem \ref{t:serrin}\eqref{it:serrin1}: 

Let $((v,\theta),\tau)$ be the $(p,\s_0,\s_1,\a_{\crit},q)$-solution of \eqref{eq:primitive_temperature} provided by \eqref{t:local_temp}.
If $\pz\in (2,\infty)$ satisfies 
\begin{equation}
\label{eq:pz_temp}
\frac{1}{\pz}+\frac{1}{q}+\frac{\s_0}{2}=1,
\end{equation}
then for all $0<s<T<\infty$ one has
\begin{equation}
\label{eq:serrin_temp}
\P\Big(s<\tau<T,\,\|v\|_{L^{\pz}(s,\tau;H^{2-\s_0,(q,2)})}+
\|\theta\|_{L^{\pz}(s,\tau;H^{2-\s_1,(q,2)})}<\infty\Big)=0.
\end{equation}
As for Theorem \ref{t:serrin}, the above follows from the Serrin criteria of \cite[Theorem 4.11]{AV19_QSEE_2} and Steps 1 and 2 of Theorem \ref{t:local_temp}. As in Theorem \ref{t:serrin}, one can prove a more general version of \eqref{eq:serrin_temp} for \eqref{eq:primitive_temperature}. However, we do not state it as it is not needed below.

\begin{proof}[Proof of Theorem \ref{t:global_temp} -- Sketch]
Let $((v,\theta),\tau)$ be the $(p,\s_0,\s_1,\a_{\crit},q)$-solution of \eqref{eq:primitive_temperature}. 
Since the proof is analogue to the one the corresponding results in the non-isothermal case in Subsection \ref{ss:global}, we only give a sketch of the proof. 

We divide it into four steps. Below $0<T<\infty$ is fixed. Let us begin by noticing that, as in \eqref{eq:regularity_of_partialz}, we have
$$
\partial_z v \in L^2_{\loc}(0,\tau;H^1) \text{ a.s.\ }
$$ 

\emph{Step 1: (Intermediate energy estimate). There exists $C_1$ such that for all, $s\in (0,T]$, $\lambda>1$ and $\Gamma\in \F_s$ satisfying $\tau>s$ a.s.\ on $\Gamma$,} 
\begin{align}
\label{eq:temperature_L2_balance_proof}
\E\Big[ \sup_{r\in [s,\tau\wedge T)} \one_{\Gamma} \|\theta(r)\|_{L^2}^2\Big] 
+\E\int_0^{\tau\wedge T}\one_{\Gamma} \|\theta(r)\|_{H^1}^2\,\dd r\qquad\qquad\qquad\quad &\\
\nonumber
\leq C_1\big(1+\E[\one_{\Gamma}\|v(s)\|_{L^2}^2]+\E[\one_{\Gamma}\|\theta(s)\|_{L^2}^2]\big),& \\
\label{eq:intermediate_estimate_proof_temp}
\max_{j\in \{0,1\}}
\P\Big(\Gamma\cap \Big\{\sup_{t\in [s,\tau\wedge T)} \|\partial_z^j v(s)\|_{L^2}^2
+ \int_{s}^{ \tau\wedge T}\|\partial_z^j v(s)\|_{H^1}^2\,\dd t \geq \lambda\Big\}\Big)\qquad\qquad &\\
\nonumber
\leq \frac{C_1}{\log(\lambda)} \big(1+\E\big[\one_{\Gamma} \|v(s)\|_{H^1}^{4}\big]
+\E\big[\one_{\Gamma} \|\theta(s)\|_{L^2}^{2}\big]\big).&
\end{align}
The estimate \eqref{eq:temperature_L2_balance_proof} follows as in \cite[Lemma 5.2]{Primitive1} by using the standard cancellation $\int_{\T^3} [(u\cdot\nabla)\theta]\,\theta\,\dd \x =0$. The latter complements the $L^2$-energy balance of Lemma \ref{l:energy_estimate} which also holds in this case, by replacing the RHS\eqref{eq:velocity_L2_balance_proof} by  
$\lesssim\big(1+\E[\one_{\Gamma}\|v(s)\|_{L^2}^2]+\E[\one_{\Gamma}\|\theta(s)\|_{L^2}^2]\big)$.
Thus, \eqref{eq:intermediate_estimate_proof_temp} follows as in Lemma \ref{l:main_intermediate_estimate_wt} with minor modifications.

\emph{Step 2: (Strong energy estimate). There exists $C_2$ such that for all, $s\in (0,T]$, $\lambda>e$ and $\Gamma\in \F_s$ satisfying $\tau>s$ a.s.\ on $\Gamma$,
\begin{align*}
&\P\Big(\Gamma\cap \Big\{\sup_{t\in [s,\tau\wedge T)}\|v(t)\|_{B^{2/q}_{(q,2),\pz}}^{\pz}
+ \int_{s}^{ \tau\wedge T}\|v(t)\|_{H^{2-\s_0,(q,2)}_{(q,2),\pz}}^{\pz} \,\dd t \geq \lambda\Big\}\Big)\\
&\leq \frac{C_2}{\log\log(\lambda)}\Big(1+\E[\one_{\Gamma}\|v(s)\|_{B^{2/q}_{(q,2),\pz}}^{\pz}] +\E[\one_{\Gamma}\|v(s)\|_{H^1}^4]+\E[\one_{\Gamma}\|\theta(s)\|_{L^2}^2]\Big)&
\end{align*}
where $\pz$ is as in \eqref{eq:pz_temp}.}
The estimate of Step 2 follows the line of the proof of Lemma \ref{l:strong_energy_estimate} using the estimates of Step 1. 
We content ourselves to check that from Step 1 it follows that, for all $\lambda> 1$,
\begin{align}
\label{eq:claim_temp_global}
\P\Big(\Gamma \cap \Big\{
\int_{s}^{\tau\wedge T}\Big\|
\int_{-h}^{\cdot}\nabla_{x,y} \theta(t,\cdot,z)\,\dd z\Big\|_{H^{-\s_0,(q,2)} }^{\pz}\, \dd t \geq \lambda\Big\} \Big)&\\
\nonumber
\lesssim \frac{1+\E\big[\one_{\Gamma}\big(\|v(s)\|^2_{L^2}+\|\theta(s)\|_{L^2}^2\big)\big]}{\log(\lambda)}&
\end{align}
where the implicit constant independent of $(s,\Gamma,v_0,\theta_0)$. 

By \eqref{eq:estimate_phi_integral_temp}, to obtain \eqref{eq:claim_temp_global} it is enough to prove that, for $\lambda>e$,
\begin{align*}
\P\Big(\Gamma \cap \Big\{
\int_{s}^{\tau\wedge T}\|\theta(t)\|_{H^{1-\s_0,q}_{x,y}(L^2_z) }^{\pz}\, \dd t\geq \lambda\Big\} \Big)
\lesssim \frac{1+\E\big[\one_{\Gamma}\big(\|v(s)\|_{L^2}^2+\|\theta(s)\|_{L^2}^2\big)\big]}{\log(\lambda)}.
\end{align*}
The above follows from \eqref{eq:temperature_L2_balance_proof} and the following embedding: For all $t\in (0,\infty]$
\begin{align*}
L^{\infty}(0,t;L^{2})\cap L^2(0,t;H^1)
&\embed L^{\pz}(0,t;H^{2/\pz})  & &\text{(Interpolation)}\\
&\embed L^{\pz}(0,t;H^{1-\s_0,(q,2)})& &\text{(Sobolev emb.\ and \eqref{eq:pz_temp}).}
\end{align*}
Hence \eqref{eq:claim_temp_global} is proved and this concludes the proof of Step 2.

\emph{Step 3: (Strong energy estimate). There exists $C_3$ such that for all $s\in (0,T]$, $\lambda>e^e$ and $\Gamma\in \F_s$ satisfying $\tau>s$ a.s.\ on $\Gamma$,
\begin{align*}
&\P\Big(\Gamma\cap \Big\{\sup_{t\in [s,\tau\wedge T)}\|\theta(t)\|_{B^{2/q-\s_1+\s_0}_{(q,2),\pz}}^{\pz}+ \int_{s}^{ \tau\wedge T}\|\theta(t)\|_{H^{2-\s_1,(q,2)}_{(q,2),\pz}}^{\pz}\,\dd t \geq \lambda\Big\}\Big)\\
&\leq \frac{C_3}{\log\log\log(\lambda)}\Big(1+\E[\one_{\Gamma}\|v(s)\|_{B^{2/q}_{(q,2),\pz}}^{\pz}] +\E[\one_{\Gamma}\|v(s)\|_{H^1}^4]\\
&\qquad \qquad\qquad \qquad\qquad
+\E[\one_{\Gamma}\|\theta(s)\|_{B^{2/q-\s_1+\s_0}_{(q,2),\pz}}^{\pz}]+\E[\one_{\Gamma}\|\theta(s)\|_{L^2}^2]\Big)
\end{align*}
where $\pz$ are as in \eqref{eq:pz_temp}.} As in the previous step, the proof follows the one of Lemma \ref{l:strong_energy_estimate} using the estimates of Steps 1-2, and replacing the use of stochastic maximal $L^p$-regularity estimate of Lemma \ref{l:smr} with the one of Step 1 in Theorem \ref{t:local_temp}. 
The only addition needed is to estimate the term
$$
\E\int_{\eta}^{\xi}\|b_{\delta_1}(v,\theta)\|_{H^{-\s_1,(q,2)}}^{\pz}\,\dd t .
$$
Here $b_{\delta_1}$ is as in \eqref{eq:def_b_delta_one}, $\eta,\xi$ are suitable stopping times such that $s\leq\eta\leq \xi\leq \tau_j$ a.s.\ for some $j\geq 1$ and
$$
\tau_j\stackrel{{\rm def}}{=}
\left\{
\begin{aligned}
&\inf\Big\{t\in [s,\tau)\,:\, 
\big(\|(v(t),\theta(t))\|_{B^{2/q}_{(q,2),\pz}\times B^{2/q-\delta_1+\delta_0}_{(q,2),\pz}}\\
&\qquad\qquad\qquad\quad + 
\|(v,\theta)\|_{L^{p_0}(s,t;H^{2-\s_0,(q,2)}\times H^{2-\s_1,(q,2)})}\big)\geq j
\Big\} &  &\text{on }\Gamma,\\
&s & &\text{on }\O\setminus\Gamma.
\end{aligned}
\right.
$$
cf.\ \eqref{eq:localization_strong_estimate}. 
By \eqref{eq:claim_step_1_local_temp} of Step 1 in the proof of Theorem \ref{t:local_temp} we have, a.e.\ on $(0,\tau)\times \O$,
$$
\|b_{\delta_1}(v,\theta)\|_{H^{-\s_1,(q,2)}}\lesssim 
\|v\|_{\Hs^{-\s_0+2\beta,(q,2)}}\|\theta\|_{H^{-\s_1+2\beta,(q,2)}}. 
$$
where $\beta=\frac{1}{2}+\frac{\s_0}{4}+\frac{1}{2q}$ (cf.\ Step 1 in the proof of Theorem \ref{t:local_temp}).
As in \eqref{eq:interpolation_B_H}, 
$$
(B^{2/q-\s_1+\s_0}_{(q,2),\pz},H^{2-\s_1,(q,2)})_{1/2,1}\embed H^{-\s_1+2\beta,(q,2)}.
$$
Hence, for all $\varepsilon>0$ and a.e.\ on $(0,\tau)\times \O$, 
$$
\|b_{\delta_1}(v,\theta)\|_{H^{-\s_1,(q,2)}}\leq \varepsilon \|\theta\|_{H^{2-\s_1,(q,2)}} + C_{\varepsilon}
\|v\|_{H^{-\s_0+2\beta,(q,2)}}^2\|\theta\|_{B^{2/q-\s_1+\s_0}_{(q,2),\pz}}. 
$$

Therefore, arguing as in the proof of Lemma \ref{l:strong_energy_estimate}, to prove the claim of Step 3 it is enough to check that, for all $\lambda>e$,
\begin{align}
\label{eq:last_claim_proof_temp_3}
&\P\Big(\Gamma \cap \Big\{\int_s^{\tau\wedge T} \|v(t)\|_{H^{-\s_0+2\beta,(q,2)}}^{2\pz}\,\dd t \geq \lambda \Big\}\Big)\\
\nonumber
&\lesssim \frac{1}{\log\log(\lambda)}\Big(1+\E[\one_{\Gamma}\|v(s)\|_{B^{2/q}_{(q,2),\pz}}^{\pz}] +\E[\one_{\Gamma}\|v(s)\|_{H^1}^4]+\E[\one_{\Gamma}\|\theta(s)\|_{L^2}^2]\Big). 
\end{align}
To prove the above, note that, due to $
(B^{2/q}_{(q,2),\pz},H^{2-\s_0,(q,2)})_{1/2,1}\embed H^{-\s_0+2\beta,(q,2)}$, the H\"older inequality yields
$$
L^{\infty}(0,t;B^{2/q}_{(q,2),\pz})\cap L^{\pz}(0,t;H^{-\s_0+2\beta,(q,2)})\embed L^{2\pz}(0,t;H^{-\s_0+2\beta,(q,2)}) \ \text{ for } \  t>0,
$$
and embedding constant independent of $t>0$.
Hence, the estimate \eqref{eq:last_claim_proof_temp_3} follows from Step 2 of the current proof.

\emph{Step 4: Conclusion}. By \eqref{eq:serrin_temp} and the estimate of Steps 2-3, as in the proof of Theorem \ref{t:global}, it readily follows that $\tau=\infty$ a.s.
Analogously, from the estimates of Step 2-3, the continuity w.r.t.\ the initial data follows by repeating almost verbatim the proof of Theorem \ref{t:continuous_dependence} by using the stochastic maximal $L^p$-regularity estimate of Step 2 in the proof of Theorem \ref{t:local_temp}.
\end{proof}

\section{Stratonovich formulation}
\label{ss:stratonovich}
In this subsection, we show how the results presented in Section \ref{s:statement} also apply to the PEs reformulated as in \eqref{eq:primitive2} with transport noise in Stratonovich form, i.e.\  
\begin{equation}
\label{eq:primitive_stra}
\left\{
\begin{aligned}
		&\partial_t v  		
		= \p\big[\Delta v- (v\cdot \nabla_{x,y})v- w(v)\partial_z v  + f(\cdot,v,\nabla v)\big]\\
		&\qquad \qquad\qquad\qquad\qquad\qquad\qquad
		 +\sum_{n\geq 1} \p[ (\sigma_{n}\cdot\nabla) v]\circ \dot{ \beta}_t^n, &\text{on }&\Tor^3,\\
&v(\cdot,0)=v_0,&\text{on }&\Tor^3,
\end{aligned}
\right.
\end{equation}
where $w(v)=-\int_{0}^{\cdot} \nabla_{x,y}\cdot v (\cdot,z')\,\dd z'$, cf.\ \eqref{eq:def_w}.
As in the previous sections, in the above, we could also consider global Lipschitz lower-order terms in the stochastic perturbation. For simplicity of exposition, we do not pursue this here. 

Below we list the main assumptions used in the subsection.

\begin{assumption} We say that Assumption \ref{ass:primitive_stra}$(p,\s,q)$ holds if the following conditions are satisfied.
\label{ass:primitive_stra}
\begin{enumerate}[{\rm(1)}]
\item $\delta\in [0,1)$ and either $[\, p\in (2,\infty)\text{ and }q\in [2,\infty)\, ]$ or $[\, p=q=2\, ]$.
\item\label{it:ass_primitive_stra1} For all $n\geq 1$ the following map is $\Progress\otimes \Borel(\T^3)$-measurable
\begin{align*}
\sigma_n&=(\sigma_{n,\xi})_{\xi\in \{x,y,z\}}:\R_+\times \O\times \T^3\to \R^3.
\end{align*}

\item\label{it:ass_primitive_stra2} For all $n\geq 1$, the maps
$$
(\sigma_{n,\xi})_{\xi\in \{x,y\}}   \text{ is independent of }z\in \T_z.
$$
\item\label{it:ass_primitive_stra3} There exist $M,\g>0$ such that $\g> 1-\s$ and 
$$
\|(\sigma_n(t,\cdot))_{n\geq 1}\|_{C^{\g}(\Tor^3;\ell^2)}\leq M \ \  \text{ a.e.\ on } \R_+\times \O.
$$
\item\label{it:ass_primitive_stra4} There exists $N>0$ such that 
$$
\|(\nabla \cdot \sigma_n)_{n\geq 1}\|_{L^{\infty}(\T^3;\ell^2)}\leq N  \ \  \text{ a.e.\ on } \R_+\times \O.
$$
\item\label{it:ass_primitive_stra5} For all $n\geq 1$, the map $f:\R_+\times \O\times \T^3\times \R^2\times \R^{2\times 3}\to \R^2$
is $\Progress\otimes \Borel(\T^3\times \R^2\times \R^{3\times 2})$-measurable and
$f(\cdot,0,0)
\in L^{\infty}(\R_+\times \O\times \T^2;\R^2).
$

Finally, there exists $K >0$ for which the following holds for all $\xi,\xi'\in \R^3$, $\eta,\eta'\in \R^{3\times 2}$ and a.e.\ on $\R_+\times \O\times \T^3$
\begin{align*}
|f(\cdot,\xi,\eta)-f(\cdot,\xi',\eta')|&\leq K\big(|\xi-\xi'|+ |\eta-\eta'|\big).
\end{align*}
\end{enumerate}
\end{assumption}

An example of transport noise satisfying Assumption \ref{ass:primitive_stra}\eqref{it:ass_primitive_stra1}-\eqref{it:ass_primitive_stra4} and related to the Kolmogorov spectrum is given in 
Example \ref{ex:Kolgomorov}.

The above should be compared with Assumption \ref{ass:primitive}. The main addition w.r.t.\ to the latter assumption is in Assumption \ref{ass:primitive_stra}\eqref{it:ass_primitive_stra4} where the divergence operator is taken in the distributional sense. Such assumption already appeared in other situations, see e.g.\ \cite[eq.\ (6.2)]{BM13_unbounded} and \cite[condition (B1)]{MR05}. Typically in stochastic fluid mechanics, one enforces $\nabla\cdot \sigma_n\equiv 0$ and therefore Assumption \ref{ass:primitive_stra}\eqref{it:ass_primitive_stra4} is satisfied. However, in the previous situation  Assumption \ref{ass:primitive_stra}\eqref{it:ass_primitive_stra2} implies $\partial_x \sigma_{n,x}+\partial_{y}\sigma_{n,y}=\partial_z \sigma_{n,z}=0$, as it follows by applying $\int_{\T_z}\cdot\,\dd z$ to the equality $\nabla\cdot \sigma_n= 0$.

As usual, we formally rewrite the Stratonovich SPDE \eqref{eq:primitive_stra} as an It\^o SPDE in which new correcting terms appear. Following \cite[Section 8]{Primitive1}, formally we obtain
\begin{align}
\label{eq:stratonovich_correction1}
\sum_{n\geq 1}\p [(\sigma_n\cdot\nabla) v]\circ \dot{\beta}_t^n
&=
\sum_{n\geq 1}\p [(\sigma_n\cdot\nabla) v]\, \dot{\beta}_t^n\\
\nonumber
&+\frac{1}{2}\sum_{n\geq 1}\p\Big[(\sigma_n\cdot\nabla) \p [(\sigma_n\cdot\nabla) v]\Big].  
\end{align} 
Next, we rewrite the last term on the RHS\eqref{eq:stratonovich_correction1}. Recalling that $\p [(\sigma_n\cdot\nabla) v]= -\ph \wt{p}_n +(\sigma_n\cdot\nabla) v$ where $\ph \wt{p}_n= \qq [(\sigma_n\cdot\nabla) v]$ does not depend on $z$ (cf.\ \eqref{eq:pressure2}), we obtain 
\begin{align}
\label{eq:stratonovich_correction2}
\p\Big[(\sigma_n\cdot\nabla) \p [(\sigma_n\cdot\nabla) v]\Big]
&= \p\Big[\nabla \cdot ([\sigma_n\otimes \sigma_n]\cdot \nabla v) - (\nabla\cdot \sigma_n) [(\sigma_n\cdot\nabla) v]\Big]\\
\nonumber
& - \p\Big[ \nabla \cdot (\ph \wt{p}_n \otimes \sigma_n) - (\nabla \cdot \sigma_n) \ph \wt{p}_n\Big].
\end{align} 
Note that \eqref{eq:stratonovich_correction1}-\eqref{eq:stratonovich_correction2} shows that \eqref{eq:primitive_stra} is a special case of \eqref{eq:primitive2} with the choice:
\begin{equation}
\label{eq:def_abh_stra}
a=\frac{1}{2}\big[ \mathrm{Id}+\sum_{n\geq 1}(\sigma_n\otimes \sigma_n) \big], \ \ 
 b= -\frac{1}{2} \sum_{n\geq 1} (\nabla \cdot \sigma_n) \sigma_n, \ \ 
  b_{0,n}=\frac{1}{2} \nabla\cdot\sigma_n, 
  \ \ h=1.
\end{equation}
In particular the notion of $(p,\a,\s,q)$-solution of Definition \ref{def:solution} extends to \eqref{eq:primitive_stra}, where $\a\in [0,\frac{p}{2}-1)$ if $p>2$ or $\a=0$ otherwise. 

The following is a consequence of the results in Section \ref{s:statement}.

\begin{corollary}[Local well-posedness, regularity, blow-up criteria and global well-posedness -- Stratonovich noise]
\label{cor:stratonovich}
Let the Assumption \ref{ass:primitive_stra}$(p,\s,q)$. Suppose that \eqref{eq:assumptions_critical_setting} holds. 
\begin{enumerate}[{\rm(1)}]
\item\label{it:cor_stratonovich1} Theorems \ref{t:local}, \ref{t:regularity1}, \ref{t:serrin}, Proposition \ref{prop:continuity}, Corollary \ref{cor:compatibility}, hold with \eqref{eq:primitive} replaced by \eqref{eq:primitive_stra}.
\item\label{it:cor_stratonovich2} If $\partial_z \sigma_{n,z}\in L^{\infty}(\T^3)$ a.s.\ for all $n\geq 1$ and for some constant $K>0$ 
$$
 \big\|(\partial_{z}\sigma_n(t,\om,\cdot))_{n\geq 1}\big\|_{L^{\infty}(\T^3;\ell^2)}
 \leq K\ \ \text{ a.e.\ on }\R_+\times \O,
$$
then Theorems \ref{t:global} and \ref{t:continuous_dependence} hold with \eqref{eq:primitive} replaced by \eqref{eq:primitive_stra}.
\end{enumerate}
\end{corollary}

A version of Theorem \ref{t:regularity2} also holds for \eqref{eq:primitive_stra}.

\begin{proof}
To begin, note that Assumption \ref{ass:primitive_stra} ensures that Assumption \ref{ass:primitive} is satisfied with the choice \eqref{eq:def_abh_stra}. Hence 
\eqref{it:cor_stratonovich1} follows. To prove 
\eqref{it:cor_stratonovich2}, it is enough to note that $
 \big\|(\partial_{z}\sigma_n)_{n\geq 1}\big\|_{L^{\infty }(\T^3;\ell^2)}
 \leq K$ implies that Assumption \ref{ass:global} hold with $a$ as in \eqref{eq:def_abh_stra}. 
\end{proof}

As in Section \ref{s:non_isothermal}, Theorems \ref{t:local_temp}-\ref{t:regularity_temp} and \ref{t:global_temp} holds in the case the temperature equation in \eqref{eq:primitive_temperature} is in Stratonovich formulation. More precisely, the $\theta$-equation in \eqref{eq:primitive_temperature} is replaced by
\begin{equation}
\label{eq:temp_stratonovich}
\partial_t \theta   -\Delta \theta= \big[-(v\cdot\nabla_{x,y})\theta-w(v)\partial_z \theta\big]
+\sum_{n\geq 1}\big[(\nt_{n}\cdot\nabla) \theta\big] \circ \dot{ \beta}_t^n, \ \ \ \text{ on }\Tor^3.
\end{equation}
Arguing as in \cite[Subsection 8.1]{Primitive1},  at least formally, one has (see eq.\ (8.4) in \cite{Primitive1})
\begin{align*}
\sum_{n\geq 1}\big[(\nt_{n}\cdot\nabla) \theta\big] \circ \dot{ \beta}_t^n
=\nabla \cdot (d\cdot \nabla \theta)+ (k\cdot\nabla) \theta+
\sum_{n\geq 1}\big[(\nt_{n}\cdot\nabla) \theta\big] \, \dot{ \beta}_t^n,
\end{align*}
where 
$$
d=\frac{1}{2} \big(\mathrm{Id}+\sum_{n\geq 1} \nt_n \otimes \nt_n\big), 
\quad \text{and }\quad b=- \frac{1}{2}\sum_{n\geq 1}(\nabla \cdot \nt_n) \nt_n. 
$$
In particular, the system formed by \eqref{eq:primitive_stra} and \eqref{eq:temp_stratonovich} is in the form of \eqref{eq:primitive_temperature}, analyzed in Section \ref{s:non_isothermal}. 
In particular, due to the results in Subsection \ref{ss:statement_non_isothermal}, one can obtain an extension of Corollary \ref{cor:stratonovich} to the non-isothermal case \eqref{eq:primitive_stra} under the additional assumptions $(\nt_n)_{n\geq 1}\in C^{\g_1}(\T^3;\ell^2)$ for some $\g_1>0$ and 
$(\nabla\cdot \nt_n)_{n\geq 1}\in L^{\infty}(\T^3;\ell^2)$.

\appendix 

\section{Some useful results on anisotropic spaces}
\label{app:anisotropic}
In this appendix, we collect some results on anisotropic Bessel potential and Besov spaces which are used in the main body of this work and the subsequent Appendix \ref{app:smr}. 
Below we only discuss function spaces in the periodic setting. Most of the results below also extend to the whole space case.

We begin with the main definitions. For $\q=(q_1,\dots,q_d)\in (1,\infty)^d$, we denote by $L^{\q}(\T^d)$ the set of (equivalence classes of) measurable maps $f:\T^d\to \R$ such that
$$
\|f\|_{L^{\q}(\T^d)}\stackrel{{\rm def}}{=} 
\Big(\int_{\T_{1}}\Big( \dots \Big(\int_{\T_d} |f(x_1,\dots,x_d)|^{q_d}\,\dd x_d\Big)^{q_{d-1}/q_d}\dots\Big)^{q_{2}/q_{1}}\,\dd x_1\Big)^{1/q_1}<\infty,
$$
where $\T^d=\T_1\times\dots \times\T_d$. As in the main body of this manuscript, we used subscripts in each coordinate to make explicit their different role.
Equivalently $L^{\q}(\T^d)$ can be defined as an iterated Bochner space $L^{q_1}(\T_1;\dots;L^{q_d}(\T_d)\dots)$, cf.\ \cite[Chapter 1]{Analysis1}.
The anisotropic Sobolev space $H^{s,\q}(\T^d)$ can be defined as the set of all $f\in \D'(\T^d)$  such that $(1-\Delta)^{s/2} f\in L^{\q}(\T^d)$ endowed with the natural norm:
$$
\|f\|_{H^{s,\q}(\T^d)}\stackrel{{\rm def}}{=}\|(1-\Delta)^{s/2} f\|_{L^{\q}(\T^d)}.
$$
It is known that classical tools from harmonic analysis such as the Mihlin multiplier theorem (see e.g.\ \cite[Subsection 5.2]{G14_classical} or \cite[Subsection 5.5]{Analysis1})
extends to the case of $L^{\q}$-spaces. Indeed, this is a straightforward consequence of the Mihlin multiplier theorem  in the weighted setting (see e.g.\ \cite[Theorem 7.1]{L21_JGA}) and Rubio de Francia extrapolation (see e.g.\ \cite{CUMP11,N23_JFA} or \cite[Theorem 6.2]{K21_rubio}). Further details on anisotropic function spaces can be found in \cite{L21_JAT} and the references therein.

In the following, we collect some useful properties of the above-introduced spaces.

\begin{lemma}
\label{l:anisotropic_B}
The following hold.
\begin{enumerate}[{\rm(1)}]
\item\label{it:anisotropic_B1} For all $s	\in \R$ and $\q\in (1,\infty)^d$ we have $C^{\infty}(\T^d)\embed H^{s,\q}(\T^d)$ densely.
\item\label{it:anisotropic_B2} For all $s_0,s_1\in \R$, $\q_0,\q_1\in (1,\infty)^d$ and $\theta\in (0,1)$,  we have
$$
[H^{s_0,\q_0}(\T^d),H^{s_1,\q_1}(\T^d)]_{\theta}= H^{s_{\theta},\q_{\theta}} (\T^d)
$$
where $s_{\theta}=(1-\theta)s_0+\theta s_1$ and $\frac{1}{q_{\theta,i}}=\frac{1-\theta}{q_{0,i}}+\frac{1-\theta}{q_{1,i}}$ for all $i\in \{1,\dots,d\}$.
\item\label{it:anisotropic_B3} For all $\q\in (1,\infty)^d$ and  $s, s_i\geq 0$ satisfying $\sum_{i=1}^d s_i=s$, we have
\begin{align*}
H^{s,\q}(\T^d)
&\embed H^{s_1,q_1}(\T_1;\dots;H^{s_d,q_d}(\T_d)\dots),\\
H^{-s_1,q_1}(\T_1;\dots;H^{-s_d,q_d}(\T_d)\dots)
&\embed 
H^{-s,\q}(\T^d).
\end{align*}
\item\label{it:anisotropic_B4} For all $s_0\geq s_1$, $\q_0,\q_1\in (1,\infty)^d$ satisfying $q_{0,i}\geq q_{1,i}$ for all $i\in \{1,\dots,d\}$ and $s_0-\sum_{i=1}^d \frac{1}{q_{0,i}}\geq s_1-\sum_{i=1}^d \frac{1}{q_{1,i}}$, the Sobolev embedding holds:
$$
H^{s_0,\q_0}(\T^d)\embed H^{s_1,\q_1}(\T^d).
$$
\end{enumerate}
\end{lemma}

The above is known to experts. For the reader's convenience, we include a proof.

\begin{proof}
\eqref{it:anisotropic_B1}: It is a straightforward consequence of the definition.

\eqref{it:anisotropic_B2}: The proof is analog to the one in the isotropic case (see e.g.\ \cite[Chapter 6]{BeLo} or \cite[Chapter 5]{Analysis1}) by employing the weighted Mihlin's multiplier theorem \cite[Theorem 7.1]{L21_JGA}  and Rubio de Francia extrapolation and \cite[Theorem 6.2]{K21_rubio}.

\eqref{it:anisotropic_B3}: For the first embedding of \eqref{it:anisotropic_B1}, it is enough to show the boundedness on $L^{\q}(\T^d)$ of  the operator
$\big[\,\prod_{i=1}^d (1-\partial^2_{x_i})^{{s_i}/2}\,\big] (1-\Delta)^{-{s}/2}$. 
Therefore, it is enough to prove $(1-\partial_{x_i}^2)^{s_i/2}(1-\Delta)^{-s_i/2}$ is bounded on $L^{\q}(\T^d))$ for all $i\in \{1,\dots,d\}$. As above,  the latter follows by combining \cite[Theorem 7.1]{L21_JGA} and \cite[Theorem 6.2]{K21_rubio} (the Mihlin condition can be checked by using the argument in \cite[Example 6.2.9]{G14_classical}).

For the second embedding of \eqref{it:anisotropic_B1}, we argue by duality. Firstly, by duality of Bochner's spaces \cite[Theorem 1.3.10]{Analysis1} and the fact that $(1-\Delta)^{s/2}:H^{s,\q}(\T^d)\to L^{\q}(\T^d)$ is an isomorphism for $s\in \R$, it follows that $(H^{s,\q}(\T^d))^*=H^{-s,\vec{q'}}(\T^d)$ where $\vec{q'}=(q_1',\dots,q_d')$ where $\frac{1}{q_i}+\frac{1}{q_i'}=1$. Moreover, \cite[Proposition 5.6.7]{Analysis1}, we have 
$$
\big(H^{s_1,q_1}(\T_1;\dots;H^{s_d,q_d}(\T_d)\dots)\big)^*=H^{-s_1,q_1'}(\T_1;\dots;H^{-s_d,q_d'}(\T_d)\dots).
$$
Hence the second embedding in \eqref{it:anisotropic_B3} follows from the first one, duality and \eqref{it:anisotropic_B1}.

\eqref{it:anisotropic_B4}: Since $(1-\Delta)^{s}:H^{s+r,\q}(\T^d)\to H^{r,\q}(\T^d)$ is an isomorphism for all $s,r\in\R$ and $\q\in (1,\infty)$, it is enough to consider $s_1=0$. Note that, if $s_1=0$, then $s_0\geq \sum_{i=1}^d (\frac{1}{q_{0,i}}-\frac{1}{q_{1,i}})\geq 0$. Moroever, by assumption there exist $s_1,\dots,s_d\geq 0$ such that  $s_0=\sum_{i=1}^d s_i$ and $s_i\geq \frac{1}{q_{0,i}}-\frac{1}{q_{1,i}}\geq 0$.
Then, by Lemma \ref{l:anisotropic_B}\eqref{it:anisotropic_B3},
\begin{align*}
H^{s_0,\q_0}(\T^d)
&\embed H^{r_1,q_{0,1}}(\T_1;\dots;H^{r_d,q_{0,d}}(\T_d)\dots)\\
&\stackrel{(i)}{\embed} L^{q_{1,1}}(\T_1;\dots;L^{q_{1,d}}(\T_d)\dots)=L^{\q_1}(\T^d), 
\end{align*}
where in $(i)$ we used the Banach-valued Sobolev embeddings, see e.g.\ \cite{MV12}.
\end{proof}

We conclude this appendix by looking at anisotropic Besov spaces. For all $s\in \R$, $\q\in (1,\infty)^d$ and $p\in (1,\infty)$ we let
\begin{equation}
\label{eq:def_B_anisotropic}
B^{s}_{\q,p}(\T^d)\stackrel{{\rm def}}{=}(H^{s_0,\q}(\T^d),H^{s_1,\q}(\T^d))_{\theta,p}
\end{equation}
whenever $s_0,s_1\in\R$ satisfies $s=(1-\theta)s_0+\theta s_1$ for some $\theta\in (0,1)$. 

The independence of the above definition on the choice of $s_0,s_1$ follows from \cite[Theorem 4.7.2]{BeLo} and Lemma \ref{l:anisotropic_B}\eqref{it:anisotropic_B2}.
The space $B^{s}_{\q,p}(\T^d)$ can also be introduced by using Littlewood-Paley decompositions. We omit this as it will not be needed here. The following elementary embeddings follow from known interpolation properties:
\begin{align}
\label{eq:elementary_emb0}
B^{s}_{\q,p_0}(\T^d)\embed B^{s_1}_{\q,p_1}(\T^d) \ \ \text{ for all }&s\in \R,\  \q\in (1,\infty)^d\text{ and }p_1\geq  p_0,\\
\label{eq:elementary_emb1}
B^{s_0}_{\q,\infty}(\T^d)\embed B^{s_1}_{\q,1}(\T^d) \ \ \text{ for all }&s_0>s_1 \text{ and } \q\in (1,\infty)^d,\\
\label{eq:elementary_emb2}
B^{s}_{\q,1}(\T^d) \embed H^{s,\q}(\T^d) \embed B^{s}_{\q,\infty}(\T^d) \ \ \text{ for all }&s\in \R\text{ and } \q\in (1,\infty)^d. 
\end{align}
Indeed \eqref{eq:elementary_emb0} and \eqref{eq:elementary_emb1} follow from \cite[Theorem 3.4.1(b)]{BeLo} and \cite[Theorem 3.4.1(d)]{BeLo}, respectively. Finally, \eqref{eq:elementary_emb2} can be proven by combining Lemma \ref{l:anisotropic_B}\eqref{it:anisotropic_B2} and \cite[Theorem 4.7.1]{BeLo} (see also \cite[Theorem C.4.1]{Analysis1} for a refinement).

We conclude this appendix with the following consequence of Lemma \ref{l:anisotropic_B}.

\begin{corollary}[Sobolev embedding for anisotropic Besov spaces]
\label{cor:sob_emb_besov}
For all $s_0\geq s_1$, $\q_0,\q_1\in (1,\infty)^d$ and $p_0,p_1\in (1,\infty)$ satisfying $q_{0,i}\geq q_{1,i}$ for all $i\in \{1,\dots,d\}$, $p_1\geq p_0$ and $s_0-\sum_{i=1}^d \frac{1}{q_{0,i}}\geq s_1-\sum_{i=1}^d \frac{1}{q_{1,i}}$, the following embedding holds:
$$
B^{s_0}_{\q_0,p_0}(\T^d)\embed 
B^{s_1}_{\q_1,p_1}(\T^d).
$$
\end{corollary}

\begin{proof}
In the case $p_0=p_1$, the embedding follows from \eqref{eq:def_B_anisotropic}, interpolation and Lemma \ref{l:anisotropic_B}\eqref{it:anisotropic_B4}. The case $p_0>p_1$ follows from the previous case and \eqref{eq:elementary_emb0}.
\end{proof}

\section{Stochastic maximal $L^p$-regularity in anisotropic spaces}
\label{app:smr}
In this appendix, we prove stochastic maximal $L^p$-regularity estimates for the following linear SPDE in the anisotropic spaces:
\begin{equation}
\label{eq:app_problem}
\left\{
\begin{aligned}
&\partial_t u + \nabla \cdot(a\cdot \nabla u) = f + \sum_{n\geq 1} \big[(b_n\cdot\nabla) u+ g_n \big]\,\dot{\beta}^n_t,&\text{ on }&\T^d,\\
&u(0)=0 ,&\text{ on }&\T^d.
\end{aligned}\right.
\end{equation}
As in Appendix \ref{app:anisotropic}, we denote by $H^{s,\q}(\T^d)$ and $B^s_{\q,p}(\T^d)$ the anisotropic Bessel potential and Besov spaces, respectively. The space $H^{s,\q}(\T^d;\ell^2)$ of $\ell^2$-valued maps is defined via the requirement $\|f\|_{H^{s,\q}(\ell^2)}\stackrel{{\rm def}}{=}\|((1-\Delta)^{s/2}f_n)_{n\geq 1}\|_{L^{\q}(\T^d;\ell^2)}<\infty
$.
Finally $(\beta^n)_{n\geq 1}$ is a sequence of standard independent Brownian motions on a given filtered probability space.

To state the main result of this appendix we need the following conditions.

\begin{assumption}
\label{ass:stoch_max_reg_anisotropic}
Let $s\in \R$ and let the following be satisfied for all $n\geq 1$:
\begin{itemize}
\item $a=(a^{i,j})_{i,j=1}^d:\R_+\times \O\times \T^d\to \R^{d\times d}$ and $b_n=(b_n^j):\R_+\times \O\times \T^d\to \R^d$ be $\Progress\otimes \Borel(\R_+\times \T^d)$-measurable.
\item There exist $M>0$ and $\g>|1-s|$ such that for a.a.\ $(t,\om)\in\R_+\times\O$
$$
\|a(t,\om,\cdot)\|_{C^{\g}(\T^d;\R^{d\times d})}+
\|(b(t,\om,\cdot))_{n\geq 1}\|_{C^{\g}(\T^d;\ell^2)}\leq M.
$$
\item There exists $\ellip>0$ such that for all $\xi\in \R^d$ and a.e.\ on $\R_+\times\O\times \T^d$
$$
\sum_{i,j=1}^d \Big(a^{i,j}-\sum_{n\geq 1} b^j_n b^i_n\Big)\xi^i\xi^j \geq \ellip |\xi|^2.
$$ 
\end{itemize}
\end{assumption}

Strong solutions to \eqref{eq:app_problem} can be defined considering its integrated version. 
More precisely, a progressively measurable process $u\in L^p((0,T)\times \O,w_{\a};H^{2-s,\q})$ is a strong solution to \eqref{eq:app_problem} if a.s.\ and for all $t\in (0,T)$ 
$$
u(t)-\int_{0}^t \nabla (a\cdot\nabla u)\,\dd s =\int_{0}^t f\,\dd s+ \int_0^t ((b\cdot\nabla)u+g_n)_{n\geq 1}\,\dd \Br_{\ell^2}(s)
$$
where $\Br_{\ell^2}(s)$ is the $\ell^2$-cylindrical Brownian motion induced by $(\beta^n)_{n\geq 1}$, see \eqref{eq:def_Br}. As below Definition \ref{def:solution}, all the integrals appearing in the above equality are well-defined due to Lemma \ref{l:pointwise} below and the required regularity of $u$.

\begin{theorem}[Stochastic maximal $L^p$-regularity in anisotropic spaces]
\label{t:smr_anisotropic}
Let Assumption \ref{ass:stoch_max_reg_anisotropic} be satisfied. Let $T\in (0,\infty)$. Assume that $p\in (2,\infty)$, $\a\in [0,\frac{p}{2}-1)$, 
\begin{equation}
\label{eq:assumption_integrability_q}
q_i\in (2,\infty) \ \  \  \text{ for all }1\leq i\leq d-1 , \quad \text{and }\quad q_d\in [2,\infty).
\end{equation}
Then for all progressively measurable maps
$$
f\in L^p((0,T)\times \O,w_{\a};H^{-s,\q}(\T^d)), \  \text{ and }\ \ g\in L^p((0,T)\times \O,w_{\a};H^{1-s,\q}(\T^d;\ell^2)),
$$
there exists a unique strong solution $u\in L^p((0,T)\times \O,w_{\a};H^{2-s,\q}(\T^d))$ such that for all $\theta\in [0,\frac{1}{2})$ one has $u\in L^p(\O,H^{\theta,p}(0,T,w_{\a};H^{2-s-2\theta,\q}(\T^d)))$ and 
\begin{align}
\label{eq:max_reg_anisotropic_theta}
\E\|u\|_{H^{\theta,p}(0,T,w_{\a};H^{2-s-2\theta,\q}(\T^d))}^p 
&\lesssim_{\theta,p}
\E\|f\|_{L^p(0,T,w_{\a};H^{-s,\q}(\T^d))}^p\\
\nonumber
&+ \E\|g\|_{L^p(0,T,w_{\a};H^{1-s,\q}(\T^d;\ell^2))}^p,
\end{align}
where the implicit constant is independent of $(f,g)$.
\end{theorem}

As the proof below shows, Theorem \ref{t:smr_anisotropic} also holds with $(0,T)$ replaced by $(0,\tau)$ where $\tau$ is a stopping time taking values in a compact interval of $[0,\infty)$. 
By \cite[Proposition 3.10]{AV19_QSEE_1}, the above also implies stochastic maximal $L^p$-regularity estimates with non-trivial initial data $u(0)\in L^p_{\F_0}(\O;B^{2-s-2\frac{1+\a}{p}}_{\q,p})$.

In \eqref{eq:assumption_integrability_q} the case $q_d=2$ is allowed, but $q_i>2$ for all $i\leq d-1$. The optimality of such a condition is unclear.
This condition arises as a consequence of the application of Proposition \ref{prop:R_boundedness}. For the latter result \eqref{eq:assumption_integrability_q} seems optimal, cf.\ the counterexample \cite[Section 8]{NVW11} in a related situation. However, it is not known whether the conclusion of Proposition \ref{prop:R_boundedness} is necessary for stochastic maximal $L^p$-regularity to hold (cf.\ condition $(H_p)$ in \cite[Section 7]{NVW13}). 
As discussed in Subsection \ref{ss:scaling_intro} the possibility of choosing $q_d=2$ is of fundamental importance for the applications of stochastic maximal $L^p$-regularity in the context of PEs with rough noise. 

Theorem \ref{t:smr_anisotropic} is an extension of \cite[Theorem 5.2]{AV_torus} to the anisotropic setting.  
The proof of Theorem \ref{t:smr_anisotropic} is an adaptation of the one of \cite[Theorem 5.2]{AV_torus}.
The main tools needed to adapt the proofs of \cite{AV_torus} in an anisotropic setting are as follows:

\begin{roadmap}\
\begin{enumerate}[{\rm(1)}]
\item\label{it:smr_time} Stochastic maximal $L^p$-regularity in $H^{-s,\q}(\T^d)$ for the stochastic heat equation, i.e.\ $a^{i,j}=\delta^{i,j}$ and $b_n^j\equiv 0$.
\item\label{it:smr_without_time} Stochastic maximal $L^p$-regularity \emph{without} time regularity, i.e.\ proof of the estimate \eqref{eq:max_reg_anisotropic_theta} with $\theta=0$, in case of constant coefficients $(a,b)$.
\item\label{it:pointwise_multiplier} Pointwise multiplication in anisotropic Bessel potential spaces.
\end{enumerate}
\end{roadmap}

The precise statement, formulation and proofs of \eqref{it:smr_time}--\eqref{it:pointwise_multiplier} are given in  Subsections \ref{ss:laplace_anisotropic}--\ref{ss:pointwise_multiplication}. The proof of Theorem \ref{t:smr_anisotropic} is given in Subsection \ref{ss:smr_anisotropic_proof}.

\subsection{The stochastic heat equation in anisotropic spaces}
\label{ss:laplace_anisotropic}
The aim of this subsection is to prove stochastic maximal $L^p$-regularity for the stochastic heat equation in anisotropic spaces.

\begin{lemma}
\label{l:smr_heat}
Theorem \ref{t:smr_anisotropic} holds in case $a^{i,j}=\delta^{i,j}$ and $b_n^j=0$.
\end{lemma}

We prove the above by using the semigroup approach due to J.\ van Neerven, M.C.\ Veraar, L.\ Weis \cite{MaximalLpregularity}. Here we argue as in \cite[Section 7]{NVW13} which generalizes \cite{MaximalLpregularity}. Note that the results \cite[Section 7]{NVW13} covers only the case $\a=0$, while the case $\a>0$ follows from the latter by extrapolation (see either \cite[Section 7]{AV19} or \cite{LV21}).

To obtain stochastic maximal $L^p$-regularity from \cite[Theorem 7.1]{NVW13} and its variant for time regularity (cf.\ also the proof of \cite[Theorem 7.16]{AV19}), we need to check assumption $(H_p)$ in \cite[Theorem 7.1]{NVW13}. In particular, one is interested in the $\Rsec$-boundedness of certain families of operators of stochastic convolution type (for the notion of $\Rsec$-boundedness see \cite[Chapter 8]{Analysis2}). A throughout investigation of the latter issue is provided in \cite{NVW11}. In particular, we employ the characterization of such $\Rsec$-boundedness given in \cite[Theorem 7.1]{NVW11}. 

Let $X$ be a Banach space with type 2 (see e.g.\ \cite[Chapter 7]{Analysis2} for the definition), and let $\eta\in [0,1/2)$ be fixed. 
For each $\lambda \in \C$ with $\Re (\lambda)>0$, let $k_{\lambda}(t)\stackrel{{\rm def}}{=} \lambda^{1/2-\eta} t^{-\eta}e^{-\lambda t}$ for $t>0$. Similar to \cite[Section 9]{LV21}, for all $p\in [2,\infty)$, we define the following bounded operator: 
\begin{align*}
N_{\lambda}&:L^p(\R_+;X)\to L^p(\R_+;\g(L^2(\R_+),X))\\  
N_{\lambda }f(s)&\stackrel{{\rm def}}{=}
k_{\lambda}(s-\cdot)f \ \text{ for } \  s>0.
\end{align*}
Compared to \cite{LV21}, the additional parameter $\eta$ is needed to obtain sharp space-time estimates in the semigroup case, see e.g.\ \cite[Theorem 1.2]{MaximalLpregularity} or \cite[Subsection 7.3]{AV19}. However, the results in \cite[Section 9]{LV21} readily extend to this case.

For $\varphi\in [0,\pi/2)$ and $p\in [2,\infty)$, we set 
\begin{align*}
\NN_{\varphi}^{(p)}(X)&\stackrel{{\rm def}}{=}\{ N_{\lambda }\,:\, |\arg(\lambda)|\leq \varphi\}
\subseteq \mathscr{L}(L^p(\R_+;X), L^p(\R_+;\g(L^2(\R_+),X)))
.
\end{align*}
The following result can be used to obtain the $\Rsec$-boundedness of $\NN_{\varphi}^{(p)}(X)$ when $X$ is an anisotropic Lebesgue space.

\begin{proposition}
\label{prop:R_boundedness}
Let $X$ be a Banach space with type $2$. Assume that $\NN^{(r)}_{\varphi}(X)$ is $\Rsec$-bounded for some $r\in [2,\infty)$ and $\varphi\in [0,\pi/2)$. Then, for any $\sigma$-finite measure space $(S,\mathcal{A},\mu)$ and any $q,p\in (2,\infty)$,
$$
\NN_{\varphi}^{(p)}(L^q(S;X)) \text{ is $\Rsec$-bounded.}
$$
\end{proposition}

Recall that, by Theorem \cite[Theorem 7.1]{NVW11}, $\NN_{\varphi}^{(p)}(X)$ is $\Rsec$-bounded for all $p\in (2,\infty)$ if $X$ is isomorphic to a closed subspace of a space $L^q(S)$ where $q\in [2,\infty)$ (the case $q=p=2$ is also true). Hence, the above result shows that ($H_p$) in \cite[Section 7]{NVW13} holds for all $p\in (2,\infty)$ and all Banach space $X$ isomorphic to a closed subspace of $L^{q_1}(S_1;L^{q_2}( S_2))$ whenever $q_1\in (2,\infty)$, $q_2\in [2,\infty)$ and $(S_1,\mathcal{A}_1,\mu_1)$, $(S_2,\mathcal{A}_2,\mu_2)$ are $\sigma$-finite measure spaces. The previous argument can be further iterated by using $L^{\xi}$-spaces with $\xi\in (2,\infty)$. In particular $\NN_{\varphi}^{(p)}(L^{\q}(\T^d))$ is $\Rsec$-bounded if \eqref{eq:assumption_integrability_q} holds.

\begin{proof} Fix $\varphi\in [0,\pi/2)$.
Arguing as in \cite[Theorem 9.1]{LV21}, we have that $\NN^{(p)}_{\varphi}(X)$ is $\Rsec$-bounded and that the $\Rsec$-boundedeness of $\NN_{\varphi}^{(p)}(L^q(S;X))$ is independent of $p\in (2,\infty)$. In particular, it is enough to consider the case $q=p$. 
Since $(S,\mu,\mathcal{A})$ is $\sigma$-finite, the $\g$-Fubini isomorphism \cite[Theorem 9.4.8]{Analysis2} yields
\begin{equation}
\label{eq:gamma_fubini_isomorphism}
\g(L^2(\R_+);L^q(S;X))=L^q(S;\g(L^2(\R_+),X)) \ \text{ isometrically}. 
\end{equation} 

Let $(\varepsilon_j)_{j\geq 1}$ be a sequence of independent Rademacher variable over a probability space $(\wt{\O},\wt{\mathcal{A}},\wt{\P})$, i.e.\ $\wt{\P}(\varepsilon_j=1)=\wt{\P}(\varepsilon_j=-1)=1/2$. Let $\wt{\E}[\cdot]\stackrel{{\rm def}}{=}\int_{\wt{\O}}\cdot\,\dd \wt{\P}$ be the corresponding expected value.
By the Kahane-Khintchine inequality (see e.g.\ \cite[Theorem 6.2.4]{Analysis2}), for all integer $J\geq 1$, $(\lambda_j)_{j=1}^J\subset \C$ such that $|\arg(\lambda_j)|\leq \varphi$ and $(f_i)_{i=1}^J\subset L^p(S;X)$,  
\begin{align*}
\wt{\E}\Big\| \sum_{j=1}^J \varepsilon_j N_{\lambda_j} f_j \Big\|_{L^q(\R_+;\g(L^2(\R_+),L^q(S;X)))}^q
\stackrel{\eqref{eq:gamma_fubini_isomorphism}}{=} 
\wt{\E}\Big\| \sum_{j=1}^J \varepsilon_j N_{\lambda_j} f_j \Big\|_{L^q(S;L^q(\R_+;\g(L^2(\R_+),X)))}^q&\\
=\int_{S} \wt{\E}
\Big\| \sum_{j=1}^J \varepsilon_j N_{\lambda_j} f_j(s) \Big\|_{L^q(\R_+;\g(L^2(\R_+),X))}^q \,\dd \mu(s)&\\
\stackrel{(i)}{\leq} \Rsec (\NN^{(q)}_{\varphi}(X)) 
\int_{S} \wt{\E}
\Big\| \sum_{j=1}^J \varepsilon_j  f_j(s) \Big\|_{L^q(\R_+;X)}^q \,\dd \mu(s)&\\
= \Rsec (\NN^{(q)}_{\varphi}(X)) 
 \wt{\E}
\Big\| \sum_{j=1}^J \varepsilon_j  f_j(s) \Big\|_{L^q(\R_+;L^q(S;X))}^q&,
\end{align*}
where in $(i)$ we use the $\Rsec$-boundedness of $\NN^{(q)}_{\varphi}(X)$ pointwise in $s\in S $.
The above estimate proves the $\Rsec$-boundedness of $\NN_{\varphi}^{(q)}(L^q(S;X))$.
\end{proof}

\begin{proof}[Proof of Lemma \ref{l:smr_heat}]
Firstly, let us note that from the Miklin multiplier theorem in $L^{\q}$ and the periodic version of \cite[Theorem 10.2.25]{Analysis2}, it follows that the operator 
\begin{equation}
\label{eq:def_A0}
\begin{aligned}
A_0&: H^{2-s,\q}(\T^d)\subseteq H^{-s,\q}(\T^d)\to H^{-s,\q}(\T^d), \\  
A_0f & \stackrel{{\rm def}}{=}-\Delta f, \qquad f\in H^{2-s,\q}(\T^d),
\end{aligned}
\end{equation}
has a bounded $H^{\infty}$-calculus with angle $0$ (for the notion of $H^{\infty}$-calculus, see either \cite[Chapter 10]{Analysis2} or \cite[Chapter 3, Subsection 3.5]{pruss2016moving}).
Now, as noticed below Proposition \ref{prop:R_boundedness}, due to \eqref{eq:assumption_integrability_q}, the space $L^{\q}(\T^d)$ satisfies the condition 
$(H_p)$ in \cite[Saction 7]{NVW13}. Hence, in virtue of the boundedness of the $H^{\infty}$-calculus of $A$, the $\a=0$-case follows as in the proof of \cite[Theorem 7.1]{NVW13}, where for the case of space-time regularity one can argue as in the proof of \cite[Theorem 1.2]{MaximalLpregularity} (see also \cite[Theorem 7.16]{AV19}). The case $\a>0$, now follows from the unweighted case and an extrapolation argument, see either \cite[Section 7]{AV19} or \cite{LV21}.
\end{proof}

\subsection{Stochastic maximal $L^p$-regularity with constant coefficients}
Here we prove 
Theorem \ref{t:smr_anisotropic} in the special case of $x$-independent coefficients. 

\begin{lemma}
\label{l:smr_x_independent}
Theorem \ref{t:smr_anisotropic} holds with $\theta=0$ in \eqref{eq:max_reg_anisotropic_theta} in case $a^{i,j},b^j_n$ are independent of $x\in \T^d$, i.e.\ $a^{i,j}, b^j_n$ depend only on $(t,\om)\in \R_+\times\O$.
\end{lemma}

\begin{proof}Firstly, as $(1-\Delta)^{t/2}:H^{r,\q}(\T^d)\to H^{r-t,\q}(\T^d)$ is an isomorphism for all $r,t\in\R$ and $a^{i,j}, b^{j}_n$ are $x$-independent, it is enough to consider $s=0$.

Next, to prove the result with $s=0$ we use two `transference' results proven in \cite{VP18}. 
More precisely, by \cite[Theorem 3.18]{VP18} and arguing as in the proof of \cite[Theorem 5.3]{VP18} the $x$-independence assumption on $a^{i,j},b^j_n$ and the invariance under translation of $L^{\q}(\T^d)$, allow us to reduce the proof of Lemma \ref{l:smr_heat} to the case $b^j_n\equiv 0$. The claim of Lemma \ref{l:smr_x_independent} with $s=0$ now follows from \cite[Theorem 3.9]{VP18} where one choses $A_0$ as \eqref{eq:def_A0}. It remains to check the assumptions of \cite[Theorem 3.9]{VP18}. The assumption $(ii)$ of \cite[Theorem 3.9]{VP18} is satisfied due to Lemma \ref{l:smr_heat}. Finally, assumption $(i)$ in  \cite[Theorem 3.9]{VP18} follows from \cite[Theorem 5.2]{DK_18_Ap} and the Rubio de Francia extrapolation type result \cite[Theorem 6.2]{K21_rubio} (see also \cite{CUMP11}).  
\end{proof}

\subsection{Pointwise multiplication in anisotropic spaces}
\label{ss:pointwise_multiplication}
We begin by extending the pointwise multiplication result of \cite[Proposition 4.1(2) and (4)]{AV_torus} to the anisotropic setting.

\begin{lemma}
\label{l:pointwise}
Let $s\geq 0$, $\q\in (1,\infty)^d$, $\g>s$ and $H$ be an Hilbert space. Then
\begin{enumerate}[{\rm(1)}]
\item\label{it:pointwise1}
$
\displaystyle{\|f g\|_{H^{s,\q}(\T^d;H)}\lesssim\|f\|_{H^{s,\q}(\T^d)}\|g\|_{L^{\infty}(\T^d;H)}+\|f\|_{L^{\q}(\T^d)}\|g\|_{C^{\g}(\T^d;H)};}
$
\item\label{it:pointwise2} 
$
\displaystyle{\|f g\|_{H^{-s,\q}(\T^d;H)}\lesssim\|f\|_{H^{-s,\q}(\T^d)}\|g\|_{L^{\infty}(\T^d;H)}+\|f\|_{H^{-s-\varepsilon,\q}(\T^d)}\|g\|_{C^{\g}(\T^d;H)}}
$
where $\varepsilon\in (0,\g-s)$ is arbitrary.
\end{enumerate}
In the above the implicit constants are independent of $(f,g)$.
\end{lemma}

\begin{proof}
For simplicity, we prove the result for $H=\R$, the general case is analog.
By a standard localization argument, it is enough to show the claimed estimates with $\T^d$ replaced by $\R^d$.

\eqref{it:pointwise1}: 
As in the proof of \cite[Proposition 4.1(2)]{AV_torus}, we now follow \cite[Chapter 2, Section 1]{ToolsPDEs}. 
Let $(\psi_j)_{j\geq 0}$ be a Littlewood-Paley partition of the unity, see \cite[pp.\ 4]{ToolsPDEs} and set $\Psi_j\stackrel{{\rm def}}{=}\sum_{ k\leq j} \psi_j$ for all $j\geq 0$. Next we write the product $fg$ by using Bony's paraproducts, i.e.\ $fg=T_f g + R(f,g)+ T_g f$, where 
$$
T_f g\stackrel{{\rm def}}{=} \sum_{k\geq 5} \Psi_{k-5}(D)f\, \psi_{k+1} (D)g , \quad \text{ and }\quad R(f,g)\stackrel{{\rm def}}{=}\sum_{|j-k|\leq 4} \psi_j (D)f \psi_k (D)g.
$$
As in \cite{ToolsPDEs}, for all smooth functions $\varphi$, we set $\varphi(D)f\stackrel{{\rm def}}{=}\mathcal{F}^{-1}(\varphi(\xi)\mathcal{F}(f))$ where $\mathcal{F}$ is the Fourier transform.

The proof now follows by combining the argument in the proof of \cite[Proposition 4.1(2)]{AV_torus} and the above-mentioned Rubio de Francia extrapolation. Indeed,  by \cite[Theorem 6.2]{K21_rubio}), one can readily check that to prove the estimate in Lemma \ref{l:pointwise}, it is enough to show that for all $r\in (1,\infty)$ and $w\in A_r$
\begin{align}
\label{eq:paraproduct_1}
\|T_g f\|_{H^{s,r}(\R^d,w)}+\|R(f,g)\|_{H^{s,r}(\R^d,w)}
&\leq C_r([w]_{A_r})\|f\|_{H^{s,r}(\R^d,w)}\|g\|_{L^{\infty}(\R^d)},\\
\label{eq:paraproduct_2}
\|T_f g \|_{H^{s,r}(\R^d,w)}
&\leq C_r([w]_{A_r}) \|g\|_{C^{\g}(\R^d)}\|f\|_{L^{r}(\R^d,w)},
\end{align}
whenever the RHS of the above estimates make sense, and
where $C_r:[0,\infty)\to [0,\infty)$ is a locally bounded map. Finally, $[w]_{A_r}$ is the Muckenhoupt $A_r$ characteristic of the weight $w$:
$$
[w]_{A_r}\stackrel{{\rm def}}{=}\sup_{Q} \Big\{\Big(\fint_{Q}w \,\dd x\Big)\Big(\fint_{Q} w^{1-r'}\,\dd x\Big)^{r-1}\Big\}
$$ 
where the supremum is over all axes-parallel cubes.

To conclude the proof, it remains to discuss the validity of \eqref{eq:paraproduct_1}--\eqref{eq:paraproduct_2}. Firstly,  recall that the Littlewood-Paley decomposition of $H^{s,r}$ still holds in the weighted setting, as it is a consequence of the weighted Mihlin multiplier theorem \cite[Theorem 7.1]{L21_JGA}, cf.\ \cite[Chapter 13, eq.\ (5.37)--(5.46)]{TayPDE3}. Hence the proof of \eqref{eq:paraproduct_1} follows as in  \cite[Chapter 2, Section 1]{ToolsPDEs} where one uses that $w\in A_r$ implies that the maximal function is bounded on $L^r(\R^d,w)$ (see e.g.\ \cite[Theorem J.1.1]{Analysis2}). The latter also shows that the estimate \eqref{eq:paraproduct_2} can be proven by a straightforward modification of the inequalities in \cite[eq.\ (4.3)]{AV_torus}.

\eqref{it:pointwise2}: As in \cite[Proposition 4.1(2)]{AV_torus}, the estimate in \eqref{it:pointwise2} follows by the duality and the decomposition involving Bony's paraproducts. 
\end{proof}

\subsection{Proof of Theorem \ref{t:smr_anisotropic}}
\label{ss:smr_anisotropic_proof}
In virtue of Lemmas \ref{l:smr_heat}--\ref{l:pointwise}, the argument in \cite[Theorem 5.2]{AV_torus} readily extends to anisotropic spaces. For brevity, we only include a short sketch where we indicate the needed modifications.

\begin{proof}[Proof of Theorem \ref{t:smr_anisotropic}  -- Sketch]
Here we indicate how Lemmas \ref{l:smr_heat} and \ref{l:smr_x_independent}-\ref{l:pointwise} can be used to modify the proof of \cite[Theorem 5.2]{AV_torus} and therefore leading to an extension of \cite[Theorem 5.2]{AV_torus} to the case of anisotropic spaces. Lemma \ref{l:smr_heat} is needed to use \cite[Proposition 3.1]{AV_torus} with $(\wh{A},\wh{B})=(-\Delta,0)$ with $X_j=H^{2j-s,\q}(\T^d)$ for $j\in \{0,1\}$. Lemmas \ref{l:smr_x_independent} and \ref{l:pointwise} are needed in the proof of Steps 1 and 3 of \cite[Lemma 5.4]{AV_torus} respectively. The remaining arguments used in the proof of \cite[Theorem 5.2]{AV_torus} remain unchanged.
\end{proof}


\bibliographystyle{plain}
\bibliography{literature}

\end{document}